\documentclass[11pt]{amsart}
\usepackage[utf8]{inputenc}
\usepackage[T1]{fontenc}
\usepackage{aeguill}
\usepackage{xspace} 
\usepackage{mathrsfs} 
\usepackage{amsfonts}
\usepackage{url}
\usepackage{relsize}
\usepackage{accents}
\usepackage{bbm}
\usepackage[makeroom]{cancel}
\usepackage{amsmath,amssymb,amsthm,calligra,mathrsfs}
\usepackage{epsfig}
\usepackage[matrix,arrow]{xy}

\usepackage{mathbbol}
\usepackage{longtable}
\usepackage{stmaryrd}
\usepackage[matrix,arrow]{xy}
\usepackage{multicol}
\usepackage{array}
\usepackage{multirow}
\usepackage{tikz-cd}
\usepackage{rank-2-roots}
\usepackage{fancybox} 
\usepackage{hyperref}
\definecolor{puces}{cmyk}{0,0,0,.5} 

\theoremstyle{plain}

\newtheorem{thm}{Theorem}[section]
\newtheorem{lem}[thm]{Lemma}
\newtheorem{coro}[thm]{Corollary}
\newtheorem{prop}[thm]{Proposition}

\newtheorem{conj}{Conjecture}[section]
\theoremstyle{remark} 

\usepackage{xcolor}
\newtheorem{rem}[thm]{Remark}

\usepackage{xcolor}

\DeclareMathOperator{\Hom}{\mathscr{H}\text{\kern -3pt {\calligra\large om}}\,}

\usetikzlibrary{shapes.geometric}
\numberwithin{equation}{section}
    \author{Emilien Zabeth}
\address{Université Clermont Auvergne, CNRS, LMBP\\
F-63000 Clermont-Ferrand, France}
\email{emilien.zabeth@uca.fr}
\begin{document}
	\title[Perverse sheaves on the semi-infinite flag variety]{Perverse sheaves on the semi-infinite flag variety and representations of the Frobenius kernel}

	\maketitle
     \begin{abstract}
 We study a category of Iwahori-equivariant modular perverse sheaves on some avatar of the semi-infinite flag variety, by adapting the work of Arkhipov-Bezrukavnikov-Braverman-Gaitsgory-Mirkovi\'c. We then construct a functor between the latter category and the category of graded representations of the Frobenius kernel, and conjecture that this functor is an equivalence. 
 \end{abstract}
 \setcounter{tocdepth}{1}
	\tableofcontents
  \section{Introduction}
Let $\mathbb{k}$ be a finite field of characteristic $\ell$. In this paper, we study a category of modular (with coefficients in $\mathbb{k}$) perverse sheaves $\mathrm{Perv}_{I_\mathrm{u}}(\mathcal{F}l^{\frac{\infty}{2}})$ on the semi-infinite flag variety associated with a reductive algebraic group $G$ defined over an algebraically closed field $\mathbb{F}$ of prime characteristic $p\neq\ell$. Under certain conditions on $\ell$ and $G$, we construct an exact functor between the latter category and the extended principal block $\mathrm{Rep}_{[0]}(\check{\mathbf{G}}_1\check{\mathbf{T}})$ of the category of $\check{\mathbf{G}}_1\check{\mathbf{T}}$-modules, where $\check{\mathbf{G}}_1$ is the Frobenius kernel of $\check{\mathbf{G}}$, the split reductive group over $\mathbb{k}$ whose Frobenius twist coincides with the Langlands dual group of $G$ over $\mathbb{k}$, and $\check{\mathbf{T}}\subset \check{\mathbf{G}}$ is a maximal torus. This functor is expected to be an equivalence, and we give a condition under which it is. Here, the semi-infinite flag variety $\mathcal{F}l^{\frac{\infty}{2}}$ is constructed using the Drinfeld compactification space $\overline{\mathrm{Bun}}_{N^-}$, and our proof follows closely the steps taken in \cite{ABBGM}, where the authors work with coefficients of characteristic zero. Namely, we construct a functor between  $\mathrm{Perv}_{I_\mathrm{u}}(\mathcal{F}l^{\frac{\infty}{2}})$ and the category of modules over the ``regular perverse sheaf'' -- which is known to be equivalent to $\mathrm{Rep}_{[0]}(\check{\mathbf{G}}_1\check{\mathbf{T}})$ thanks to the recent work of Achar-Riche \cite{achar2022geometric}. The key result in the construction is the fact that the IC-sheaf on the stack $\overline{\mathrm{Bun}}_{N^-}$ satisfies the Hecke eigensheaf property with respect to $\mathrm{Rep}(\check{\mathbf{G}}^{(1)})$ (see Proposition \ref{prop hecke geometric}). This result was known for characteristic zero coefficients and used in the context of the geometric Langlands program (see \cite{BBGM}), but had not been observed yet (to the knowledge of the author) in the case of positive characteristic.

\subsection{Representations of the Frobenius kernel}
We fix a Borel subgroup $B\subset G$ and a maximal torus $T\subset B$. Let $\mathrm{Fr}:\check{\mathbf{G}}\to \check{\mathbf{G}}^{(1)}$ denote the Frobenius endomorphism, and $\check{\mathbf{G}}_1$ be its kernel. If we let $L(\lambda)$ be the simple $\check{\mathbf{G}}$-module
associated with a dominant character $\lambda$ of $\check{\mathbf{T}}$, then $L(\lambda)$ remains simple as a $\check{\mathbf{G}}_1$-module when $\lambda$
is $\ell$-restricted, and any simple $\check{\mathbf{G}}_1$-module is obtained as such a restriction. Thus (thanks to Steinberg's tensor product formula), the understanding of simple $\check{\mathbf{G}}_1$-modules is equivalent to that of simple $\check{\mathbf{G}}$-modules (under the assumption that the derived subgroup of $\check{\mathbf{G}}$ is simply connected).

For technical reasons, it is often more convenient to work with $\check{\mathbf{G}}_1\check{\mathbf{T}}$ -modules. Simple $\check{\mathbf{G}}_1\check{\mathbf{T}}$-modules
$\{\widehat{L}(\lambda),~\lambda\in X^*(\check{\mathbf{T}})\}$ are labeled by the lattice of characters $X^*(\check{\mathbf{T}})$, and remain simple when viewed as $\check{\mathbf{G}}_1$-modules. For each $\lambda\in X^*(\check{\mathbf{T}})$, the object $\widehat{L}(\lambda)$ is defined as the socle of the baby co-Verma module $\widehat{Z}'(\lambda)$, and coincides with the head of the baby Verma module $\widehat{Z}(\lambda)$.

A central question in the theory of modular representations of algebraic groups is to understand how baby Verma modules decompose in the basis of simple modules in the Grothendieck group of $\mathrm{Rep}(\check{\mathbf{G}}_1\check{\mathbf{T}})$, as one would then get character formulas for simple modules. Recent progress have been obtained by Riche-Williamson in \cite{riche2021simple}, where the authors obtain a formula -- proved in the case where $\ell\geq 2h-1$, where $h$ is the Coxeter number of $\check{\mathbf{G}}$ -- involving the $\ell$-canonical basis of the affine Hecke algebra. Their proof does not rely on a geometric model for the category $\mathrm{Rep}(\check{\mathbf{G}}_1\check{\mathbf{T}})$, but rather on the geometric Satake equivalence, which gives a geometric model for the category of $\check{\mathbf{G}}$-modules. The goal of this paper is to obtain a geometric model for the category of $\check{\mathbf{G}}_1\check{\mathbf{T}}$-modules (or rather the principal block of this category), with the hope that one might be able to use new techniques afforded by the geometry for the understanding of the characters of simple modules.

We will in fact restrict our attention to the extended principal block $\mathrm{Rep}_{[0]}(\check{\mathbf{G}}_1\check{\mathbf{T}})$ of the category of $\check{\mathbf{G}}_1\check{\mathbf{T}}$-modules, whose
understanding is sufficient for the study of the general case (we assume for this that $\ell > h$). This is the
Serre subcategory of $\mathrm{Rep}(\check{\mathbf{G}}_1\check{\mathbf{T}})$ generated by the family
$$\{\widehat{L}(w\bullet_\ell0)~|~w\in W_{\mathrm{ext}}\},$$
where $\bullet_\ell$ denotes the ``dot" action of the extended affine Weyl group $W_{\mathrm{ext}}:=W\ltimes X^*(\check{\mathbf{T}})$ on the affine space $X^*(\check{\mathbf{T}})\otimes_\mathbb{Z}\mathbb{R}$.
\begin{rem}
    In view of \cite{riche2022smith}, it would be desirable to obtain a geometric model for the whole category $\mathrm{Rep}(\check{\mathbf{G}}_1\check{\mathbf{T}})$, instead of merely the principal block. Indeed, assume that there exists a space $X$ (for instance an ind-algebraic stack over $\mathbb{F}$), endowed with an action of the multiplicative group $\mathbb{G}_\mathrm{m}$ and an equivalence of categories \begin{equation}\label{eq conjecture}
        \mathrm{Perv}(X)\xrightarrow{\sim}\mathrm{Rep}(\check{\mathbf{G}}_1\check{\mathbf{T}}),
    \end{equation}(where the category $\mathrm{Perv}(X)$ might involve an equivariance condition or ``factorization" property, cf. §\ref{section Main definition}). Then, following the philosophy of \textit{loc. cit.}, the characters of projective indecomposable $\check{\mathbf{G}}_1\check{\mathbf{T}}$-modules (which are the analogues of indecomposable tilting $\check{\mathbf{G}}$-modules) might be expressed in terms of the stalks of indecomposable parity sheaves on the connected components of the fixed points $X^{\mu_\ell}$, where $\mu_\ell\subset\mathbb{G}_\mathrm{m}$ is the subgroup of $\ell$-th roots of unity. To the knowledge of the author, there is no conjecture in the literature concerning a space $X$ realizing the equivalence \eqref{eq conjecture}.
\end{rem}
 \subsection{The geometric Satake equivalence and Finkelberg-Mirkovi\'c conjecture} The present work builds on several other celebrated equivalences of categories. First, we have an equivalence of monoidal categories (whose first complete proof was obtained in \cite{Mirkovic2004GeometricLD})
 $$\mathrm{Sat}:(\mathrm{Perv}_{G[[t]]}(\mathrm{Gr}),\star)\xrightarrow{\sim} (\mathrm{Rep}_\mathbb{k}(\check{\mathbf{G}}^{(1)}),\otimes_\mathbb{k}),$$
 where the left-hand side denotes the \textit{Satake category} (equipped with a convolution product), consisting of perverse (\'etale) sheaves on the affine Grassmannian $\mathrm{Gr}$, with coefficients in $\mathbb{k}$ and which respect an equivariance condition for the left action of the positive loop group $G[[t]]$ on $\mathrm{Gr}$. This result is called the \textit{geometric Satake equivalence}, and is ubiquitous in this paper.
 
 Next, denote by $\mathrm{Rep}_{[0]}(\check{\mathbf{G}})$ the extended principal block of the category $\mathrm{Rep}(\check{\mathbf{G}})$, defined as the Serre subcategory generated by the family $\{L(\lambda),~\lambda\in (W_{\mathrm{ext}}\bullet_\ell 0)\cap\mathbf{Y}^+\}$. We let $\mathrm{Perv}_{I_\mathrm{u}}(\mathrm{Gr})$ denote the category of perverse sheaves on $\mathrm{Gr}$ with coefficients in $\mathbb{k}$, equivariant with respect to the Iwahori subgroup $I_\mathrm{u}$ associated with $B$. The latter category admits a right action of the monoidal category  $(\mathrm{Perv}_{G[[t]]}(\mathrm{Gr}),\star)$. Below, the functor $\mathrm{sw}^*$ denotes a certain auto-equivalence of the category $\mathrm{Perv}_{G[[t]]}(\mathrm{Gr})$ (see §\ref{section The affine Grassmannian and the affine flag variety} for the definition), and $W^S_\mathrm{ext}$ a certain subset of $W_\mathrm{ext}$ labelling the simple objects $\{\mathrm{L}_w,~w\in W^S_\mathrm{ext}\}$ of $\mathrm{Perv}_{I_\mathrm{u}}(\mathrm{Gr})$.  The following statement is known as the Finkelberg-Mirkovi\'c conjecture (cf. \cite[§1.5]{finkelberg1997semiinfinite}). It was recently proved by Bezrukavnikov-Riche (it is still conjectured to hold when $\ell>h$).  
	\begin{thm}\label{conj FM}
		Assume that $\ell> h$, and that $\ell\neq19$ (resp. $\ell\neq31$) if $\check{\mathbf{G}}$ has a component of type $E_7$ (resp. $E_8$). There exists an equivalence of highest weight categories
		$$\mathrm{FM}:\mathrm{Perv}_{I_\mathrm{u}}(\mathrm{Gr})\xrightarrow{\sim} \mathrm{Rep}_{[0]}(\check{\mathbf{G}}) $$
		which satisfies
		$$\mathrm{FM}(\mathrm{L}_w)\simeq L(w^{-1}\bullet_\ell0)~~\forall  w\in W^S_\mathrm{ext}.$$
		Moreover, for any $\mathcal{F}\in \mathrm{Perv}_{I_\mathrm{u}}(\mathrm{Gr}),~\mathcal{G}\in \mathrm{Perv}_{G[[t]]}(\mathrm{Gr})$, there exists a bifunctorial isomorphism
		$$\mathrm{FM}(\mathcal{F}\star \mathcal{G})\simeq \mathrm{FM}(\mathcal{F})\otimes \mathrm{Fr}^*(\mathrm{Sat}(\mathrm{sw}^*\mathcal{G})).$$
	\end{thm}

\subsection{Results of Achar-Riche}
In their recent work \cite{achar2022geometric}, Achar and Riche give an incarnation of $\mathrm{Rep}_{[0]}(\check{\mathbf{G}}_1\check{\mathbf{T}})$ in terms of perverse sheaves. More precisely, they define a category $\mathrm{mod}^\mathbf{Y}_{I_\mathrm{u}}(\mathcal{R})$ consisting of $\mathbf{Y}$-graded ind-objects of the category $\mathrm{Perv}_{I_\mathrm{u}}(\mathrm{Gr})$, endowed with a right-action of the ``regular perverse sheaf" $\mathcal{R}$. We refer to §\ref{section Perverse sheaves on affine Grassmannians} for a detailed review of this construction. In particular, they prove that the category $\mathrm{mod}^\mathbf{Y}_{I_\mathrm{u}}(\mathcal{R})$ is equivalent to $\mathrm{Rep}_{[0]}(\check{\mathbf{G}}_1\check{\mathbf{T}})$ under the assumptions of Theorem \ref{conj FM} (cf. Theorem \ref{thm BR}). This result will be important for us, since it allows the theory of perverse sheaves to enter the picture.
  
 \subsection{Main statement} The main part of this paper is dedicated to the definition and study of an ``artificial" category of perverse sheaves on the semi-infinite flag variety $\mathcal{F}l^{\frac{\infty}{2}}$, equivariant with respect to the unipotent radical of the Iwahori subgroup $I_\mathrm{u}$. The idea of considering perverse sheaves on $\mathcal{F}l^{\frac{\infty}{2}}$ goes back to Lusztig in the 80's (see \cite[§11]{lusztigintersection}). It was then defined as the quotient 
 $$G(\mathbb{F}((t)))/T\cdot N(\mathbb{F}((t))),$$
 where $t$ is a formal variable. The above set can be endowed with a structure of ind-scheme (see \cite[§4]{finkelberg1997semiinfinite}), on which the group $I_\mathrm{u}$ acts. The orbits are labelled by the affine Weyl group and their closure relations are given by Lusztig's periodic order (these are the first requirements for the category of $I_\mathrm{u}$-equivariant perverse sheaves on it to be equivalent to $\mathrm{Rep}_{[0]}(\check{\mathbf{G}}_1\check{\mathbf{T}})$); the problem is that the closures of these orbits are schemes of infinite type (so, in particular, are not varieties). The usual methods are therefore not suited to define $I_\mathrm{u}$-equivariant perverse sheaves on such a space. 
 
 To overcome this difficulty, a general idea is to replace a local problem by a global problem: namely, replacing the formal disc $\mathrm{Spec}(\mathbb{F}[[t]])$ by a smooth complete curve $X$ (for technical simplicity, we will take $X$ to be $\mathbb{P}^1$ in the main body text). The formal disc is seen as the formal neighborhood of a fixed point $x$. This point of view led the authors of \cite{finkelberg1997semiinfinite} and \cite{feigin2semi} to consider perverse sheaves on spaces of Drinfeld compactifications (or more precisely on quasimap spaces), and was taken up in \cite{ABBGM} to establish an equivalence of categories between $\mathrm{Perv}_{I_\mathrm{u}}(\mathcal{F}l^{\frac{\infty}{2}})$ and the extended principal block of the category of graded modules for the small quantum group (associated with $\check{\mathbf{G}}$) at an $\ell$-th root of unity -- the latter category being the ``quantum analogue'' of $\mathrm{Rep}_{[0]}(\check{\mathbf{G}}_1\check{\mathbf{T}})$. In these previous references, the authors worked with sheaves whose field of coefficients is of characteristic zero. We check that the constructions of \cite{ABBGM} still make sense for a field of positive characteristic. Let us give a brief overview of these constructions\footnote{To simplify exposition, we assume in this part that $G$ is simply connected.} .
 
 We denote by $\mathbf{X}$ (resp. $\mathbf{Y}$) the lattice of characters (resp. cocharacters) of $T$, and by $\mathbf{X}^+\subset\mathbf{X}$ the subset of dominant characters determined by $B$. Let also $B^-$ be the Borel subgroup wich is opposite to $B$ with respect to $T$, and $N$ (resp. $N^-$) be the unipotent radical of $B$ (resp. of $N^-$). Fix a smooth complete curve $X$ over $\mathbb{F}$, a closed point $x\in X$, and let $\mathrm{Bun}_G$ (resp. $\mathrm{Bun}_{N^-}$) denote the stack of principal $G$-bundles (resp. principal $N^-$-bundles) over $X$. The space $\overline{\mathrm{Bun}}_{N^-}$ is an algebraic stack locally of finite type over $\mathbb{F}$ which is a ``compactification" of the canonical projection $\mathrm{Bun}_{N^-}\to \mathrm{Bun}_{G}$, whose construction is recalled in §\ref{section Drinfeld's compactification and variants}. For any test scheme $S$, a point of $\overline{\mathrm{Bun}}_{N^-}(S)$ consists of a pair $(\mathcal{F}_G,(\kappa^-_\lambda)_{\lambda\in\mathbf{X}^+})$, where $\mathcal{F}_G\in \mathrm{Bun}_{G}(S)$ and $\kappa^-_\lambda:\mathcal{V}_{\mathcal{F}_G}^\lambda\to \mathcal{O}_{X\times S}$ is a morphism of coherent sheaves whose dual is injective, such that the collection $(\kappa^-_\lambda)$ satisfies the Pl\"ucker relations. 
 
 In order to get a space which is stratified by the extended affine Weyl group, we need to introduce some variants of the previous stack. First, we consider the ind-algebraic stack $_\infty\overline{\mathrm{Bun}}_{N^-}$, where the morphism $\kappa^-_\lambda$ is allowed to have a pole at $x$ for all $\lambda$. For any $\nu\in\mathbf{Y}$, we let $_{\leq \nu}\overline{\mathrm{Bun}}_{N^-}$ be the variant where  each $\kappa^-_\lambda$ is now allowed to have a pole of order $\leq\langle\nu,\lambda\rangle$ at $x$ (in particular,  $_\infty\overline{\mathrm{Bun}}_{N^-}$ is the direct limit of the stacks $_{\leq \nu}\overline{\mathrm{Bun}}_{N^-}$ for $\nu\in\mathbf{Y}$), and $_{\nu}\overline{\mathrm{Bun}}_{N^-}$ be its open substack where we impose that the pole is of order exactly $\langle\nu,\lambda\rangle$ at $x$. Next, in order to make the Weyl group appear in the stratification, we consider the stacks $^1_\infty\overline{\mathrm{Bun}}_{N^-}$, $^1_{\leq\nu}\overline{\mathrm{Bun}}_{N^-}$ and $^1_{\nu}\overline{\mathrm{Bun}}_{N^-}$, where we add a structure of level one at $x$ for $\mathcal{F}_G$ (namely, an isomorphism between $G$ and the fiber of $\mathcal{F}_G$ at $x$). In particular, the group $G$ acts on the latter three stacks, and we denote by $^{I_\mathrm{u}}_\infty\overline{\mathrm{Bun}}_{N^-}$ the quotient stack $N\backslash{^{1}_\infty\overline{\mathrm{Bun}}_{N^-}}$. We have an isomorphism
 $$^1_\nu\overline{\mathrm{Bun}}_{N^-}\simeq G\times^{N^-}{^1_\nu\mathcal{N}}$$
 where $^1_\nu\mathcal{N}$ is the canonical $N^-$-torsor over $_\nu\overline{\mathrm{Bun}}_{N^-}$ (cf. §\ref{section torsors}). The Bruhat decomposition of $G/B^{-}$ induces a ``stratification" (each stratum is not necessarily smooth, since $_\nu\overline{\mathrm{Bun}}_{N^-}$ is not) of $^1_\nu\overline{\mathrm{Bun}}_{N^-}$ by Schubert strata. For $x:=wt_\nu\in W_\mathrm{ext}$, we denote by 
 $$\mathrm{IC}_x^{\frac{\infty}{2}},\quad\Delta_x^{\frac{\infty}{2}},\quad \nabla_x^{\frac{\infty}{2}} $$ the objects of $\mathrm{Perv}_T(^1_\infty\overline{\mathrm{Bun}}_{N^-})$ (where the index $T$ stands for $T$-equivariant) respectively defined as the intermediate extension, the $!$-extension and the $*$-extension of the IC-sheaf on the Schubert stratum in $^1_\nu\overline{\mathrm{Bun}}_{N^-}$ associated with $w$. In particular, these objects are equivariant with respect to $N$ and each $\mathrm{IC}_x^{\frac{\infty}{2}}$ is a simple object of $\mathrm{Perv}(^1_\infty\overline{\mathrm{Bun}}_{N^-})$.

 In the theorem below, $\mathrm{Perv}_{I_\mathrm{u}}(\mathcal{F}l^{\frac{\infty}{2}})$ denotes a certain full subcategory of the category of $T$-equivariant perverse sheaves $\mathrm{Perv}_{T}(^{I_\mathrm{u}}_\infty\overline{\mathrm{Bun}}_{N^-})$ (see §\ref{section Main definition}).  The simple objects of $\mathrm{Perv}_{I_\mathrm{u}}(\mathcal{F}l^{\frac{\infty}{2}})$ are the IC-sheaves $\{\mathrm{IC}_x^{\frac{\infty}{2}},~x\in W_\mathrm{ext}\}$. For any $\mathcal{S}\in\mathrm{Perv}_{I_\mathrm{u}}(\mathrm{Gr})$, the object $\mathcal{S}\star\mathcal{IC}^\frac{\infty}{2}_0$ denotes the object of $\mathrm{Perv}_{I_\mathrm{u}}(\mathcal{F}l^{\frac{\infty}{2}})$ obtained by the convolution process (§\ref{section Definition of convolution}), and we put $$\widetilde{\mathrm{For}}_0:\mathrm{Perv}_{I_\mathrm{u}}(\mathrm{Gr})\to \mathrm{Perv}_{I_\mathrm{u}}(\mathcal{F}l^{\frac{\infty}{2}}),~\mathcal{S}\mapsto\mathcal{S}\star\mathcal{IC}^\frac{\infty}{2}_0.$$
Moreover, the monoidal category $(\mathrm{Perv}_{G[[t]]}(\mathrm{Gr}),\star)$ acts on the left on $\mathrm{Perv}_{I_\mathrm{u}}(\mathcal{F}l^{\frac{\infty}{2}})$. Finally, recall that the forgetful functor induces a functor
 $$\mathrm{For}_0:\mathrm{Rep}_{[0]}(\check{\mathbf{G}})\to \mathrm{Rep}_{[0]}(\check{\mathbf{G}}_1\check{\mathbf{T}}). $$
 The following conjecture has its roots in \cite[§1.5]{finkelberg1997semiinfinite}.
 \begin{conj}\label{main conj}
    Assume that $\ell\geq h$. There exists an equivalence of categories 
     $$\widetilde{\mathrm{FM}}:\mathrm{Perv}_{I_\mathrm{u}}(\mathcal{F}l^{\frac{\infty}{2}})\xrightarrow{\sim}\mathrm{Rep}_{[0]}(\check{\mathbf{G}}_1\check{\mathbf{T}})$$
     such that
     $$\widetilde{\mathrm{FM}}\circ\widetilde{\mathrm{For}}_0\simeq \mathrm{For}_0\circ \mathrm{FM}, $$
     satisfying
     $$\widetilde{\mathrm{FM}}(\mathrm{IC}_x^{\frac{\infty}{2}})\simeq \widehat{\mathrm{L}}(x^{-1}\cdot_\ell0),~\widetilde{\mathrm{FM}}(\nabla_x^{\frac{\infty}{2}})\simeq \widehat{\mathrm{Z}}'(x^{-1}\cdot_\ell0),~~\widetilde{\mathrm{FM}}(\Delta_x^{\frac{\infty}{2}})\simeq \widehat{\mathrm{Z}}(x^{-1}\cdot_\ell0) $$
     for all $x\in W_\mathrm{ext}$ and
     $$\widetilde{\mathrm{FM}}(\mathcal{G}\star\mathcal{F})\simeq \widetilde{\mathrm{FM}}(\mathcal{F})\otimes\mathrm{Fr}^*(\mathrm{Sat}(\mathrm{sw}^*\mathcal{G})) $$
     functorially in $\mathcal{F}\in\mathrm{Perv}_{I_\mathrm{u}}(\mathcal{F}l^{\frac{\infty}{2}})$, $\mathcal{G}\in\mathrm{Perv}_{G[[t]]}(\mathrm{Gr})$.
 \end{conj}
 \begin{rem}
 \begin{enumerate}
     \item The strategy to prove the above conjecture is to construct a functor from  $\mathrm{mod}^\mathbf{Y}_{I_\mathrm{u}}(\mathcal{R})$ to $\mathrm{Perv}_{{I_\mathrm{u}}}(\mathcal{F}l^{\frac{\infty}{2}})$.
     \item  The objects of $\mathrm{Perv}_{I_\mathrm{u}}(\mathcal{F}l^{\frac{\infty}{2}})$ are perverse sheaves on $^{I_\mathrm{u}}_{\infty}\overline{\mathrm{Bun}}_{N^-}$ satisfying a certain ``factorization property". This property implies that the objects are constructible with respect to the Schubert strata (see Remark \ref{rem construct}) and that the simple objects are labelled by $W_\mathrm{ext}$. From this perspective, one can think of the factorization property as a replacement for the $I_\mathrm{u}$-equivariance on $\mathcal{F}l^\frac{\infty}{2}$.
 \end{enumerate}
 \end{rem}
 \subsection{What is done in this paper} We study the category $\mathrm{Perv}_{I_\mathrm{u}}(\mathcal{F}l^{\frac{\infty}{2}})$, and relate it to the category $\mathrm{mod}^\mathbf{Y}_{I_\mathrm{u}}(\mathcal{R})$. For our construction to work, we need to make the following hypothesis.
 \begin{conj}\label{conj stalks}
Let $\mathbb{k}$ be either a finite field or a finite extension of $\mathbb{Q}_\ell$. For any $i\in\mathbb{Z}$ and closed point $x\in \overline{\mathrm{Bun}}_{N^-}$, the dimension of the $\mathbb{k}$-vector space $\mathrm{H}^i((\mathrm{IC}_{\overline{\mathrm{Bun}}_{N^-}})_x)$ does not depend on the characteristic of $\mathbb{k}$.
 \end{conj}
 Here is our main result.
 \begin{thm}
Fix $\ell$, and  assume that Conjecture \ref{conj stalks} holds. Then there is an exact functor
$$\mathrm{Conv}^{\mathrm{Hecke}}:\mathrm{mod}^\mathbf{Y}_{I_\mathrm{u}}(\mathcal{R})\to\mathrm{Perv}_{{I_\mathrm{u}}}(\mathcal{F}l^{\frac{\infty}{2}}), $$
which commutes with Verdier duality, and induces a bijection between the isomorphism classes of simple objects.
 \end{thm}
Conjecture \ref{conj stalks} is expected to hold for a general reductive group $G$, but there is no proof available at the moment.
\begin{rem}
    In an upcoming work, Achar-Riche-Dhillon should give a proof of Conjecture \ref{conj stalks} (with minor restrictions on $\ell$).
\end{rem}
Below (§\ref{section Intersection cohomology}), we add the 
 following observation, well-known by experts.
 \begin{prop}
     Conjecture \ref{conj stalks} holds when $G=\mathrm{SL}_n$.
 \end{prop}
 We conjecture that $\mathrm{Conv}^{\mathrm{Hecke}}$ is an equivalence of categories. In §\ref{section Strategy of proof for the conjecture}, we give a strategy of proof for this conjecture, which follows the same steps as the analogue statement of \cite{ABBGM}. We plan to achieve the proof of this conjecture in future work.
 \subsection{Overview} After having introduced the necessary conventions in §\ref{section conventions}, we recall the main results of \cite{achar2022geometric} in §\ref{section Perverse sheaves on affine Grassmannians}. Namely, we explain the construction of the category $\mathrm{mod}^\mathbf{Y}_{I_\mathrm{u}}(\mathcal{R})$, which consists of certain ind-objects in the category $\mathrm{Perv}_{I_\mathrm{u}}(\mathrm{Gr})$. Thanks to \textit{loc. cit.} together with \cite{bezrukavnikov2024modular}, this category is equivalent to $\mathrm{Rep}_{[0]}(\check{\mathbf{G}}_1\check{\mathbf{T}})$ under certain hypothesis. In §\ref{section Drinfeld's compactifications}, we recall the constructions of the main geometric objects, namely the Drinfeld compactification and Zastava spaces. In §\ref{section Perverse sheaves on the semi-infinite flag variety}, we check that the definitions of the categories of perverse sheaves on the semi-infinite flag variety from \cite{ABBGM} still make sense in the modular setting. The main result of this section is that the dimensions of the stalks of the IC-sheaf $\mathcal{IC}^\frac{\infty}{2}_0$ on $\overline{\mathrm{Bun}}_{N^-}$ do not depend on the characteristic (see Lemma \ref{lem kuznetsov}) when $G=\mathrm{SL}_n$. This allows us to take back the main constructions of \textit{loc. cit.}. In particular, we exhibit an action of the monoidal category $(\mathrm{Perv}_{G[[t]]}(\mathrm{Gr}),\star)$ on $\mathrm{Perv}_{G[[t]]}(\mathcal{F}l^{\frac{\infty}{2}})$, see Proposition \ref{prop hecke geometric}. This property plays a crucial role in §\ref{section equivalence}, where we are able to construct a functor $\mathrm{mod}^\mathbf{Y}_{I_\mathrm{u}}(\mathcal{R})\to\mathrm{Perv}_{{I_\mathrm{u}}}(\mathcal{F}l^{\frac{\infty}{2}})$ which is exact, sends simple objects to simple objects, and commutes with Verdier duality and convolution.
 \subsection{Acknowledgments} This work was conducted during PhD studies at the Université Clermont Auvergne, under the supervision of Simon Riche and Geordie Williamson. I especially want to thank S.R. for many rereadings of earlier versions.

 This project has received funding from the European Research Council
(ERC) under the European Union’s Horizon 2020 research and innovation
programme (grant agreement No. 101002592).

	\section{Conventions}\label{section conventions}
	\subsection{The group $G$}\label{subsection the group}
	We let $\mathbb{k}$ be a finite field of characteristic $\ell>0$, $\mathbb{F}$ be an algebraically closed field of prime characteristic $p\neq \ell$, and $\check{\mathbf{G}}$ be a split connected reductive group over $\mathbb{k}$. We denote by $G$ the reductive group defined over $\mathbb{F}$, such that the Langlands dual group $G^\vee_\mathbb{k}$ over $\mathbb{k}$ is the Frobenius twist $\check{\mathbf{G}}^{(1)}$ of $\check{\mathbf{G}}$. We let $T\subset B\subset G$ be a maximal torus and Borel subgroup, $W=N_G(T)/T$ (where $N_G(T)$ is the normalizer of $T$ in $G$) denote the associated Weyl group, and $\check{\mathfrak{R}}_s\subset\check{\mathfrak{R}}_+\subset\check{\mathfrak{R}}$ (resp. $\mathfrak{R}_s\subset \mathfrak{R}_+\subset\mathfrak{R}$) denote the sets of simple coroots, positive\footnote{We take opposite conventions compared with \cite{achar2022geometric}.} coroots and coroots (resp. simple roots, positive roots and roots) determined by $B$ (i.e. the positive roots consist of the $T$ weights in $\mathrm{Lie}(B)$). We let $B^-$ be the Borel subgroup which is opposite to $B$ (i.e. such that $B\cap B^{-}=T$), put $N=[B,B]$, $N^-=[B^-,B^-]$.  
 
 Denote by $\mathbf{Y}:=X_*(T)$ (resp. $\mathbf{X}:=X^*(T)$) the coweight lattice (resp. weight lattice) of $G$, let $\mathbf{Y}^{pos}$ (resp. $\mathbf{X}^{pos}$) be the sub-semigroup spanned by $\check{\mathfrak{R}}_s$ (resp. by the positive roots), and $\mathbf{Y}^+$ (resp. $\mathbf{X}^+$) be the subset of dominant coweights (resp. dominant weights) determined by $B$. We also define the set of strictly dominant coweights
 $$\mathbf{Y}^{++}=\{\lambda\in\mathbf{Y}~|~\langle\alpha,\lambda\rangle>0~\forall \alpha\in\mathfrak{R}_+\},$$
 where $\langle\cdot,\cdot\rangle:\mathbf{X}\times \mathbf{Y}\to\mathbb{Z}$ is the usual perfect pairing. Denote by $W_{\mathrm{ext}}:=W\ltimes\mathbf{Y}$ the extended affine Weyl group. For $(w,\lambda)\in W\times\mathbf{Y}$, we will denote by $wt_\lambda$ the corresponding element of $W_\mathrm{ext}$; we will simply denote by $\lambda$ the element $et_\lambda$, where $e\in W$ is the identity. For $\lambda_1,\lambda_2\in \mathbf{Y}$, we will write $\lambda_1\leq \lambda_2$ if $\lambda_2-\lambda_1\in \mathbf{Y}^{pos}$.
	
	We assume that the quotient $\mathbf{X}/\mathbb{Z}\mathfrak{R}$ of $\mathbf{X}$ by the subgroup spanned by the roots is torsion-free. This means that the morphism of groups
	$$\mathbf{Y}\to\mathrm{Hom}_\mathbb{Z}(\mathbb{Z}\mathfrak{R},\mathbb{Z})$$
	induced by restriction is surjective, and is equivalent to requiring that the scheme-theoretic center of $G$ be a torus. In particular, there exists $\zeta\in\mathbf{Y}$ such that $\langle\alpha,\zeta\rangle=1$ for all $\alpha\in\mathfrak{R}_s$.

 We let $\check{\mathbf{T}}\subset\check{\mathbf{B}}\subset\check{\mathbf{G}}$ be the maximal torus and Borel subgroup such that $T^\vee_\mathbb{k}\simeq\check{\mathbf{T}}^{(1)}$ and $B^\vee_\mathbb{k}\simeq\check{\mathbf{B}}^{(1)}$, where the superscript $(1)$ denotes the Frobenius twist, and $T^\vee_\mathbb{k}$ (resp. $B^\vee_\mathbb{k}$) denotes the split maximal torus (resp. Borel subgroup) of the Langlands dual group of $G$ over $\mathbb{k}$ corresponding to $T$ (resp. to $B$). In the sequel, we will identify the character lattice $X^*(\check{\mathbf{T}})$ with $\mathbf{Y}$, in such a way that the morphism of $\mathbb{Z}$-modules 
$$X^*(\check{\mathbf{T}}^{(1)})\to \mathbf{Y}$$ induced by the Frobenius morphism $\mathrm{Fr}|_{\check{\mathbf{T}}}:\check{\mathbf{T}}\to \check{\mathbf{T}}^{(1)}$ identifies with $\lambda\mapsto\ell\cdot\lambda$. 
	
	For any group scheme $H$ over $\mathbb{k}$, we will denote by $\mathrm{Rep}(H)$ the category of finite dimensional representations of $H$ over $\mathbb{k}$. For any $V\in\mathrm{Rep}(\check{\mathbf{G}}^{(1)})$ and $\mu\in\mathbf{Y}$, we will denote by $V(\mu)$ the corresponding $\mu$-weight space.

	For all $\lambda\in \mathbf{X}^+$, we let $V^\lambda$ be the Weyl module of $G$ associated with the weight $\lambda$, and fix\footnote{There is a canonical choice for such a vector, see the proof of \cite[Lemma II.2.13]{jantzen2003representations}.} a highest weight vector $v_\lambda\in V^\lambda$. For $\lambda_1,\lambda_2\in\mathbf{X}^+$, there is a uniquely defined $G$-equivariant morphism \begin{equation}\label{canonical morphism}
		V^{\lambda_1+\lambda_2}\to V^{\lambda_1}\otimes V^{\lambda_2}
	\end{equation}
	sending $v_{\lambda_1+\lambda_2}$ to $v_{\lambda_1}\otimes v_{\lambda_2}$. If $\mathcal{F}_G$ is a $G$-bundle and $\mathcal{F}_T$ is a $T$-bundle on some scheme $S$ and $\lambda\in\Lambda$, we will denote by 
	$$\mathcal{V}_{\mathcal{F}_G}^\lambda:=\mathcal{F}_G\times^GV^\lambda,\quad \mathcal{L}_{\mathcal{F}_T}^\lambda:=\mathcal{F}_T\times^{T}\mathbb{A}^1_\mathbb{F}$$ 
	the induced vector bundle and line bundle respectively on $S$, where $T$ acts on the affine line $\mathbb{A}^1_\mathbb{F}$ via the character $\lambda$ in the second twisted product.
	\subsection{Stacks}
	In this note, we will study algebraic stacks and schemes defined over $\mathbb{F}$. The stacks will be attached to the smooth complete curve $X$ and to $G$.  For any group scheme $H$ over $\mathbb{F}$, we will denote by $\mathrm{Bun}_H$ the stack of principal $H$-bundles over $X$. This means that $\mathrm{Bun}_H$ is a functor from the category of $\mathbb{F}$-schemes to the category of groupoids, such that for any test scheme $S$, $\mathrm{Bun}_H(S)$ consists of the category of principal $H$-bundles over $X\times S$. We will denote by $\mathcal{F}^0_H:=H\times X\times S$ the trivial $H$-bundle. For any morphism of schemes $S'\to S$, the functor $\mathrm{Bun}_H(S)\to \mathrm{Bun}_H(S')$ is obtained by pullback. 
	
	In the sequel, we will work with variants of this stack, known as Drinfeld's compactifications. When introducing such a stack $\mathcal{X}$, we will say that it classifies \textit{something}, where something will be the category $\mathcal{X}(\mathbb{F})$. More precisely, $S$ will always be a test $\mathbb{F}$-scheme, and the groupoid $\mathcal{X}(S)$ will consist in the data of some vector bundles over $X\times S$ together with some morphisms of sheaves between them, and the functors $\mathcal{X}(S)\to \mathcal{X}(S')$ will be obtained by pullback for any morphism of $\mathbb{F}$-schemes $S'\to S$.
	
	When defining locally closed substacks of a given stack, we will often only work with the field-valued points. This is justified by the following facts. For $\mathcal{X}$ an algebraic stack, we denote by $|\mathcal{X}|$ its set of points, endowed with the Zariski topology. Recall that by definition, $|\mathcal{X}|$ consists of the set $\bigsqcup_{(\mathbb{L},f)}\mathrm{Ob}(\mathcal{X}(\mathrm{Spec}(\mathbb{L})))/\sim$, where $(\mathbb{L},f)$ runs through the maps $\mathrm{Spec}(\mathbb{L})\to \mathrm{Spec}(\mathbb{F})$, with $\mathbb{L}$ a field. Two points $x_{(\mathbb{L},f)}$, $x'_{(\mathbb{L}',f')}$ are equivalent if there is a common field extension $\mathbb{L}''$ of $\mathbb{L}$ and $\mathbb{L}'$ such that both points become isomorphic when pullbacked to $\mathcal{X}(\mathrm{Spec}(\mathbb{L}''))$. The opens of $|\mathcal{X}|$ are the subsets $|U|$, where $U\hookrightarrow\mathcal{X}$ is an open substack.  When $\mathcal{X}$ is a scheme, we recover the underlying topological space. If $p:X\to \mathcal{X}$ is a presentation, with $X$ a scheme, then the induced morphism $|X|\to |\mathcal{X}|$ is surjective, continuous and open (\cite[§ 100.4]{stacks-project}). If $T\subset|\mathcal{X}|$ is a locally closed subset, then there exists a unique reduced locally closed substack $\mathcal{Z}\subset\mathcal{X}$ such that $T=|\mathcal{Z}|$ as subsets of $|\mathcal{X}|$ (\cite[Remark 100.10.5]{stacks-project}).
	
	\subsection{Group arc schemes}\label{section Group arc schemes}
	For any closed point $x\in X$, we denote by $\hat{\mathcal{O}}_{x}$ the completion of the local ring of $X$ at $x$, let $\mathcal{K}_{x}$ be the fraction field of $\hat{\mathcal{O}}_{x}$, and put $D_x:=\mathrm{Spec}(\hat{\mathcal{O}}_{x})$. If $t$ denotes a local coordinate on $X$ at $x$, then we get isomorphisms  $\hat{\mathcal{O}}_{x}\simeq\mathbb{F}[[t]]$, $\mathcal{K}_{x}\simeq\mathbb{F}((t))$. For any $\mathbb{F}$-algebra $R$, we consider the completed tensor products $R\hat{\otimes}\hat{\mathcal{O}}_{x}$ and $R\hat{\otimes}\mathcal{K}_{x}$, so that we get identifications
	$$R\hat{\otimes}\hat{\mathcal{O}}_{x}\simeq R[[t]]\qquad  R\hat{\otimes}\mathcal{K}_{x}\simeq R((t)).$$
	For an affine group scheme $H$ over $\mathbb{F}$, we will denote by $H[[t]]$, resp. $H((t))$, resp. $H[t^{-1}]$, the functor defined on points via
	$$H[[t]](R):=H(R\hat{\otimes}\hat{\mathcal{O}}_{x}),\quad~H((t)):=H(R\hat{\otimes}\mathcal{K}_{x}),\quad H[t^{-1}]:=H(R[t^{-1}]). $$
	It is a classical fact that $H[[t]]$ is a group scheme and $H((t))$ a group ind-scheme. By \cite[§2.3]{zhu2016introduction}, $H[t^{-1}]$ is a group ind-scheme.
	
	We fix once and for all a closed point $x\in X$. For $k\in\mathbb{Z}_{\geq 0}$, we let $G_k$ be the group scheme representing the functor $R\mapsto G(R[t]/(t^k))$, and $G^k$ be the congruence subgroup defined as the kernel of the canonical morphism $G[[t]]\to G_k$. Notice that $G^0=G[[t]]$.
	
	Finally, we define the Iwahori subgroup $I\subset G[[t]]$ (resp. $I_\mathrm{u}\subset I$, resp. $I^-\subset G[[t]]$, resp. $I_\mathrm{u}^-\subset I$) as the pre-image of $B$ (resp. of $N$, resp. of $B^-$, resp. of $N^-$) through the map $G[[t]]\to G,~t\mapsto0$.
	\subsection{Perverse sheaves on stacks}\label{section Perverse sheaves on stacks}
	This paper will be concerned with defining various categories of perverse sheaves on algebraic stacks defined over $\mathbb{F}$. Since our stacks will be locally of finite type, we will freely use the formalism developed in \cite{LO}. For an algebraic stack $\mathcal{X}$ locally of finite type, we will denote by $\mathrm{D}(\mathcal{X})$ the derived category of bounded complexes of $\mathbb{k}$-sheaves, whose cohomology is constructible. We will also denote by $({^p\mathrm{D}(\mathcal{X})}^{\leq 0},{^p\mathrm{D}(\mathcal{X})}^{\geq 0} )$ the perverse $t$-structure, and by $\mathrm{Perv}(\mathcal{X})$ its heart. It follows from the constructions that, if $\mathcal{X}$ is an $\mathbb{F}$-scheme of finite type acted on by a connected $\mathbb{F}$-algebraic group $K$, then the constructible equivariant derived category $\mathrm{D}^b_K(\mathcal{X})$ of $\mathbb{k}$-sheaves -- in the sense of Bernstein-Lunts \cite{MR1299527}-- coincides with $\mathrm{D}(K\backslash\mathcal{X})$, where $K\backslash\mathcal{X}$ denotes the quotient stack. Likewise, the perverse $t$-structures coincide. We will often write  $\mathrm{D}_K(\mathcal{X})$ (resp. $\mathrm{Perv}_K(\mathcal{X})$) instead of $\mathrm{D}(K\backslash\mathcal{X})$ (resp. $\mathrm{Perv}(K\backslash\mathcal{X})$).
	
	In the sequel, we will often have the following situations. We are given a system of $\mathbb{F}$-stacks locally of finite type $(\mathcal{X}_n)_{n\in\mathbb{N}}$ together with closed embeddings $\mathcal{X}_m\hookrightarrow \mathcal{X}_n$ for $m\leq n$, and denote by $\mathcal{X}:=\varinjlim_{n}\mathcal{X}_n$ the associated ind-stack. Then for any $m\leq n$, the $*$-pushforward through the morphism $\mathcal{X}_m\hookrightarrow \mathcal{X}_m$ induces a fully-faithful and $t$-exact functor $\mathrm{D}(\mathcal{X}_m)\to \mathrm{D}(\mathcal{X}_n)$. We let $\mathrm{D}(\mathcal{X})$ and $\mathrm{Perv}(\mathcal{X})$ be the categories defined as direct limits of the systems $(\mathrm{D}(\mathcal{X}_n))_{n\geq 0}$ and $(\mathrm{Perv}(\mathcal{X}_n))_{n\geq 0}$ respectively.  Finally, assume that each $\mathcal{X}_n$ is a scheme (so that $\mathcal{X}$ is an ind-scheme)\footnote{The cases we have in mind are when $\mathcal{X}$ is either the affine Grassmannian or the affine flag variety.} and that $\mathcal{X}$ is acted on by a group scheme $K\subset G[[t]]$ such that, for all $n\geq 0$, each $\mathcal{X}_n$ is stable under the action of $K$, and there exists $u_n\in\mathbb{Z}_{\geq 0}$ such that the subgroup $K^{u_n}:=K\cap G^{u_n}$ acts trivially on $\mathcal{X}_n$. Then we put $\mathrm{D}(K\backslash \mathcal{X}_n):=\mathrm{D}(K[t/t^{u_n}]\backslash \mathcal{X}_n)$, where $K[t/t^{u_n}]$ is defined as the algebraic group $ K/K^{u_n}$. Using the fact that, for any $m\geq u_n$, the kernel of the canonical morphism $K[t/t^m]\to K[t/t^{u_n}]$ is unipotent, one can show that the previous definition does not depend on the choice of $u_n$ (cf. the appendix of \cite{baumann:hal-01491529} for a detailed treatment in the analytic case). For $m\leq n$, pushforward induces a fully faithful and $t$-exact functor $\mathrm{D}(K\backslash \mathcal{X}_m)\to \mathrm{D}(K\backslash \mathcal{X}_n)$. We define the categories $\mathrm{D}(K\backslash \mathcal{X})$ and $\mathrm{Perv}(K\backslash\mathcal{X})$ as direct limits.
	\subsection{Ind-objects} Our constructions will require to pass to categories of ind-objects. An ind-object of a (locally small) category $\mathcal{C}$ is a formal filtered colimit of objects of $\mathcal{C}$, where ``formal" means that the colimit is taken in the category of presheaves $\mathrm{Fun}(\mathcal{C}^\mathrm{op},\mathrm{Sets})$. The inductive limits in the latter category will be denoted by $\varinjlim $. For a complete account of these constructions, we refer the reader to \cite[Chapter 6]{Kash}. We will denote by $\mathrm{Ind}(\mathcal{C})$ the full subcategory of $\mathrm{Fun}(\mathcal{C}^\mathrm{op},\mathrm{Sets})$ consisting of ind-objects of $\mathcal{C}$. If $\mathcal{C}$ is abelian, then so is $\mathrm{Ind}(\mathcal{C})$. The natural functor $\mathcal{C}\to \mathrm{Ind}(\mathcal{C})$ is fully faithful. We will use the fact that a functor $F:\mathcal{C}\to \mathcal{G}$ between two categories induces a functor $\mathrm{Ind}(\mathcal{C})\to \mathrm{Ind}(\mathcal{G})$, which we will still denote by $F$. 
	\section{Perverse sheaves on affine Grassmannians}\label{section Perverse sheaves on affine Grassmannians}
	\subsection{The extended affine Weyl group} We state a few well known combinatorial facts about the extended affine Weyl group. A detailed treatment of this topic can be found in \cite[§2]{achar2021geometric} (notice that we take opposite conventions with respect to \textit{loc. cit.}, so that our positive roots correspond to the negative roots in \textit{loc. cit.}). 
	
	Recall that the affine Weyl group $W_{\mathrm{aff}}:=W\ltimes\mathbb{Z}\check{\mathfrak{R}}$ is a subgroup of $W_\mathrm{ext}$, which is endowed with a structure of Coxeter group. Namely, each root $\alpha\in \mathfrak{R}$ defines a reflection $s_\alpha\in W$, and we will denote by $S$ the set of simple reflections (i.e. associated with a simple root) of the finite Weyl group, which is known to generate $W$. It is also well known (cf. for instance \cite[II.6.3]{jantzen2003representations}) that the set of simple reflections 
	$$S_{\mathrm{aff}}:=\{(s,0),s\in S_0\}\cup\{(s_\beta,-\beta^\vee)\},$$ where $\beta^\vee$ runs through the set of largest short coroots of the irreducible components of $\mathfrak{R}^\vee$, generates $W_{\mathrm{aff}}$. The set $S_{\mathrm{aff}}$ is then a Coxeter generating system for $W_{\mathrm{aff}}$. We will denote by $\ell:W_{\mathrm{aff}}\to\mathbb{Z}_{\geq0}$ the associated length function, and by $\leq$ the associated Bruhat order. The length function and Bruhat order can be extended to the group $W_{\mathrm{ext}}$. We let $W^S_\mathrm{ext}$ denote the subset of $W_\mathrm{ext}$ consisting of elements $w$ whose length is minimal in $wW$.
	
	The group $W_{\mathrm{ext}}$ acts naturally on the vector space $E:=\mathbf{Y}\otimes_\mathbb{Z} \mathbb{R}$ via the  action 
	$$wt_\mu\cdot \lambda:=w(\lambda+\mu)$$ for any $w\in W,~\mu\in\mathbf{Y},~\lambda\in \mathbf{Y}\otimes_\mathbb{Z} \mathbb{R}$. This action defines a hyperplane arrangement in $E$, and thus a set of alcoves (which are the connected components of the complement of $E$ by the set of hyperplanes), on which $W_\mathrm{ext}$ acts. The \textit{fundamental alcove} is the alcove defined as
	$$\mathfrak{A}_{\mathrm{fund}}:=\{v\in E~|~\forall\beta\in \check{\mathfrak{R}}_+,~-1<\langle \beta,v\rangle<0\},$$
	and is a fundamental domain for the action of $W_\mathrm{ext}$ on $E$. If we set
	$$\mathcal{C}:=\{v\in E~|~\forall\beta\in \check{\mathfrak{R}}_+,~\langle \beta,v\rangle<0\}, $$
	then we have
	\begin{equation}\label{equ dominant}
		W^S_\mathrm{ext}=\{w\in W_\mathrm{ext}~|~w^{-1}\cdot(\mathfrak{A}_{\mathrm{fund}})\subset\mathcal{C}\}. 
	\end{equation}
	We next define the \textit{fundamental box} 
	$$\Pi:=\{v\in E~|~\forall\beta\in \check{\mathfrak{R}}_s,~-1<\langle \beta,v\rangle<0\},$$
	together with the set of restricted elements
	$$W^\mathrm{res}_\mathrm{ext}=\{w\in W_\mathrm{ext}~|~w^{-1}\cdot(\mathfrak{A}_{\mathrm{fund}})\subset\Pi\}.$$
	Since any alcove inside $\mathcal{C}$ is included in a subset of the form $\lambda+\Pi$ for a unique $\lambda\in -\mathbf{Y}^+$, we have
	$$W^S_\mathrm{ext}=\{wt_\lambda,~w\in W^\mathrm{res}_\mathrm{ext},~\lambda\in\mathbf{Y}^+ \}. $$
	Lastly, we will need the following property.
	\begin{lem}\label{lem length}
		For any $w\in W^S_\mathrm{ext}$ and $\lambda\in\mathbf{Y}^+$ we have $\ell(wt_\lambda w_0)=\ell(w)+\ell(t_\lambda w_0)$.
	\end{lem}
 \begin{proof}
     Notice that we have $wt_\lambda w_0=ww_0w_0t_\lambda w_0=ww_0t_{w_0(\lambda)}$. Since $w_0(\lambda)\in-\mathbf{Y}^+$ (and we have chosen opposite conventions for the choice of the fundamental alcove compared \cite{achar2022geometric}), the result then follows from \cite[Lemma 2.7]{achar2022geometric}.
 \end{proof}
	\subsection{Coset representatives} We introduce the set 
 $$^SW_\mathrm{ext}^S\subset W_\mathrm{ext},$$  
 consisting of elements $w\in W_\mathrm{ext}$ such that $w$ is minimal in $Ww$ and $vw\in W^S_\mathrm{ext}$ for all $v\in W$ (see \cite[Lemma 2.4]{achar2021geometric} for equivalent definitions). In particular, we have the inclusion $^SW_\mathrm{ext}^S\subset W_\mathrm{ext}^S$.

 For any $\lambda\in \mathbf{Y}$, we denote by $w^R_\lambda$ the minimal element in $t_\lambda W$.
 \begin{lem}{{\cite[Lemma 2.5]{achar2021geometric}}}\label{lem maximal elmt}
     Let $\lambda\in\mathbf{Y}$. We have $w^R_\lambda\in{^SW_\mathrm{ext}^S}$ iff $w_0(\lambda)\in\mathbf{Y}^{++}$. Moreover, in this case we have $w^R_\lambda=w_0t_{w_0(\lambda)}$.
 \end{lem}
 \begin{coro}\label{coro domi}
     We have a bijection
     $$\mathbf{Y}^{++}\to  {^SW_\mathrm{ext}^S},~w_0(\lambda)\mapsto w_0t_{w_0(\lambda)}.$$
 \end{coro}
 \begin{proof}
     This follows from the fact that the map $\mathbf{Y}\to W_\mathrm{ext}^S,~\lambda\to w^R_\lambda$ is a bijection, together with Lemma \ref{lem maximal elmt}.
 \end{proof}
 Finally, if we put 
 $${^SW_\mathrm{ext}^\mathrm{res}}:={^SW_\mathrm{ext}^S}\cap W_\mathrm{ext}^\mathrm{res},$$
 then as above we have
 $$^SW^S_\mathrm{ext}=\{wt_\lambda,~w\in {^SW_\mathrm{ext}^\mathrm{res}},~\lambda\in\mathbf{Y}^+ \}. $$
	
	\subsection{Perverse sheaves on flag varieties}\label{section The affine Grassmannian and the affine flag variety} We start by recalling a few well-known facts about the affine Grassmannian and affine flag variety, which can be found for instance in \cite{MR1961134}. We consider the $\mathbb{F}$-ind-schemes $\mathrm{Gr}$ and $\mathrm{Fl}$ defined as the fppf quotients
	$$\mathrm{Gr}:=G((t))/G[[t]],\qquad\mathrm{Fl}:=G((t))/I. $$
	For any $\lambda\in\mathbf{Y}$, we denote by $t^\lambda\in T((t))$ the image of $t$ through $\lambda:\mathbb{G}_\mathrm{m}((t))\to T((t))$. For any $v\in W$, we choose a lift $\dot{v}\in N_G(T)$ and, for any element $w:=vt_\lambda\in W_\mathrm{ext}$, we set $\dot{w}:=\dot{v}t^\lambda\in G((t))$. We define $\mathrm{Fl}_w$ as the $I$-orbit of the image of $\dot{w}$ under the canonical projection $G((t))\to\mathrm{Fl}$. Notice that $\mathrm{Fl}_w$ is also the $I_\mathrm{u}$-orbit of that same element. It is well known that each $\mathrm{Fl}_w$ is an affine $\mathbb{F}$-space of dimension $\ell(w)$, and that $\overline{\mathrm{Fl}_v}\subset \overline{\mathrm{Fl}_w}\Leftrightarrow v\leq w$ for all $v,w\in W_\mathrm{ext}$. We have a decomposition of the associated reduced ind-scheme
	$$(\mathrm{Fl})_{\mathrm{red}}=\bigsqcup_{w\in W_\mathrm{ext}}\mathrm{Fl}_w.$$
	Likewise, for any $w\in W_\mathrm{ext}$, we denote by $\mathrm{Gr}_w$ the $I$-orbit (which coincides with the $I_\mathrm{u}$-orbit) of the image of $\dot{w}$ under the canonical projection $G((t))\to\mathrm{Gr}$. Each $\mathrm{Gr}_w$ is an affine space over $\mathbb{F}$, of dimension $\ell(w)$ when $w\in W_\mathrm{ext}^S$.  From the fact that the canonical projection $\mathrm{Fl}\to \mathrm{Gr}$ is $I$-equivariant, one can deduce the stratification
	$$(\mathrm{Gr})_{\mathrm{red}}=\bigsqcup_{w\in W_\mathrm{ext}^S}\mathrm{Gr}_w,\quad\text{and}~\overline{\mathrm{Gr}_v}\subset \overline{\mathrm{Gr}_w}\Leftrightarrow v\leq w ~\forall v,w\in W_\mathrm{ext}^S.$$
	Finally, for any $\lambda\in \mathbf{Y}$, we let $\mathrm{Gr}^\lambda:=G[[t]]\cdot[t^\lambda]$ be the $G[[t]]$-orbit of the image $[t^\lambda]$ of $t^\lambda$ in $\mathrm{Gr}$, and let $j^\lambda:\mathrm{Gr}^\lambda\hookrightarrow\mathrm{Gr}$ denote the locally closed immersion. We have 
	$$\mathrm{dim}(\mathrm{Gr}^\lambda)=\langle2\rho,\lambda\rangle.$$
	Using the Cartan decomposition, one can show that the action of $G[[t]]$ on $\mathrm{Gr}$ induces the following equality
	$$(\mathrm{Gr})_{\mathrm{red}}=\bigsqcup_{\lambda\in \mathbf{Y}^+}\mathrm{Gr}^\lambda.$$
	Each orbit $\mathrm{Gr}^\lambda$ is a smooth $\mathbb{F}$-scheme of finite type and
	$$\overline{\mathrm{Gr}}^\lambda\subset \overline{\mathrm{Gr}}^\mu \Leftrightarrow \lambda\leq\mu\qquad\forall\lambda,\mu\in \mathbf{Y}^+.$$
	
	Next, we define three categories of perverse sheaves which will be constructible with respect to the above stratifications. Each time, the formalism recalled in the second paragraph of §\ref{section Perverse sheaves on stacks} applies.
	
	First notice that, for each $\lambda\in \mathbf{Y}^+$, there exists a large enough $k$ such that $G^k$ acts trivially on $\overline{\mathrm{Gr}}^\lambda$. So as recalled in §\ref{section Perverse sheaves on stacks}, we can define the $G[[t]]$-equivariant categories $\mathrm{D}_{G[[t]]}(\mathrm{Gr})$ and $\mathrm{Perv}_{G[[t]]}(\mathrm{Gr})$. The latter category is a \textit{highest weight category}, with weight poset $\mathbf{Y}^+$. To each $\mu\in\mathbf{Y}^+$, we associate the standard, costandard and simple objects
	$$\mathcal{I}^\mu_!:={^p\tau}^{\geq0}(j^\mu_!\underline{\mathbb{k}}_{\mathrm{Gr}^\mu}[\langle2\rho,\mu\rangle]),\quad  \mathcal{I}^\mu_*:={^p\tau}^{\leq0}(j^\mu_*\underline{\mathbb{k}}_{\mathrm{Gr}^\mu}[\langle2\rho,\mu\rangle]),\quad \mathcal{IC}^\mu:=j^\mu_{!*}\underline{\mathbb{k}}_{\mathrm{Gr}^\mu}[\langle2\rho,\mu\rangle].$$
	
	The category $\mathrm{Perv}_{I_\mathrm{u}}(\mathrm{Gr})$ is also a highest weight category, with weight poset $W^S_\mathrm{ext}$. For each $w\in W^S_\mathrm{ext}$, we will denote by $\Delta_{w}$, $\nabla_{w}$ the corresponding standard and costandard objects. They are respectively defined as the $!$-pushforward and  $*$-pushforward of the constant sheaf $\underline{\mathbb{k}}_{\mathrm{Gr}_w}[\ell(w)]$ under the inclusion $\mathrm{Gr}_w\hookrightarrow\mathrm{Gr}$ (note that these objects are indeed perverse since the locally closed inclusion is affine). The simple object associated with $w$ will be denoted by $\mathrm{L}_w$: it is defined as the intersection cohomology complex associated with $\mathrm{Gr}_w$ (and coincides with the image of the canonical morphism $\Delta_{w}\to\nabla_{w}$).
	
	We introduce notation for the category $\mathrm{Perv}_{I_\mathrm{u}}(\mathrm{Fl})$. For $w\in W_\mathrm{ext}$, we let $\mathcal{D}_w$ and $\mathcal{N}_w$ be the associated standard and costandard perverse sheaves, defined respectively as the $!$-pushforward and  $*$-pushforward of the constant sheaf $\underline{\mathbb{k}}_{\mathrm{Fl}_w}[\ell(w)]$ under the inclusion $\mathrm{Fl}_w\hookrightarrow\mathrm{Fl}$. We will denote by $\mathcal{L}_w\in \mathrm{Perv}_{I_\mathrm{u}}(\mathrm{Fl})$ the intersection cohomology complex associated with $\mathrm{Fl}_w$ (it coincides with the image of the canonical morphism $\mathcal{D}_w\to \mathcal{N}_w$). The family $(\mathcal{L}_w,~w\in W_\mathrm{ext})$ are then representatives for the isomorphism classes of the simple objects in $\mathrm{Perv}_{I_\mathrm{u}}(\mathrm{Fl})$. Likewise, the standard and costandard objects of the category $\mathrm{Perv}_{I}(\mathrm{Fl})$ are labelled by $W_\mathrm{ext}$, and we will still denote by $\mathcal{D}_w$, $\mathcal{N}_w$ the corresponding objects associated with $w$.

	\subsection{Convolution}
	We now briefly recall the construction of convolution products defined on the above categories, which will be used profusely in the sequel. We start with the bifunctor 
	\begin{equation}\label{eq monoid}
	    (-)\star(-):\mathrm{D}_{G[[t]]}(\mathrm{Gr})\times \mathrm{D}_{G[[t]]}(\mathrm{Gr})\to \mathrm{D}_{G[[t]]}(\mathrm{Gr}), 
	\end{equation}
	which will endow $\mathrm{D}_{G[[t]]}(\mathrm{Gr})$ with a structure of monoidal category (cf. \cite[§3.2.2]{BG}). For that, we consider the diagram
	\begin{equation}\label{conv diagram}
		\mathrm{Gr}\times\mathrm{Gr}\xleftarrow{p}G((t))\times \mathrm{Gr}\xrightarrow{q} G((t))\times^{G[[t]]} \mathrm{Gr}\xrightarrow{m}\mathrm{Gr}, 
	\end{equation}
	where $p$ and $q$ are the canonical projections, and $m$ is the multiplication morphism. For $(\mathcal{F},\mathcal{G})\in \mathrm{D}_{G[[t]]}(\mathrm{Gr})\times  \mathrm{D}_{G[[t]]}(\mathrm{Gr})$, $\mathcal{F}\star\mathcal{G}$ is defined as $m_*(\mathcal{F}\widetilde{\boxtimes}\mathcal{G})$, where $\mathcal{F}\widetilde{\boxtimes}\mathcal{G}$ is the image of $p^*(\mathcal{F}\boxtimes\mathcal{G})$ through the inverse of the equivalence of categories
	$$q^*:\mathrm{D}_{G[[t]]\times G[[t]]}(G((t))\times \mathrm{Gr})\to  \mathrm{D}_{G[[t]]}(G((t))\times^{G[[t]]} \mathrm{Gr}).$$
	It is a classical result (cf. for instance \cite[§6.3]{baumann:hal-01491529}) that the convolution product is $t$-exact, i.e. that $\mathcal{F}\star\mathcal{F}'\in \mathrm{Perv}_{G[[t]]}(\mathrm{Gr})$ when $\mathcal{F},\mathcal{F}'\in \mathrm{Perv}_{G[[t]]}(\mathrm{Gr})$. For any subgroup scheme $K\subset G[[t]]$, a similar construction endows the category $\mathrm{D}_{K}(\mathrm{Gr})$ with a right action of the monoidal category $\mathrm{D}_{G[[t]]}(\mathrm{Gr})$.
	
	In the sequel, we will also need the convolution bifunctor 
	$$(-)\star^I(-):\mathrm{D}_{I_\mathrm{u}}(\mathrm{Fl})\times  \mathrm{D}_{I}(\mathrm{Gr})\to \mathrm{D}_{I_\mathrm{u}}(\mathrm{Gr}).$$
	The definition is similar to that \eqref{eq monoid}, but using the diagram 
	$$\mathrm{Fl}\times\mathrm{Gr}\xleftarrow{p}G((t))\times \mathrm{Gr}\xrightarrow{q} G((t))\times^{I} \mathrm{Gr}\xrightarrow{m}\mathrm{Gr} $$
	instead of \eqref{conv diagram}. We have the following standard properties, which are stated in \cite[Lemma 3.1]{achar2022geometric}.
	\begin{lem}\label{lem conv standard}
 \begin{enumerate}
     \item For any $w\in W_\mathrm{ext}$, there exist canonical isomorphisms
     $$\mathcal{D}_w\star^I\mathcal{N}_{w^{-1}}\simeq \mathcal{D}_e\simeq \mathcal{N}_w\star^I\mathcal{D}_{w^{-1}}. $$
     \item For $w,y\in W_\mathrm{ext}$ such that $\ell(ww_0y)=\ell(w)+\ell(w_0y)$ and both $wy$ and $y$ belong to $W_\mathrm{ext}^S$, there exist canonical isomorphisms
		$$\mathcal{D}_w\star^I\Delta_y\simeq\Delta_{wy}\qquad \mathcal{N}_w\star^I\nabla_y\simeq\nabla_{wy}.$$
 \end{enumerate}
	\end{lem}

	\subsection{The geometric Satake equivalence and Finkelberg-Mirkovi\'c conjecture} Recall that $G^\vee_\mathbb{k}$ is $\check{\mathbf{G}}^{(1)}$. The following celebrated result will be used profusely in the sequel.
	\begin{thm}[\cite{Mirkovic2004GeometricLD}]\label{satake equivalence1}
		The convolution product $\star$ makes $\mathrm{Perv}_{G[[t]]}(\mathrm{Gr})$ a monoidal category, and there is an equivalence of monoidal categories 
		$$\mathrm{Sat}:(\mathrm{Perv}_{G[[t]]}(\mathrm{Gr}),\star)\xrightarrow{\sim} (\mathrm{Rep}_\mathbb{k}(\check{\mathbf{G}}^{(1)}),\otimes_\mathbb{k}).$$
	\end{thm}

	This equivalence respects the highest weight structures. This means that $\mathcal{I}^\mu_!,~\mathcal{I}^\mu_*$ and $ \mathcal{IC}^\mu$ respectively correspond -- via the functor $\mathrm{Sat}$ -- to the Weyl module $\mathrm{M}^{(1)}(\mu)$, the induced module $\mathrm{N}^{(1)}(\mu)$, and the simple module $\mathrm{L}^{(1)}(\mu)$ of highest weight $\mu$. 

  We let $\mathrm{sw}:\mathrm{Gr}\to\mathrm{Gr}$ be the isomorphism induced by the inversion map $G((t))\to G((t)),~g\mapsto g^{-1}$, so that pulling back induces an equivalence of categories $$\mathrm{sw}^*:\mathrm{Perv}_{G[[t]]}(\mathrm{Gr})\to \mathrm{Perv}_{G[[t]]}(\mathrm{Gr}).$$
  \begin{prop}{{\cite[VI.12.1]{fargues2024geometrization}}}\label{prop Chevalley inv}
      The auto-equivalence $\mathrm{sw}^*$ corresponds-- via $\mathrm{Sat}$-- to the Chevalley involution on  $\mathrm{Rep}_\mathbb{k}(\check{\mathbf{G}}^{(1)})$. In particular, for any $\mathcal{G}\in \mathrm{Perv}_{G[[t]]}(\mathrm{Gr})$ and $\mu\in\mathbf{Y}$, we have an isomorphism
      $$\mathrm{Sat}(\mathrm{sw}^*\mathcal{G})(-\mu)\simeq \mathrm{Sat}(\mathcal{G})(w_0(\mu)). $$
  \end{prop}

 The following lemma will be useful in the sequel.
 \begin{lem}\label{lem Steinberg weights}
     Let $r\in\mathbb{Z}_{\geq0}$ and put $\mu_r:=(\ell^r-1)\zeta$. The canonical morphisms
  $$\mathcal{I}^{\mu_r}_!\to\mathcal{IC}^{\mu_r} \to\mathcal{I}^{\mu_r}_*$$
  are isomorphisms. In particular, the dimensions of the stalks of $\mathcal{IC}^{\mu_r}$ are independant of $\mathrm{char}(\mathbb{k})$.
 \end{lem}
 \begin{proof}This follows from the geometric Satake equivalence together with \cite[II.3.19, (4)]{jantzen2003representations}. Let us propose a geometric proof. In \cite{zabeth2022block}, we gave a description of the blocks of the category $\mathrm{Perv}_{G[[t]]}(\mathrm{Gr})$. From this description, we see that if $\mathcal{IC}^\lambda$ is a composition factor of $\mathcal{I}^{\mu_r}_!$ (or of $\mathcal{I}^{\mu_r}_*$), then we must have $\lambda\in W_{\mathrm{aff}}\bullet_{\ell^r}\mu_r$, and $\lambda\leq\mu_r$. But $\mu_r$ is minimal in $(W_{\mathrm{aff}}\bullet_{\ell^r}\mu_r)\cap\mathbf{Y}^+$, so that $\lambda=\mu_r$. We deduce that $\mathcal{I}^{\mu_r}_!$ is simple, concluding the proof of the first assertion.

 Finally, the last assertion follows from the fact that, when $\mathrm{char}(\mathbb{k})=0$, the category $\mathrm{Perv}_{G[[t]]}(\mathrm{Gr})$ is known to be semi-simple, so that the canonical morphisms from the assertion are also isomorphisms. But the stalks of $\mathcal{I}^{\mu_r}_!$ are clearly independant of the characteristic of $\mathbb{k}$.
 \end{proof}
	
	The category $\mathrm{Perv}_{I_\mathrm{u}}(\mathrm{Gr})$ also has an interpretation in terms of representation theory. We let $\mathrm{Rep}_{[0]}(\check{\mathbf{G}})$ denote the extended principal block of the category $\mathrm{Rep}(\check{\mathbf{G}})$, defined as the Serre subcategory generated by the family $\{L(\lambda),~\lambda\in (W_{\mathrm{ext}}\bullet_\ell 0)\cap\mathbf{Y}^+\}$. Notice that, for any $w\in W^S_\mathrm{ext}$, the weight $w^{-1}\bullet_\ell0$ is dominant by \eqref{equ dominant}. We also denote by $\mathrm{Fr}:\check{\mathbf{G}}\to \check{\mathbf{G}}^{(1)}$ the Frobenius morphism. The following statement is known as the Finkelberg-Mirkovi\'c conjecture (cf. \cite[§1.5]{finkelberg1997semiinfinite}). A proof was recently obtained in \cite{bezrukavnikov2024modular}.
	\begin{thm}\label{Thm FM}
		Assume that, for any indecomposable constituent in the root system of $(\check{\mathbf{G}},\check{\mathbf{T}})$, $\ell$ is greater than the bound in Figure \ref{figure 1}. There exists an equivalence of highest weight categories
		$$\mathrm{FM}:\mathrm{Perv}_{I_\mathrm{u}}(\mathrm{Gr})\xrightarrow{\sim} \mathrm{Rep}_0(\check{\mathbf{G}}) $$
		which satisfies
		$$\mathrm{FM}(\mathrm{L}_w)\simeq L(w^{-1}\bullet_\ell0)~~\forall  w\in W^S_\mathrm{ext}.$$
		Moreover, for any $\mathcal{F}\in \mathrm{Perv}_{I_\mathrm{u}}(\mathrm{Gr}),~\mathcal{G}\in \mathrm{Perv}_{G[[t]]}(\mathrm{Gr})$, there exists a bifunctorial isomorphism
		$$\mathrm{FM}(\mathcal{F}\star \mathcal{G})\simeq \mathrm{FM}(\mathcal{F})\otimes \mathrm{Fr}^*(\mathrm{Sat}(\mathrm{sw}^*\mathcal{G})).$$
	\end{thm}
 \begin{figure}\label{figure 1}\begin{tabular}{|l|l|l|l|l|l|l|l|l|}
				\hline
				$\mathbf{A}_n$ $(n\geq1)$ & $\mathbf{B}_n$ $(n\geq2)$ & $\mathbf{C}_n$ $(n\geq3)$ & $\mathbf{C}_n$ $(n\geq4)$ & $\mathbf{E}_6$ &$\mathbf{E}_7$ &$\mathbf{E}_8$ &$\mathbf{F}_4$ & $\mathbf{G}_2$\\ \hline
				$n+1$ & $2n$ & $2n$ & $2n-2$ & $12$ &19 &31 &12 &6 \\
				\hline
		\end{tabular}\caption{Bounds on $\ell$}\label{figure 1}\end{figure}
 \begin{rem}
     The above theorem takes a more natural form in \cite{bezrukavnikov2024modular}, as the category $\mathrm{Perv}_{I_\mathrm{u}}(\mathrm{Gr})$ is replaced by the category of equivariant perverse sheaves on the \textit{opposite affine Grassmannian} $\mathrm{Gr}^{\mathrm{op}}:=G[[t]]\backslash G((t))$. For our purposes, it is however more convenient to work with the usual affine Grassmannian.  
 \end{rem}

\subsection{Iwahori-Whittaker model for the Satake equivalence}\label{section Iwahori-Whittaker model for the Satake equivalence}
	Fix a primitive $p$-th root of unity $\zeta\in\mathbb{k}$, and consider the Artin-Schreier map $\mathrm{AS}:\mathbb{G}_\mathrm{a}\to \mathbb{G}_\mathrm{a}$ determined by the map of rings $x\mapsto x^p-x$. This morphism is a Galois cover of group $\mathbb{Z}/p\mathbb{Z}$, so determines a continuous group morphism $\pi_1(\mathbb{G}_\mathrm{a},0)\to \mathbb{Z}/p\mathbb{Z}$, where $\pi_1(\mathbb{G}_\mathrm{a},0)$ is the \'etale fundamental group of $\mathbb{G}_\mathrm{a}$ with geometric base point $0$. The composition of this map with the morphism $\mathbb{Z}/p\mathbb{Z}\to \mathbf{k}^\times$ (induced by $\zeta$) yields a continuous representation of the fundamental group, and thus a local system on $\mathbb{G}_\mathrm{a}$ of rank one. We denote this local system by $\mathcal{L}_{\mathrm{AS}}$. We also let $\psi:N^-\to\mathbb{G}_a$ be a morphism of $\mathbb{F}$-algebraic groups which restricts to a non-zero morphism on each subgroup $N_{-\alpha}$ for all $\alpha\in\mathfrak{R}_s$, where $N_{-\alpha}\simeq\mathbb{G}_\mathrm{a}$ is the root subgroup associated with $-\alpha$.
	
	Define $\chi:I_{\mathrm{u}}^-\to \mathbb{G}_a$ to be the composition of $\psi$ with the evaluation map. For any $Y\subset \mathrm{Gr}$ which is a locally closed finite union of $I_{\mathrm{u}}^-$-orbits, one can show (cf. \cite[Lemma 3.2]{ciappara2021hecke}) that the $I_{\mathrm{u}}^-$-action on $Y$ factors through a quotient group of finite type $J$ such that $\chi$ factors through $\chi_J:J\to \mathbb{G}_\mathrm{a}$; we can then consider the $(J,\chi_J^*\mathcal{L}_{\mathrm{AS}})$-equivariant derived category of \'etale $\mathbb{k}$-sheaves $\mathrm{D}_{J,\chi_J^*\mathcal{L}_{\mathrm{AS}}}(Y)$. This category is by definition the subcategory of $\mathrm{D}(Y)$ consisting of constructible complexes of \'etale $\mathbb{k}$-sheaves $\mathcal{F}$ such that there exists an isomorphism $a^*\mathcal{F}\simeq \chi_J^*\mathcal{L}_{\mathrm{AS}}\boxtimes\mathcal{F}$, where $a:J\times Y\to Y$ is the action map; this definition does not depend on the choice of $J$. We then define the category $\mathrm{D}_{\mathcal{IW}}(\mathrm{Gr})$ as a direct limit of the categories $\mathrm{D}_{J,\chi_J^*\mathcal{L}_{\mathrm{AS}}}(Y)$, indexed by finite and closed unions of $I^-_\mathrm{u}$-orbits $Y$ which are ordered by inclusion. Notice that, since the transition maps are push-forwards of closed immersions, they are fully faithful functors, and one can see this limit as an increasing union of categories.
	
	The category $\mathrm{D}_{\mathcal{IW}}(\mathrm{Gr})$ admits a canonical perverse $t$-structure, and we will denote by $\mathrm{Perv}_{\mathcal{IW}}(\mathrm{Gr})$ its heart. Since the $I_\mathrm{u}^-$ orbits are affine $\mathbb{F}$-schemes, the category $\mathrm{Perv}_{\mathcal{IW}}(\mathrm{Gr})$ admits a transparent highest weight structure (in the sense of  \cite[3.7]{riche}).

 Note that the same construction (simply replacing $\mathrm{Gr}$ with $\mathrm{Fl}$) allows one to define the category $\mathrm{D}_{\mathcal{IW}}(\mathrm{Fl})$.

 For each $\lambda\in \mathbf{Y}$, let us denote by $Y_\lambda\subset\mathrm{Gr}$ the $I_\mathrm{u}^-$-orbit associated with $w_0(\lambda)$ and by $j^\lambda:Y_\lambda\hookrightarrow\mathrm{Gr}$ the locally closed immersion.  One can show that an orbit $Y_\lambda$ supports a non-zero Iwahori-Whittaker local system iff $\lambda$ is strictly dominant (i.e if $\lambda\in \zeta+\mathbf{Y}_+$), and that in this case there exists exactly one (up to isomorphism) such local system of rank one on $Y_\lambda$, which we will denote by $\mathcal{L}^\lambda_{\mathrm{AS}}$. Thus, the weight poset is given by $\mathbf{Y}^{++}=\zeta+\mathbf{Y}^{+}$, and the standard, costandard and simple objects associated with some $\lambda\in\mathbf{Y}^{++}$ are respectively given by
	$$\Delta^{\mathcal{IW}}_\lambda:=j^\lambda_{ !}\mathcal{L}^\lambda_{\mathrm{AS}}[\mathrm{dim}(Y_\lambda)],\qquad \nabla^{\mathcal{IW}}_\lambda:=j^\lambda_{ *}\mathcal{L}^\lambda_{\mathrm{AS}}[\mathrm{dim}(Y_\lambda)],\qquad \mathrm{L}_\lambda^\mathcal{IW},$$
	where $\mathrm{L}_\lambda^\mathcal{IW}$ is obtained as the image of the canonical morphism $\Delta^{\mathcal{IW}}_\lambda\to \nabla^{\mathcal{IW}}_\lambda$.  Moreover, the usual convolution product (cf. for instance \cite[§2.4]{JEP_2019__6__707_0}) endows $\mathrm{D}_{\mathcal{IW}}(\mathrm{Gr})$ with a right action 
	$$(-)\star (-):\mathrm{D}_{\mathcal{IW}}(\mathrm{Gr})\times \mathrm{D}_{G[[t]]}(\mathrm{Gr})\to \mathrm{D}_{\mathcal{IW}}(\mathrm{Gr})$$
	of the monoidal category $\mathrm{D}_{G[[t]]}(\mathrm{Gr})$. 
	
	The following result -- which gives another incarnation of the category $\mathrm{Rep}_\mathbb{k}(G^\vee)$ -- was proved in \cite[§2.2.1]{ABBGM} in the case where $\mathrm{char}(\mathbb{k})=0$, and generalized by the authors of \cite{JEP_2019__6__707_0} to the case of positive characteristic. 
	\begin{thm}\label{thm IW equivalence}
		The functor  $\mathcal{F}\mapsto \nabla^{\mathcal{IW}}_{\zeta}\star \mathcal{F}$ induces an equivalence of categories 
		$$\mathrm{Perv}_{G[[t]]}(\mathrm{Gr})\xrightarrow{\sim}\mathrm{Perv}_{\mathcal{IW}}(\mathrm{Gr}).$$
	\end{thm}

In the sequel, we will work with the $t$-exact ``averaging functor''
\begin{align*}
    &\mathrm{Av}_{I_\mathrm{u},\psi}:\mathrm{D}_{I_\mathrm{u}}(\mathrm{Gr})\to \mathrm{D}_{\mathcal{IW}}(\mathrm{Gr}),\\
    &\mathrm{Av}_{I_\mathrm{u},\psi}:\mathrm{D}_{I_\mathrm{u}}(\mathrm{Fl})\to \mathrm{D}_{\mathcal{IW}}(\mathrm{Fl}). 
\end{align*}
We refer to \cite[§3.6-3.8]{achar2021geometric} for the construction of this functor.
 \begin{rem}
     Compared with \cite{JEP_2019__6__707_0}, we work with the opposite Iwahori subgroup $I_\mathrm{u}^-$ for the equivariance condition in $\mathrm{Perv}_{\mathcal{IW}}(\mathrm{Gr})$, as opposed to $I_\mathrm{u}$ in \textit{loc. cit.}. This condition is imposed by our choice to work with $\mathrm{D}_{I_\mathrm{u}}(\mathrm{Gr})$ (as opposed to $\mathrm{D}_{I_\mathrm{u}^-}(\mathrm{Gr})$).
 \end{rem}
 \subsection{Representations of Frobenius kernels}
We define the group schemes
 $$\check{\mathbf{G}}_1:=\mathrm{ker}(\mathrm{Fr}),\quad\check{\mathbf{B}}_1:=\mathrm{ker}(\mathrm{Fr}|_{\check{\mathbf{B}}}). $$
Recall (§\ref{subsection the group}) that we have identified the character lattice $X^*(\check{\mathbf{T}})$ with $\mathbf{Y}$. For any $\lambda\in\mathbf{Y}$, we denote by $\mathbb{k}_\lambda$ the one-dimensional $\check{\mathbf{B}}_1\check{\mathbf{T}}$ induced by $\lambda$, and define the associated baby co-Verma module:
$$\widehat{Z}'(\lambda):=\mathrm{ind}_{\check{\mathbf{B}}_1\check{\mathbf{T}}}^{\check{\mathbf{G}}_1\check{\mathbf{T}}}\mathbb{k}_\lambda.$$
One can show (cf. \cite[Part II, Proposition 9.6]{jantzen2003representations}) that each $\widehat{Z}'(\lambda)$ has a simple socle -- denoted by $\widehat{L}(\lambda)$ -- and that any simple $\check{\mathbf{G}}_1\check{\mathbf{T}}$-module is isomorphic to exactly one such $\widehat{L}(\lambda)$. As in \cite[Proposition 9.6]{jantzen2003representations}, we define the baby Verma module
$$\widehat{Z}(\lambda):=\mathrm{coind}_{\check{\mathbf{B}}_1\check{\mathbf{T}}}^{\check{\mathbf{G}}_1\check{\mathbf{T}}}\mathbb{k}_\lambda.$$

In the sequel, we will work with twisted modules. For any $w\in W$, we choose a lift $\dot{w}\in N_{\check{\mathbf{G}}}(\check{\mathbf{T}})$. Since $\dot{w}$ normalises $\check{\mathbf{T}}$ and $\check{\mathbf{G}}_1$ (because the latter is a normal subgroup of $\check{\mathbf{G}}$), it also normalises $\check{\mathbf{G}}_1\check{\mathbf{T}}$. Therefore we can twist any $\check{\mathbf{G}}_1\check{\mathbf{T}}$-module $M$ by $\dot{w}$ (cf. \cite[§I.2.15]{jantzen2003representations}). The resulting module will be denoted by ${^wM}$. 
 
	\subsection{Representations of Frobenius kernels through representations of reductive groups}\label{section Representations of Frobenius kernels}
	We briefly recall how one can see $\check{\mathbf{G}}_1\check{\mathbf{T}}$-modules as $\check{\mathbf{G}}$-modules with extra structure (cf. \cite{ARKHIPOV2003114} for the original reference treating the case of quantum groups, and \cite[§4.1]{achar2022geometric} for our present setting). This identification comes from the following equivalences of categories
	\begin{equation}\label{equivalences}
		\mathrm{Rep}(\check{\mathbf{G}}_1\check{\mathbf{T}})\simeq \mathrm{Coh}^{\check{\mathbf{G}}_1\check{\mathbf{T}}}(\mathrm{pt})\simeq \mathrm{Coh}^{\check{\mathbf{G}}\times\check{\mathbf{T}}^{(1)}}(\check{\mathbf{G}}\times\check{\mathbf{T}}^{(1)}/ \check{\mathbf{G}}_1\check{\mathbf{T}}),
	\end{equation}
	where $\check{\mathbf{G}}_1\check{\mathbf{T}}$ is seen as a subgroup of $\check{\mathbf{G}}\times\check{\mathbf{T}}^{(1)}$ through the morphism $g\mapsto (g,\mathrm{Fr}(g))$. In fact, the map $(g,t)\mapsto \mathrm{Fr}(g)t^{-1}$ induces an isomorphism
	$$\check{\mathbf{G}}\times\check{\mathbf{T}}^{(1)}/ \check{\mathbf{G}}_1\check{\mathbf{T}}\simeq \check{\mathbf{G}}^{(1)}.$$
	We deduce that the category of $\check{\mathbf{G}}_1\check{\mathbf{T}}$-modules is equivalent, via \eqref{equivalences}, to the category of $\mathcal{O}(\check{\mathbf{G}}^{(1)})$-modules of finite type, which are  $\check{\mathbf{G}}$-equivariant and endowed with a $\mathbf{Y}$-grading compatible with the $\mathbf{Y}$-grading of $\mathcal{O}(\check{\mathbf{G}}^{(1)})$. Here $\mathcal{O}(\check{\mathbf{G}}^{(1)})$ is endowed with the left regular $\check{\mathbf{G}}$-representation structure, and the grading is given by
	$$\mathcal{O}(\check{\mathbf{G}}^{(1)})_\lambda=\mathrm{Ind}_{\check{\mathbf{T}}^{(1)}}^{\check{\mathbf{G}}^{(1)}}(-\lambda)\qquad\forall \lambda\in X^*(\check{\mathbf{T}}^{(1)})=\mathbf{Y}. $$
 The equivalence \eqref{equivalences} takes a $\check{\mathbf{G}}_1\check{\mathbf{T}}$-module $M$ to the $\mathbf{Y}$-graded object whose component associated with $\lambda\in\mathbf{Y}$ is
 $$\mathrm{Ind}_{\check{\mathbf{G}}_1\check{\mathbf{T}}}^{\check{\mathbf{G}}}(M\otimes\mathrm{Fr}^*(\mathbb{k}_{\check{\mathbf{T}}}(-\lambda))). $$
It follows from \cite[Lemma 4.2]{achar2022geometric} that, for $\lambda\in X^*(\check{\mathbf{T}}^{(1)})$, $\mathcal{O}(\check{\mathbf{G}}^{(1)})_\lambda$ is an ind-object of $\mathrm{Rep}(\check{\mathbf{G}}^{(1)})$. 

Finally, we spell out how to express twists by elements of $W$ in the category on the right-hand side of \eqref{equivalences}. Let $M\in\mathrm{Rep}(\check{\mathbf{G}}_1\check{\mathbf{T}})$, $w\in W$, and denote by $\widetilde{M}$ the corresponding $\mathbf{Y}$-graded, $\check{\mathbf{G}}$-equivariant, $\mathcal{O}(\check{\mathbf{G}}^{(1)})$-module of finite type. We let $^w\widetilde{M}$ be the object such that, for any $\lambda\in\mathbf{Y}$, we have
$$^w\widetilde{M}_\lambda:=M_{w(\lambda)}, $$
and such that the structure of $\mathcal{O}(\check{\mathbf{G}}^{(1)})$-module is obtained by twisting by $w$ the corresponding action on $M$.
\begin{prop}\label{prop twist coh}
    For any $w\in W$, and $M\in\mathrm{Rep}(\check{\mathbf{G}}_1\check{\mathbf{T}})$, the image of the representation $^wM$ through the equivalence \eqref{equivalences} is canonically isomorphic to $^w\widetilde{M}$.
\end{prop}
\begin{proof}
    Notice that we have canonical isomorphisms
    \begin{align*}
        \mathrm{Ind}_{\check{\mathbf{G}}_1\check{\mathbf{T}}}^{\check{\mathbf{G}}}\left({^wM}\otimes\mathrm{Fr}^*(\mathbb{k}_{\check{\mathbf{T}}}(-\lambda))\right)&\simeq \mathrm{Ind}_{\check{\mathbf{G}}_1\check{\mathbf{T}}}^{\check{\mathbf{G}}}\left({^w(M\otimes\mathrm{Fr}^*(\mathbb{k}_{\check{\mathbf{T}}}(-w(\lambda)))}\right)\\
        &\simeq {^w\left(\mathrm{Ind}_{\check{\mathbf{G}}_1\check{\mathbf{T}}}^{\check{\mathbf{G}}}\left(M\otimes\mathrm{Fr}^*(\mathbb{k}_{\check{\mathbf{T}}}(-w(\lambda)))\right)\right)},
    \end{align*}
    where the last isomorphism follows from equation (4) in \cite[§I.3.5]{jantzen2003representations}. Finally, left-translation by $\dot{w}$ induces an isomorphism of $\check{\mathbf{G}}$-modules
    $$\mathrm{Ind}_{\check{\mathbf{G}}_1\check{\mathbf{T}}}^{\check{\mathbf{G}}}\left(M\otimes\mathrm{Fr}^*(\mathbb{k}_{\check{\mathbf{T}}}(-w(\lambda)))\right)\simeq {^w\left(\mathrm{Ind}_{\check{\mathbf{G}}_1\check{\mathbf{T}}}^{\check{\mathbf{G}}}\left(M\otimes\mathrm{Fr}^*(\mathbb{k}_{\check{\mathbf{T}}}(-w(\lambda)))\right)\right)}. $$
    One the easily deduces the desired result.
\end{proof}
 
	\subsection{The regular perverse sheaf}\label{section The regular perverse sheaf}

	Let us define the ``regular perverse sheaf" $\mathcal{R}$, following \cite[§5.2]{achar2022geometric}. For any $\lambda\in\mathbf{Y}^+$, we have by construction a canonical morphism
	$$\mathcal{I}^\lambda_!\to\mathcal{I}^\lambda_*. $$
	Since the category $(\mathrm{Perv}_{G[[t]]}(\mathrm{Gr}),\star)$ is \textit{rigid} with unit object $\mathcal{IC}^0$ and  $\mathcal{I}^{-w_0(\lambda)}_*\simeq (\mathcal{I}^\lambda_!)^\vee$ (these facts can be proved directly or using the geometric Satake equivalence), the above morphism induces a morphism
	\begin{equation}\label{can0}
		\mathcal{IC}^0\to  \mathcal{I}^\lambda_*\star\mathcal{I}^{-w_0(\lambda)}_*.
	\end{equation}
	Also notice that, for any $\delta,\delta'\in\mathbf{Y}^+$, the perverse sheaf $\mathcal{I}^\delta_*\star\mathcal{I}^{\delta'}_*$ is supported on $\overline{\mathrm{Gr}}^{\delta+\delta'}$, and its restriction to $\mathrm{Gr}^{\delta+\delta'}$ is $\underline{\mathbb{k}}_{\mathrm{Gr}^{\delta+\delta'}}[\langle2\rho,\delta+\delta'\rangle]$. Thus, adjunction yields a canonical morphism
	\begin{equation}\label{can1}
		\mathcal{I}^\delta_*\star\mathcal{I}^{\delta'}_*\to \mathcal{I}^{\delta+\delta'}_*.
	\end{equation}
	Let $\mu\in\mathbf{Y}$. We put 
	$$\mathcal{R}_\mu:=\varinjlim_{\lambda\in \mathbf{Y}^+\cap(-w_0(\mu)+\mathbf{Y}^+)}  \mathcal{I}^{w_0(\mu)+\lambda}_*\star\mathcal{I}^{-w_0(\lambda)}_*.$$
	
	The transition morphisms are given by
	\begin{align*}
		\mathcal{I}^{w_0(\mu)+\lambda}_*\star\mathcal{I}^{-w_0(\lambda)}_*&\to \mathcal{I}^{w_0(\mu)+\lambda}_*\star(\mathcal{I}^\nu_*\star \mathcal{I}^{-w_0(\nu)}_*) \star\mathcal{I}^{-w_0(\lambda)}_*\to \mathcal{I}^{w_0(\mu)+\lambda+\nu}_*\star\mathcal{I}^{-w_0(\lambda+\nu)}_*,
	\end{align*}
	where the first, resp. second, morphism is given by \eqref{can0}, resp. \eqref{can1}. 
	\begin{prop}\label{Proposition regular perverse sheaf}
		For all $\mu\in\mathbf{Y}$, we have an isomorphism of ind-objects
		$$ \mathrm{Sat}(\mathrm{sw}^*(\mathcal{R}_\mu))\simeq  \mathcal{O}(\check{\mathbf{G}}^{(1)})_\mu.$$
	\end{prop}
	\begin{proof}
		This follows from the description of $\mathcal{O}(\check{\mathbf{G}}^{(1)})_\mu$ as an ind-object of $\mathrm{Rep}(\check{\mathbf{G}}^{(1)})$ (\cite[Lemma 4.2]{achar2022geometric}), together with the fact that $\mathrm{sw}^*(\mathcal{I}^{\lambda}_*)\simeq\mathcal{I}^{-w_0(\lambda)}_* $ for all $\lambda\in\mathbf{Y}^+$ thanks to Proposition \ref{prop Chevalley inv}.
	\end{proof}
	We then define the $\mathbf{Y}$-graded ind-object
	$$\mathcal{R}:=\bigoplus_{\mu\in\mathbf{Y}} \mathcal{R}_\mu.$$
	The morphism \eqref{can0} defines a unit map $\eta:\mathcal{IC}^0\to \mathcal{R}_0$, and one can define multiplication morphisms $\mathcal{R}_\mu\star\mathcal{R}_{\mu'}\to \mathcal{R}_{\mu+\mu'}$ for all $\mu,\mu'\in\mathbf{Y}$, in a similar way to the definition of the transition morphisms above (see \cite[(5.5)]{achar2022geometric}). These morphisms make $\mathcal{R}$ an algebra object, called the \textit{regular perverse sheaf}.
	
	\subsection{Modules over the regular perverse sheaf}\label{section Modules over the regular perverse sheaf} Here we follow \cite[§5.4]{achar2022geometric} (here we only consider the case $A=\emptyset$, in the notation of \textit{loc. cit.}). By definition, a $\mathbf{Y}$-graded right $\mathcal{R}$-module
	$$\mathcal{F}=\bigoplus_{\lambda\in\mathbf{Y}}\mathcal{F}_\lambda$$
	is a $\mathbf{Y}$-graded ind-object in the category $\mathrm{Perv}_{I_\mathrm{u}}(\mathrm{Gr})$, along with a collection of morphisms
	$$\mathcal{F}_\mu\star\mathcal{R}_\lambda\to \mathcal{F}_{\mu+\lambda}$$
	for all $\lambda,\mu$, satisfying obvious associativity and unit axioms. We let 
	$$\mathrm{Mod}^\mathbf{Y}_{I_\mathrm{u}}(\mathcal{R}) $$
	denote the category of $\mathbf{Y}$-graded right $\mathcal{R}$-modules (which is a full subcategory of the category of $\mathbf{Y}$-graded ind-objects in $\mathrm{Perv}_{I_\mathrm{u}}(\mathrm{Gr})$). This category is abelian. For any $\nu\in\mathbf{Y}$, there is a \textit{shift-of-grading} endofunctor $\langle\nu\rangle$ defined by
	$$(\mathcal{F}\langle\nu\rangle)_\lambda:=\mathcal{F}_{\lambda-\nu} \quad \forall \lambda\in\mathbf{Y}.$$
	Next, we define a functor 
	$$\Phi:\mathrm{Perv}_{I_\mathrm{u}}(\mathrm{Gr})\to \mathrm{Mod}^\mathbf{Y}_{I_\mathrm{u}}(\mathcal{R}),~\mathcal{F}\mapsto\bigoplus_{\mu\in\mathbf{Y}}\mathcal{F}\star\mathcal{R}_\mu. $$
	This functor is exact (by exactness of the convolution product) and faithful, since it doesn't kill any nonzero object. Indeed, since the unit map $\eta$ is injective and $\star$ is exact, we get an injective map $\mathcal{F}\simeq \mathcal{F}\star\mathcal{IC}^0\to \mathcal{F}\star\mathcal{R}_0$. 
  \begin{prop}{{\cite[Lemma 5.1]{achar2022geometric}}}\label{prop iso morphisms}
     Let $\mathcal{F}\in \mathrm{Perv}_{I_\mathrm{u}}(\mathrm{Gr})$, $\mathcal{G}\in \mathrm{Mod}^\mathbf{Y}_{I_\mathrm{u}}(\mathcal{R})$. There is a natural isomorphism
     $$\mathrm{Hom}_{\mathrm{Mod}^\mathbf{Y}_{I_\mathrm{u}}(\mathcal{R})}(\Phi(\mathcal{F}),\mathcal{G})\simeq \mathrm{Hom}_{\mathrm{Perv}_{I_\mathrm{u}}(\mathrm{Gr})}(\mathcal{F},\mathcal{G}_0).$$
 \end{prop}Since the above result will be used a lot in §\ref{section equivalence}, we recall its proof.
\begin{proof}
    A morphism $\alpha:\Phi(\mathcal{F})\to \mathcal{G}$ is sent to the composition of the morphism $\mathcal{F}\to \mathcal{F}\star\mathcal{R}_0$ with the degree-zero part $\alpha_0:\mathcal{F}\star\mathcal{R}_0=\Phi(\mathcal{F})_0\to \mathcal{G}_0$. 
    
    Conversely, consider a morphism $\beta:\mathcal{F}\to \mathcal{G}_0$. For any $\lambda\in\mathbf{Y}$, we construct a morphism $\Phi(\mathcal{F})_\lambda\to \mathcal{G}_\lambda$ as the composition
    $$\mathcal{F}\star\mathcal{R}_\lambda\xrightarrow{\beta\star\mathrm{id}}\mathcal{G}_0\star\mathcal{R}_\lambda\to \mathcal{G}_\lambda, $$
    where the last morphism comes from the $\mathcal{R}$-module structure of $\mathcal{G}$.

    One can easily check that the two functors that we have constructed are inverse to each other.
\end{proof}	
An object of $\mathrm{Mod}^\mathbf{Y}_{I_\mathrm{u}}(\mathcal{R})$ is called a \textit{free graded $\mathcal{R}$-module of finite type} if it is isomorphic to a direct sum of shifts of objects of the form $\Phi(\mathcal{F})$, for $\mathcal{F}\in\mathrm{Perv}_{I_\mathrm{u}}(\mathrm{Gr})$.  An object of $\mathrm{Mod}^\mathbf{Y}_{I_\mathrm{u}}(\mathcal{R})$ is said to be \textit{finitely generated} if it is  a quotient of a free graded $\mathcal{R}$-module.

	We now describe the simple objects of the abelian category $\mathrm{Mod}^\mathbf{Y}_{I_\mathrm{u}}(\mathcal{R})$. Let $w\in W_{\mathrm{ext}}$. By construction of the set $W_{\mathrm{ext}}^{\mathrm{res}}$, there exist $y\in W_{\mathrm{ext}}^{\mathrm{res}},~\lambda\in \mathbf{Y}$ such that $w=yt_\lambda$. We then put 
	$$\widehat{\mathcal{L}}_w:=\Phi(\mathrm{L}_y)\langle-\lambda\rangle\in \mathrm{Mod}^\mathbf{Y}_{I_\mathrm{u}}(\mathcal{R})$$
 (here the choice of a couple $(y,\lambda)\in W_{\mathrm{ext}}^{\mathrm{res}}\times\mathbf{Y}$ such that $yt_\lambda=w$ is non-unique, but a different choice $(y',\lambda')$ yields an isomorphic object $\Phi(\mathrm{L}_{y'})\langle-\lambda'\rangle$ thanks to the comments in \cite[§5.5]{achar2022geometric}). We thus have the equality $\widehat{\mathcal{L}}_{yt_\lambda}=\widehat{\mathcal{L}}_y\langle-\lambda\rangle$ for all $y\in W_{\mathrm{ext}}^{\mathrm{res}},~\lambda\in \mathbf{Y}$.
	\begin{thm}{{\cite[Theorem 5.6]{achar2022geometric}}}\label{thm simple g1t}
		The assignment $w\mapsto \widehat{\mathcal{L}}_{w}$ induces a bijection between $W_{\mathrm{ext}}$ and the set of isomorphism classes of simple objects of $\mathrm{Mod}^\mathbf{Y}_{I_\mathrm{u}}(\mathcal{R})$.
	\end{thm}
	We will denote by $\mathrm{mod}^\mathbf{Y}_{I_\mathrm{u}}(\mathcal{R})$ the full subcategory of $\mathrm{Mod}^\mathbf{Y}_{I_\mathrm{u}}(\mathcal{R})$ consisting of finite length objects.  
	\begin{prop}{{\cite[Theorem 7.9]{achar2022geometric}}}\label{prop finite length}
		An object of $\mathrm{Mod}^\mathbf{Y}_{I_\mathrm{u}}(\mathcal{R})$ is finitely generated iff it is of finite length.
	\end{prop}
 We end this subsection with the crucial \textit{Hecke eigensheaf property} of objects of the category $\mathrm{Mod}^\mathbf{Y}_{I_\mathrm{u}}(\mathcal{R})$.
 \begin{prop}\label{prop Hecke property}
     Let $\mathcal{F}\in \mathrm{Mod}^\mathbf{Y}_{I_\mathrm{u}}(\mathcal{R}),~\mathcal{G}\in \mathrm{Rep}(G^\vee_\mathbf{k})$. We have canonical isomorphisms
     $$\mathcal{F}\star\mathcal{G}\simeq\bigoplus_{\nu\in\mathbf{Y}}\mathrm{Sat}(\mathrm{sw}^*\mathcal{G})(\nu)\otimes\mathcal{F}\langle\nu\rangle\simeq\bigoplus_{\nu\in\mathbf{Y}}\mathrm{Sat}(\mathcal{G})(-w_0(\nu))\otimes\mathcal{F}\langle\nu\rangle.  $$
 \end{prop}
 \begin{proof}
     Consider the canonical morphism $\mathcal{F}\to\mathcal{F}\star\mathcal{R}_0$ induced by the unit map. Convolving with $\mathcal{G}$ we get a map $\mathcal{F}\star\mathcal{G}\to \mathcal{F}\star\mathcal{R}_0\star \mathcal{G}$. But one can show (by applying the equivalence $(\mathrm{Sat}\circ\mathrm{sw}^*)^{-1}$ and using Proposition \ref{Proposition regular perverse sheaf}, cf. the proof of \cite[Lemma 5.2]{achar2022geometric}) that we have a canonical isomorphism
     $$\mathcal{R}_0\star \mathcal{G}\simeq \bigoplus_\nu\mathrm{Sat}(\mathrm{sw}^*\mathcal{G})(\nu)\otimes\mathcal{R}_{-\nu}. $$
    So we obtain a morphism 
    $$\mathcal{F}\star\mathcal{G}\to\bigoplus_\nu \mathrm{Sat}(\mathrm{sw}^*\mathcal{G})(\nu)\otimes\mathcal{F}\star\mathcal{R}_{-\nu}. $$
    Composing the above morphism with the morphism $\bigoplus_\nu\mathcal{F}\star\mathcal{R}_{-\nu}\to\bigoplus_\nu\mathcal{F}\langle\nu\rangle$ coming from the $\mathcal{R}$-module structure of $\mathcal{F}$, we obtain a map
    $$\mathcal{F}\star\mathcal{G}\to\bigoplus_{\nu\in\mathbf{Y}}\mathrm{Sat}(\mathrm{sw}^*\mathcal{G})(\nu)\otimes\mathcal{F}\langle\nu\rangle.$$
    This morphism is functorial in $\mathcal{F}$ and $\mathcal{G}$, and satisfies an obvious compatibility with the tensor product and convolution products (cf. \cite[§1.3.6]{ABBGM}). Therefore, one can easily check (as in the proof of \cite[Proposition 2.3]{ARKHIPOV2003114}) that this morphism is in fact an isomorphism.
 \end{proof}
	\subsection{Verdier duality}\label{section Verdier} We now recall the definition of the Verdier duality operation on $\mathrm{mod}^\mathbf{Y}_{I_\mathrm{u}}(\mathcal{R})$ (cf. \cite[§9.1]{achar2022geometric}). The goal is to explain the construction of an exact and involutive anti-equivalence $\mathbb{D}:\mathrm{mod}^\mathbf{Y}_{I_\mathrm{u}}(\mathcal{R})\to \mathrm{mod}^\mathbf{Y}_{I_\mathrm{u}}(\mathcal{R})$ which satisfies
	\begin{equation}\label{eq prescribe D}
		\mathbb{D}(\Phi(\mathcal{F})\langle\lambda\rangle)\simeq \Phi(\mathbb{D}(\mathcal{F}))\langle\lambda\rangle\qquad \forall \mathcal{F}\in\mathrm{Perv}_{I_\mathrm{u}}(\mathrm{Gr}),~\lambda\in\mathbf{Y}.
	\end{equation}
	In fact, \eqref{eq prescribe D} allows one to define $\mathbb{D}$ for free graded $\mathcal{R}$-modules of finite type. This functor can then be extended (although this is not completely obvious) to the whole category $\mathrm{mod}^\mathbf{Y}_{I_\mathrm{u}}(\mathcal{R})$.
	
	We let $\mathrm{Free}^\mathbf{Y}_{I_\mathrm{u}}(\mathcal{R})$ denote the additive $\mathbb{k}$-linear category whose objects are formal finite direct sums
	$$\bigoplus_{j\in J} (\mathcal{F}_j,\lambda_j), $$
	where $J$ is a finite set, and $\mathcal{F}_j\in\mathrm{Perv}_{I_\mathrm{u}}(\mathrm{Gr})$, $\lambda_j\in\mathbf{Y}$ for all $j$. The morphisms are defined by
	\begin{align*}
		\mathrm{Hom}_{\mathrm{Free}^\mathbf{Y}_{I_\mathrm{u}}(\mathcal{R})}(\bigoplus_{j\in J} (\mathcal{F}_j,\lambda_j),\bigoplus_{i\in I} (\mathcal{F}'_i,\lambda_i))&:=\mathrm{Hom}_{\mathrm{mod}^\mathbf{Y}_{I_\mathrm{u}}(\mathcal{R})}(\bigoplus_{j\in J} \Phi(\mathcal{F}_j)\langle\lambda_j\rangle,\bigoplus_{i\in I} \Phi(\mathcal{F}'_i)\langle\lambda_i\rangle)\\
		&\simeq \bigoplus_{j,i}\mathrm{Hom}_{\mathrm{mod}^\mathbf{Y}_{I_\mathrm{u}}(\mathcal{R})}( \Phi(\mathcal{F}_j), \Phi(\mathcal{F}'_i)\langle\lambda_i-\lambda_j\rangle)
	\end{align*}
	for an object $\bigoplus_{i\in I} (\mathcal{F}'_i,\lambda_i)$ of $\mathrm{Free}^\mathbf{Y}_{I_\mathrm{u}}(\mathcal{R})$. We define the exact and fully faithful functor 
	$$\Phi_2:\mathrm{Free}^\mathbf{Y}_{I_\mathrm{u}}(\mathcal{R})\to  \mathrm{mod}^\mathbf{Y}_{I_\mathrm{u}}(\mathcal{R}),~\bigoplus_{j\in J} (\mathcal{F}_j,\lambda_j)\to \bigoplus_{j\in J} \Phi(\mathcal{F}_j)\langle\lambda_j\rangle.$$
	By construction, the essential image of $\Phi_2$ coincides with the free graded $\mathcal{R}$-modules.
	
	We now define a functor 
	$$\mathbb{D}_{\mathrm{fr}}:\mathrm{Free}^\mathbf{Y}_{I_\mathrm{u}}(\mathcal{R})\to \mathrm{Free}^\mathbf{Y}_{I_\mathrm{u}}(\mathcal{R})^\mathrm{op},$$
 where $(-)^\mathrm{op}$ means that we take the opposite category. On objects, it sends $\bigoplus_{j\in J} (\mathcal{F}_j,\lambda_j)$ to $\bigoplus_{j\in J} (\mathbb{D}(\mathcal{F}_j),\lambda_j)$. To describe how it behaves on morphisms, it is enough to treat the case of a morphism between two objects of the form $(\mathcal{F},\lambda),(\mathcal{G},\mu)$ by additivity. We will construct a bifunctorial isomorphism  
 $$\mathrm{Hom}_{\mathrm{Free}^\mathbf{Y}_{I_\mathrm{u}}(\mathcal{R})}((\mathcal{F},\lambda),(\mathcal{G},\mu))\to \mathrm{Hom}_{\mathrm{Free}^\mathbf{Y}_{I_\mathrm{u}}(\mathcal{R})}((\mathbb{D}(\mathcal{G}),\mu),(\mathbb{D}(\mathcal{F}),\lambda)).$$  
 By definition, the space $\mathrm{Hom}_{\mathrm{Free}^\mathbf{Y}_{I_\mathrm{u}}(\mathcal{R})}((\mathcal{F},\lambda),(\mathcal{G},\mu))$ coincides with the first space below
	\begin{align*}
		\mathrm{Hom}_{\mathrm{mod}^\mathbf{Y}_{I_\mathrm{u}}(\mathcal{R})}(\Phi(\mathcal{F}),\Phi(\mathcal{G})\langle\mu-\lambda\rangle)&\simeq  \mathrm{Hom}_{\mathrm{mod}^\mathbf{Y}_{I_\mathrm{u}}(\mathcal{R})}(\mathcal{F},(\Phi(\mathcal{G})\langle\mu-\lambda\rangle)_0)\\
		&=\mathrm{Hom}_{\mathrm{mod}^\mathbf{Y}_{I_\mathrm{u}}(\mathcal{R})}(\mathcal{F},\mathcal{G}\star\mathcal{R}_{\lambda-\mu})\\
		&=\varinjlim_\nu\mathrm{Hom}_{\mathrm{Perv}_{I_\mathrm{u}}(\mathrm{Gr})}(\mathcal{F},\mathcal{G}\star\mathcal{I}^{w_0(\lambda-\mu)+\nu}_*\star\mathcal{I}^{-w_0(\nu)}_*),
	\end{align*}
	where $\nu\in \mathbf{Y}^+\cap(-w_0(\lambda-\mu)+\mathbf{Y}^+)$. But for any such $\nu$, we have the following isomorphisms
	\begin{align*}
		\mathrm{Hom}_{\mathrm{Perv}_{I_\mathrm{u}}(\mathrm{Gr})}(\mathcal{F},\mathcal{G}\star\mathcal{I}^{w_0(\mu-\lambda)+\nu}_*\star\mathcal{I}^{-w_0(\nu)}_*)&\simeq \mathrm{Hom}_{\mathrm{Perv}_{I_\mathrm{u}}(\mathrm{Gr})}(\mathcal{F}\star\mathcal{I}^{\lambda-\mu-w_0(\nu)}_!\star\mathcal{I}^{\nu}_!,\mathcal{G})\\
		&\simeq \mathrm{Hom}_{\mathrm{Perv}_{I_\mathrm{u}}(\mathrm{Gr})}(\mathbb{D}(\mathcal{G}),\mathbb{D}(\mathcal{F})\star\mathcal{I}^{\lambda-\mu-w_0(\nu)}_*\star\mathcal{I}^{\nu}_*),
	\end{align*}
	where the last isomorphism is due to the fact that Verdier duality commutes with the convolution product and sends standard objects to costandard ojects. These isomorphisms yield an isomorphism between $\mathrm{Hom}_{\mathrm{Free}^\mathbf{Y}_{I_\mathrm{u}}(\mathcal{R})}((\mathcal{F},\lambda),(\mathcal{G},\mu))$ and 
	\begin{equation}\label{eq lim D}
		\varinjlim_\nu \mathrm{Hom}_{\mathrm{Perv}_{I_\mathrm{u}}(\mathrm{Gr})}(\mathbb{D}(\mathcal{G}),\mathbb{D}(\mathcal{F})\star\mathcal{I}^{\lambda-\mu-w_0(\nu)}_*\star\mathcal{I}^{\nu}_*).
	\end{equation}
	Setting $\nu':=-w_0(\mu-\lambda)+\nu$, we obtain an isomorphism between $\varinjlim_\nu \mathcal{I}^{\lambda-\mu-w_0(\nu)}_*\star\mathcal{I}^{\nu}_*$ and $\mathcal{R}_{\mu-\lambda}$. And finally between \eqref{eq lim D} and 
	\begin{align*}
		\mathrm{Hom}_{\mathrm{mod}^\mathbf{Y}_{I_\mathrm{u}}(\mathcal{R})}(\mathbb{D}(\mathcal{G}),(\Phi(\mathbb{D}(\mathcal{F}))\langle\lambda-\mu\rangle)_0)&\simeq \mathrm{Hom}_{\mathrm{mod}^\mathbf{Y}_{I_\mathrm{u}}(\mathcal{R})}(\Phi(\mathbb{D}(\mathcal{G})),\mathbb{D}(\Phi(\mathcal{F}))\langle\lambda-\mu\rangle)\\
		&=\mathrm{Hom}_{\mathrm{Free}^\mathbf{Y}_{I_\mathrm{u}}(\mathcal{R})}((\mathbb{D}(\mathcal{G}),\mu),(\mathbb{D}(\mathcal{F}),\lambda)).
	\end{align*}
	\begin{lem}{{\cite[Lemma 9.2]{achar2022geometric}}}\label{lem exact D}
		The functor  $\mathbb{D}_{\mathrm{fr}}$ is exact.
	\end{lem}
	We will admit the following fact, which concludes the desired construction and whose full proof is given in \cite[§9.1]{achar2022geometric}.
	\begin{prop}\label{prop D}
		There exists an exact and involutive anti-equivalence $\mathbb{D}:\mathrm{mod}^\mathbf{Y}_{I_\mathrm{u}}(\mathcal{R})\to \mathrm{mod}^\mathbf{Y}_{I_\mathrm{u}}(\mathcal{R})$ which coincides with (the functor induced by) $\mathbb{D}_{\mathrm{fr}}$ on the subcategory of free graded $\mathcal{R}$-modules.
	\end{prop}
	\begin{rem}\label{rem duality}
		We can get an ``explicit" description of $\mathbb{D}$. Indeed, any object $\mathcal{F}\in \mathrm{mod}^\mathbf{Y}_{I_\mathrm{u}}(\mathcal{R})$ can be written as the cokernel of a morphism $\alpha:\mathcal{F}_1\to \mathcal{F}_2$ between two free graded $\mathcal{R}$-modules thanks to Proposition \ref{prop finite length}. The object $\mathbb{D}(\mathcal{F})$ then coincides with the kernel of the morphism $\mathbb{D}(\alpha):\mathbb{D}(\mathcal{F}_2)\to \mathbb{D}(\mathcal{F}_1)$, whose construction was explained above. 
	\end{rem}
	From \eqref{eq prescribe D} and the definition of the object $\widehat{\mathcal{L}}_w$ (together with the fact that Verdier duality preserves simple objects of the category $\mathrm{Perv}_{I_\mathrm{u}}(\mathrm{Gr})$, since those are IC-sheaves), we get that
	\begin{equation}\label{eq D reserves simples}
		\mathbb{D}(\widehat{\mathcal{L}}_w)\simeq \widehat{\mathcal{L}}_w\qquad\forall w\in W_{\mathrm{ext}}. 
	\end{equation}
	\subsection{Baby Verma and co-Verma modules}\label{section Baby Verma and co-Verma modules} In this subsection we follow \cite[§5.6]{achar2022geometric}, defining the analogue of baby co-Verma modules as ind-objects of $\mathrm{Perv}_{I_\mathrm{u}}(\mathrm{Gr})$. For any $\mu\in\mathbf{Y}^+$, the orbit $\mathrm{Gr}_{t_{\mu}}$ is open inside $\mathrm{Gr}^\mu$ (recall that our choice of Borel subgroup $B$ corresponds to the opposite Borel subgroup in \cite{achar2022geometric}). So by adjunction, we obtain a canonical morphism 
 \begin{equation}\label{eq convo}
	    \mathcal{I}^\mu_*\to \nabla_{\mu}.
	\end{equation} Thanks to Lemma \ref{lem conv standard}, we get a canonical isomorphism 
 $$\mathcal{N}_w\star^I\nabla_{\mu}\simeq \nabla_{wt_{\mu}},$$
 for all $w\in W_{\mathrm{ext}}^S$, $\mu\in\mathbf{Y}^+$ (notice that $\ell(wt_{\mu}w_0)=\ell(w)+\ell(t_{\mu}w_0)$ by Lemma \ref{lem length}. since $w\in W_{\mathrm{ext}}^S$). On the other hand, if we denote by $\pi:\mathrm{Fl}\to\mathrm{Gr}$ the canonical projection, we easily obtain isomorphisms
 \begin{equation}\label{eq standard}
     \mathcal{N}_w\star^I \mathcal{I}^\mu_*\simeq(\pi_*\mathcal{N}_w)\star\mathcal{I}^\mu_*\simeq \nabla_{w}\star \mathcal{I}^\mu_*,
 \end{equation} 
 where the first isomorphism follows from \cite[Lemma 2.5]{JEP_2019__6__707_0}, and the second one from the fact that $\pi_*\mathcal{N}_w\simeq\nabla_w$ (see \cite[(3.4)]{achar2022geometric}). 

  Thus, we get a canonical morphism
	\begin{equation}\label{eq 513}
		\nabla_{w}\star \mathcal{I}^\mu_*\to \nabla_{wt_{\mu}}
	\end{equation}
by convolving \eqref{eq convo} with $\mathcal{N}_{w}$. 
	
	Next, for $w\in W_{\mathrm{ext}}$, $\mu\in\mathbf{Y}$, we put
	$$(\widehat{\mathcal{Z}}'_w)_\mu=\varinjlim_{\lambda} \nabla_{wt_{\mu+\lambda}}\star \mathcal{I}^{-w_0(\lambda)}_*,$$
	where $\lambda$ runs over the elements of $\mathbf{Y}^+$ such that $wt_{\mu+\lambda}\in W^S_\mathrm{ext}$. The transition morphisms in the above inductive limit are given by
	$$ \nabla_{wt_{\mu+\lambda}}\star \mathcal{I}^{-w_0(\lambda)}_*\to \nabla_{wt_{\mu+\lambda}}\star\mathcal{I}^{\nu}_*\star\mathcal{I}^{-w_0(\nu)}_*\star \mathcal{I}^{-w_0(\lambda)}_*\to \nabla_{wt_{\mu+\lambda+\nu}}\star \mathcal{I}^{-w_0(\lambda+\nu)}_* $$
	where the first, resp. second, morphism is obtained via \eqref{can0}, resp. \eqref{can1} and \eqref{eq 513}. We then put 
	$$\widehat{\mathcal{Z}}'_w:=\bigoplus_{\mu\in\mathbf{Y}}(\widehat{\mathcal{Z}}'_w)_\mu. $$
	The morphisms $(\widehat{\mathcal{Z}}'_w)_\mu\star \mathcal{R}_{\mu'}\to (\widehat{\mathcal{Z}}'_w)_{\mu+\mu'}$ giving the structure of $\mathcal{R}$-module are obtained thanks to the morphisms
	\begin{align*}
		(\nabla_{wt_{\mu+\lambda}}\star \mathcal{I}^{-w_0(\lambda)}_*) \star(\mathcal{I}^{w_0(\mu')+\lambda'}_*\star\mathcal{I}^{-w_0(\lambda')}_*)&\to (\nabla_{wt_{\mu+\lambda}}\star \mathcal{I}^{w_0(\mu')+\lambda'}_*\star\mathcal{I}^{-w_0(\lambda'+\lambda)}_*)\\
		& \to\nabla_{wt_{\mu+\mu'+(\lambda+\lambda')}}\star\mathcal{I}^{-w_0(\lambda'+\lambda)}_*,
	\end{align*}
	where the first morphism is obtained by the commutativity constraint and \eqref{can0}, and the second one follows from \eqref{eq 513}.
	
	The next result allows to recover all the co-Verma modules by acting on the one associated with the trivial element $0$. This will later allow us to reduce our considerations to the case of $\widehat{\mathcal{Z}}'_0$.
	\begin{prop}\label{prop convolution costandard}
		For any  $\gamma\in\mathbf{Y}^+,~\nu\in\mathbf{Y}$, we have a canocical isomorphism
		$$ \mathcal{N}_\gamma\star^I\widehat{\mathcal{Z}}'_\nu\simeq \widehat{\mathcal{Z}}'_{\gamma+\nu} $$
	\end{prop}
	\begin{proof}
		Recall that, for any $\mu\in\mathbf{Y}$, we have
		$$(\widehat{\mathcal{Z}}'_\nu)_\mu=\varinjlim_{\lambda\in\mathbf{Y}^+-\mu-\nu} \nabla_{\nu+\mu+\lambda}\star \mathcal{I}^{-w_0(\lambda)}_*.$$
		Since $\ell(t_\gamma t_{\nu+\mu+\lambda})=\ell(t_\gamma)+\ell(t_{\mu+\lambda})$, we can apply  Lemma \ref{lem conv standard} to get a canonical isomorphism
		$$\mathcal{N}_\gamma\star^I\nabla_{\nu+\mu+\lambda}\simeq\nabla_{\gamma+\nu+\mu+\lambda}.$$
		Thus, we have a canonical isomorphism 
		$$\mathcal{N}_\gamma\star^I(\nabla_{\nu+\mu+\lambda}\star \mathcal{I}^{-w_0(\lambda)}_*)\simeq \nabla_{\gamma+\nu+\mu+\lambda}\star \mathcal{I}^{-w_0(\lambda)}_*$$ whenever $\gamma$ and $\nu+\mu+\lambda$ are dominant. Therefore we obtain a canonical isomorphism
		$$\mathcal{N}_\gamma\star^I(\widehat{\mathcal{Z}}'_\nu)_\mu\simeq \varinjlim_{\lambda} \nabla_{\gamma+\nu+\mu+\lambda}\star \mathcal{I}^{-w_0(\lambda)}_*, $$
		where $\lambda$ runs though the set of elements of $\mathbf{Y}^+$ such that $\gamma+\nu+\mu+\lambda$ is dominant. The fact that the transition morphisms in the above limit coincide with the transition morphisms appearing in the definition of $(\widehat{\mathcal{Z}}'_{\gamma+\nu})_\mu$ follows from the construction of the morphism \eqref{eq 513}.
	\end{proof}
 Similarly, one can can check the following result.
 \begin{prop}\label{prop convolution costandard 2}
     For any $w,w'\in W$ such that $\ell(ww'w_0)=\ell(w)+\ell(w'w_0)$ and $\lambda\in \mathbf{Y}$, we have a canonical isomorphism
     $$ \mathcal{N}_{w}\star^I\widehat{\mathcal{Z}}'_{w't_\lambda}\simeq \widehat{\mathcal{Z}}'_{ww't_\lambda}. $$
 \end{prop}
	\begin{prop}{{\cite[Corollary 6.9]{achar2022geometric}}}\label{prop head}
		For any $w\in W_{\mathrm{ext}}$ the object $\widehat{\mathcal{Z}}'_w$ has finite length and a simple socle, isomorphic to $\widehat{\mathcal{L}}_w$. In particular, $\widehat{\mathcal{Z}}'_w$ is indecomposable.
	\end{prop}
	We are thus allowed to define the analogue of the baby Verma module
	
	$$\widehat{\mathcal{Z}}_w:=\mathbb{D}(\widehat{\mathcal{Z}}'_w) \in \mathrm{mod}^\mathbf{Y}_{I_\mathrm{u}}(\mathcal{R})$$
	for all $w\in W_{\mathrm{ext}}$. Since $\mathbb{D}$ is an (exact) equivalence of categories sending $\widehat{\mathcal{L}}_w$ to $\widehat{\mathcal{L}}_w$, we deduce from Proposition \ref{prop head} that $\widehat{\mathcal{Z}}_w$ is indecomposable and has a simple head, isomorphic to $\widehat{\mathcal{L}}_w$.
	\begin{prop}{{\cite[Proposition 9.25]{achar2022geometric}}}\label{prop ext ind objects}
		For $w,w'\in W_{\mathrm{ext}}$ we have
		$$\mathrm{Ext}^i_{\mathrm{mod}^\mathbf{Y}_{I_\mathrm{u}}(\mathcal{R})}(\widehat{\mathcal{Z}}_w,\widehat{\mathcal{Z}}'_{w'})=\begin{cases} \mathbb{k}\quad\text{if}~i=0~\text{and}~w=w',\\0\quad\text{otherwise}.
		\end{cases}$$
	\end{prop}
 \subsection{Averaging functor}
 We denote by $\mathrm{Mod}^\mathbf{Y}_{\mathcal{IW}}(\mathcal{R})$ the version of $\mathrm{Mod}^\mathbf{Y}_{I_\mathrm{u}}(\mathcal{R})$ where the category $\mathrm{Perv}_{I_\mathrm{u}}(\mathrm{Gr})$ is replaced by $\mathrm{Perv}_{\mathcal{IW}}(\mathrm{Gr})$. Namely, an object $\mathcal{F}$ of $\mathrm{Mod}^\mathbf{Y}_{\mathcal{IW}}(\mathcal{R})$ is a $\mathbf{Y}$-graded ind-object in $\mathrm{Perv}_{\mathcal{IW}}(\mathrm{Gr})$, endowed with a right-action of $\mathcal{R}$ respecting the grading on $\mathcal{R}$ and $\mathcal{F}$. There is an exact and faithful functor 
 $$\Phi^\mathcal{IW}:\mathrm{Perv}_{\mathcal{IW}}(\mathrm{Gr})\to \mathrm{Mod}^\mathbf{Y}_{\mathcal{IW}}(\mathcal{R}),~\mathcal{F}\mapsto\bigoplus_{\mu\in\mathbf{Y}}\mathcal{F}\star\mathcal{R}_\mu. $$
The above functor satisfies the same type of properties as $\Phi$. Namely, for any $w\in {^SW_\mathrm{ext}}$, written as $w=yt_\lambda$ for some $(y,\lambda)\in {^SW_\mathrm{ext}^\mathrm{res}}\times\mathbf{Y}$, we define 
$$\widehat{\mathcal{L}}_w^{\mathcal{IW}}:=\Phi^\mathcal{IW}(\mathrm{L}^\mathcal{IW}_y)\langle-\lambda\rangle\in \mathrm{Mod}^\mathbf{Y}_{\mathcal{IW}}(\mathcal{R}),$$
where $y$ is seen as an element of $\mathbf{Y}^{++}$ through the bijection of Corollary \ref{coro domi}. Then the above object is simple, and the assignment  $w\mapsto \widehat{\mathcal{L}}_{w}^\mathcal{IW}$ induces a bijection between $^SW_{\mathrm{ext}}$ and the set of isomorphism classes of objects of $\mathrm{Mod}^\mathbf{Y}_{\mathcal{IW}}(\mathcal{R})$. We denote by $\mathrm{mod}^\mathbf{Y}_{\mathcal{IW}}(\mathcal{R})$ the subcategory of objects of finite length. As in §\ref{section Baby Verma and co-Verma modules}, one can define the baby co-Verma module $\widehat{\mathcal{Z}}'^{\mathcal{IW}}_w$ attached to an element $w\in {^SW_{\mathrm{ext}}}$, and show that this object admits a unique simple socle, isomorphic to $\widehat{\mathcal{L}}_{w}^\mathcal{IW}$. Using the geometric Satake equivalence together with Theorem \ref{thm IW equivalence} (and the identifications obtained in §\ref{section Representations of Frobenius kernels}), one can check that the category $\mathrm{Mod}^\mathbf{Y}_{\mathcal{IW}}(\mathcal{R})$ identifies with the category of $\mathbf{Y}$-graded ind-objects of $\mathrm{Rep}(\check{\mathbf{G}}^{(1)})$ endowed with a graded action of $\mathcal{O}(\check{\mathbf{G}}^{(1)})$. We therefore obtain the following result.
\begin{prop}{{\cite[Proposition 5.7]{achar2022geometric}}}\label{prop semi-simple IW}
    The category $\mathrm{Mod}^\mathbf{Y}_{\mathcal{IW}}(\mathcal{R})$ (resp. $\mathrm{mod}^\mathbf{Y}_{\mathcal{IW}}(\mathcal{R})$) is equivalent to the category of algebraic $T^\vee_\mathbb{k}$-modules (resp. finite dimensional algebraic $T^\vee_\mathbb{k}$-modules). In particular, the category $\mathrm{mod}^\mathbf{Y}_{\mathcal{IW}}(\mathcal{R})$ is semi-simple, and we have isomorphisms
    $$\widehat{\mathcal{Z}}'^{\mathcal{IW}}_w\simeq  \widehat{\mathcal{L}}_w^{\mathcal{IW}}\quad\forall w\in {^SW_{\mathrm{ext}}}.$$
\end{prop}
Recall the averaging functor
$$\mathrm{Av}_{I_\mathrm{u},\psi}:\mathrm{Perv}_{I_\mathrm{u}}(\mathrm{Gr})\to \mathrm{Perv}_{\mathcal{IW}}(\mathrm{Gr}). $$
 This functor induces an exact functor between the abelian categories
$$\mathrm{Av}_{I_\mathrm{u},\psi}:\mathrm{mod}^\mathbf{Y}_{I_\mathrm{u}}(\mathcal{R})\to \mathrm{mod}^\mathbf{Y}_{\mathcal{IW}}(\mathcal{R}), $$
and the latter functor commutes with $\Phi$, in the sense that we have a functorial isomorphism
$$\Phi^\mathcal{IW}(\mathrm{Av}_{I_\mathrm{u},\psi}(\mathcal{F}))\simeq \mathrm{Av}_{I_\mathrm{u},\psi}(\Phi(\mathcal{F})) $$
for any $\mathcal{F}\in\mathrm{mod}^\mathbf{Y}_{I_\mathrm{u}}(\mathcal{R})$. We now describe the behaviour of the averaging functor with respect to costandard (and therefore also standard) objects.
\begin{prop}{{\cite[Lemma 6.3]{achar2022geometric}}}\label{Prop averaging simples}
   Let $w\in W_\mathrm{ext}$, and write it as $w=vu$, with $v\in W$, $u\in {^SW_{\mathrm{ext}}}$. We have
   $$\mathrm{Av}_{I_\mathrm{u},\psi}(\widehat{\mathcal{Z}}'_{w})\simeq\widehat{\mathcal{Z}}'^{\mathcal{IW}}_{u}.$$
\end{prop}
	\subsection{Representations of the Frobenius kernel in terms of $\mathcal{R}$-modules}\label{section Representations of the Frobenius kernel in terms of}
 The following statement is a consequence of the main results of \cite{achar2022geometric}, togeher with the proof of the Finkelberg-Mirkovi\'c-conjecture in \cite{bezrukavnikov2024modular}.
 \begin{thm}{{\cite[Proposition 1.3]{achar2022geometric}}}\label{thm BR}
     Assume that, for any indecomposable constituent in the root system of $(\check{\mathbf{G}},\check{\mathbf{T}})$, $\ell$ is greater than the bound in Figure \ref{figure 1}. Then there exists an equivalence of categories
     $$\mathrm{FM}_{\mathrm{Frob}}:\mathrm{mod}^\mathbf{Y}_{I_\mathrm{u}}(\mathcal{R})\xrightarrow{\sim} \mathrm{Rep}_{[0]}(\check{\mathbf{G}}_1\check{\mathbf{T}})$$
     such that 
     $$\mathrm{FM}_{\mathrm{Frob}}\circ\Phi\simeq\mathrm{FM}\circ\mathrm{For}_0$$
     which satisfies
     $$\mathrm{FM}_{\mathrm{Frob}}(\widehat{\mathcal{L}}_w)\simeq \widehat{\mathrm{L}}(w^{-1}\cdot_\ell0),~\mathrm{FM}_{\mathrm{Frob}}(\widehat{\mathcal{Z}}'_w)\simeq \widehat{\mathrm{Z}}'(w^{-1}\cdot_\ell0),~~\mathrm{FM}_{\mathrm{Frob}}(\widehat{\mathcal{Z}}_w)\simeq \widehat{\mathrm{Z}}(w^{-1}\cdot_\ell0) $$
     for any $w\in W_\mathrm{ext}$, and such that
     $$\mathrm{FM}_{\mathrm{Frob}}(\mathcal{F}\star\mathcal{G})\simeq\mathrm{FM}_{\mathrm{Frob}}(\mathcal{F})\otimes\mathrm{Fr}^*(\mathrm{Sat}(\mathrm{sw}^*\mathcal{G})) $$
     functorially in $\mathcal{F}\in \mathrm{mod}^\mathbf{Y}_{I_\mathrm{u}}(\mathcal{R})$, $\mathcal{G}\in\mathrm{Perv}_{G[[t]]}(\mathrm{Gr})$.
 \end{thm}

 The next result follows from the constructions of baby Verma and co-Verma modules in terms of coinduced and induced representations respectively.
 \begin{prop}
     There is an isomorphism of $\check{\mathbf{G}}_1\check{\mathbf{T}}$-modules
     $$\widehat{\mathrm{Z}}(0)\simeq {^{w_0}\widehat{\mathrm{Z}}'((w_0t_{-2\zeta})^{-1}\cdot_\ell0)}. $$
 \end{prop}
 \begin{proof}
     We want to establish an isomorphism $\widehat{\mathrm{Z}}(0)\simeq {^{w_0}\widehat{\mathrm{Z}}'(2(\ell-1)\zeta}) $, or equivalently $^{w_0}\widehat{\mathrm{Z}}(0)\simeq \widehat{\mathrm{Z}}'(2(\ell-1)\zeta) $. By \cite[Lemma 9.2, Part II]{jantzen2003representations}, we have an isomorphism
     $$\widehat{\mathrm{Z}}(0)\simeq \mathrm{ind}^{\check{\mathbf{G}}_1\check{\mathbf{T}}}_{\check{\mathbf{B}}_1^+\check{\mathbf{T}}}(-2(\ell-1)\zeta). $$
     We deduce that
     $$^{w_0}\widehat{\mathrm{Z}}(0)\simeq \mathrm{ind}^{\check{\mathbf{G}}_1\check{\mathbf{T}}}_{w_0(\check{\mathbf{B}}_1^+\check{\mathbf{T}})}(w_0(-2(\ell-1)\zeta))=\mathrm{ind}^{\check{\mathbf{G}}_1\check{\mathbf{T}}}_{\check{\mathbf{B}}_1\check{\mathbf{T}}}(2(\ell-1)\zeta)=\widehat{\mathrm{Z}}'(2(\ell-1)\zeta), $$
     where the first isomorphism is due to \cite[(4)-I.3.5]{jantzen2003representations}.
 \end{proof}
\begin{rem}
     In fact, on can easily deduce from the previous proposition how to express any baby Verma module as a twisted baby co-Verma module (in the same fashion as \cite[Corollary 3.2.11]{ABBGM}), but only the above case will be useful for us in the sequel. The point is that it allows one to express baby Verma modules as ind-objects of the category $\mathrm{Perv}_{I_\mathrm{u}}(\mathrm{Gr})$.
\end{rem}

Let $\mathcal{F}\in \mathrm{Mod}^\mathbf{Y}_{I_\mathrm{u}}(\mathcal{R})$, and $w\in W$. We define the object $^w\mathcal{F}\in \mathrm{Mod}^\mathbf{Y}_{I_\mathrm{u}}(\mathcal{R})$ by putting 
$$(^w\mathcal{F})_\lambda:=\mathcal{F}_{w(\lambda)}\quad\forall \lambda\in\mathbf{Y}$$
and defining the right-$\mathcal{R}$ action by twisting by $w$ the corresponding action on $\mathcal{F}$ (the twist of $\mathcal{R}$ by $w$ can be defined via the geometric Satake equivalence). 
\begin{coro}\label{coro ind verma}
    Assume that, for any indecomposable constituent in the root system of $(\check{\mathbf{G}},\check{\mathbf{T}})$, $\ell$ is greater than the bound in Figure \ref{figure 1}. We have an isomorphism
   $$\widehat{\mathcal{Z}}_0\simeq {^{w_0}(\widehat{\mathcal{Z}}'_{w_0t_{-2\zeta}})}.$$ 
\end{coro}
\begin{proof}
   Thanks to the assumptions, we may use Theorem \ref{thm BR} to prove this isomorphism. We have isomorphisms
   $$\mathrm{FM}_{\mathrm{Frob}}({^{w_0}(\widehat{\mathcal{Z}}'_{w_0t_{-2\zeta}})})\simeq {^{w_0}\mathrm{FM}_{\mathrm{Frob}}({(\widehat{\mathcal{Z}}'_{w_0t_{-2\zeta}})})}\simeq {^{w_0}\widehat{\mathrm{Z}}'((w_0t_{-2\zeta})^{-1}\cdot_\ell0)}$$
thanks to Proposition \ref{prop twist coh} and Theorem \ref{thm BR}. We can then conclude thanks to Corollary \ref{coro ind verma}.
\end{proof}

	\section{Drinfeld's compactifications}\label{section Drinfeld's compactifications}
	\subsection{Drinfeld's compactification and variants}\label{section Drinfeld's compactification and variants}
	We put $X:=\mathbb{P}^1$. Thanks to \cite[§4.1.1]{ABBGM}, there exists a reductive group $\tilde{G}$ over $\mathbb{F}$ whose derived subgroup $[\tilde{G},\tilde{G}]$ is simply connected, endowed with a central isogeny $\varphi:\tilde{G}\to G$ such that $\mathrm{ker}(\varphi)$ is connected. We fix such a pair $(\tilde{G},\varphi)$ and consider the maximal torus and Borel subgroup $ \tilde{T}:=\varphi^{-1}(T)$, $\tilde{B}:=\varphi^{-1}(B)$. We also let $\tilde{N}^-$ be the unipotent radical of the Borel subgroup which is opposite to $\tilde{B}$ with respect to $\tilde{T}$. Notice that, if we denote by $\widetilde{\mathbf{X}}^+\subset\widetilde{\mathbf{X}}$ the set of dominant characters and characters associated with the pair $(\tilde{T},\tilde{B})$, then $\varphi$ induces a morphism $\mathbf{X}\to \widetilde{\mathbf{X}}$ that sends $\mathbf{X}^+$ to $\widetilde{\mathbf{X}}^+$.
	
	We have a diagram
	\begin{equation*}
		\xymatrix{
			\mathrm{Bun}_B\ar[r]^{\mathfrak{p}^+}\ar[d]_{\mathfrak{q}^+} &\mathrm{Bun}_T  \\
			\mathrm{Bun}_G
		}
	\end{equation*}
	defined by sending a $B$-bundle $\mathcal{F}_B$ to $\mathfrak{q}^+(\mathcal{F}_B):=\mathcal{F}_B\times^{B}G$ (resp. $\mathfrak{p}^+(\mathcal{F}_B):=\mathcal{F}_B\times^{B}T$). There is an important description of the stack $\mathrm{Bun}_B$, called the Pl\"ucker description (cf. \cite[§1.2.1]{BG}). Namely, when the derived subgroup $[G,G]$ is simply connected, $\mathrm{Bun}_B$ classifies the data of a principal $G$-bundle $\mathcal{F}_G$ over $X\times S$, endowed with a $B$-reduction, i.e. a principal $T$-bundle $\mathcal{F}_T\in\mathrm{Bun}_T(S)$ together with a collection of morphisms of vector bundles
	$$\kappa^\lambda:\mathcal{L}_{\mathcal{F}_T}^\lambda\to \mathcal{V}_{\mathcal{F}_G}^\lambda\qquad\forall \lambda\in \mathbf{X}^+,$$
	which are required to satisfy the Pl\"ucker relations. This means that, for $\lambda_1,\lambda_2\in\mathbf{X}^+$, the morphism
	$$\mathcal{L}_{\mathcal{F}_T}^{\lambda_1+\lambda_2}\xrightarrow{\kappa^{\lambda_1+\lambda_2}} \mathcal{V}_{\mathcal{F}_G}^{\lambda_1+\lambda_2}\to \mathcal{V}_{\mathcal{F}_G}^{\lambda_1}\otimes \mathcal{V}_{\mathcal{F}_G}^{\lambda_2} $$
	(where the second arrow is induced by \eqref{canonical morphism}) must coincide with the composition
	$$\mathcal{L}_{\mathcal{F}_T}^{\lambda_1+\lambda_2}\simeq \mathcal{L}_{\mathcal{F}_T}^{\lambda_1}\otimes \mathcal{L}_{\mathcal{F}_T}^{\lambda_2}\xrightarrow{\kappa^{\lambda_1}\otimes \kappa^{\lambda_2}}  \mathcal{V}_{\mathcal{F}_G}^{\lambda_1}\otimes \mathcal{V}_{\mathcal{F}_G}^{\lambda_2}.  $$
	One can then define the Drinfeld compactification $\overline{\mathrm{Bun}}_B$. Since we want to work with a general group $G$, we first need to define a larger space $\overline{\mathrm{Bun}}'_B(G)$. It is the functor classifying the data of a principal $G$-bundle $\mathcal{F}_G$, endowed with a \textit{degenerate} $B$-reduction, i.e. a principal $T$-bundle $\mathcal{F}_T$ together with a collection of morphisms 
	$$\kappa^\lambda:\mathcal{L}_{\mathcal{F}_T}^\lambda\to \mathcal{V}_{\mathcal{F}_G}^\lambda\qquad\forall \lambda\in \mathbf{X}^+,$$
	satisfying the Pl\"ucker relations, but which are no longer required to be morphisms of vector bundles. We merely ask that, for every geometric point $s$ of $S$, the restriction $\kappa^\lambda|_{X\times s}$ is injective.  When $[G,G]$ is simply connected, $\overline{\mathrm{Bun}}'_B(G)$ is the sought for Drinfeld compactification. But in general, $\overline{\mathrm{Bun}}_B$ will be a closed substack of $\overline{\mathrm{Bun}}'_B(G)$.
	
	As in the previous paragraph, one can define the space $\overline{\mathrm{Bun}}_{\tilde{B}}'(\tilde{G})$ attached to $\tilde{G}$. There is a morphism of stacks  $\overline{\mathrm{Bun}}_{\tilde{B}}'(\tilde{G})\to\overline{\mathrm{Bun}}'_B(G)$, obtained by inducing the bundles via the isogeny $\tilde{G}\to G$ (the Pl\"ucker relations for $\tilde{G}$ induce Pl\"ucker relations for $G$ since $\mathbf{X}^+\subset\widetilde{\mathbf{X}}^+$). The same arguments as in the proof of \cite[Lemma 4.1.2]{ABBGM} can be used to prove the next result.
	\begin{prop}\label{prop closed immersion}
		The morphism  of stacks $\overline{\mathrm{Bun}}_{\tilde{B}}'(\tilde{G})\to\overline{\mathrm{Bun}}'_B(G)$ is a closed embedding, and its image does not depend on the choice of $(\tilde{G},\varphi)$.
	\end{prop}
	
	We define $\overline{\mathrm{Bun}}_B$ as the image of the above morphism inside $\overline{\mathrm{Bun}}'_B(G)$. One can show that $\overline{\mathrm{Bun}}_B$ is an algebraic stack in which $\mathrm{Bun}_B$ is a dense open substack, and that the canonical projection $\overline{\mathfrak{p}^+}:\overline{\mathrm{Bun}}_B\to \mathrm{Bun}_G$ is proper (\cite[Proposition 1.2.2]{BG}). Notice that the isomorphism $B\simeq B^-$ (obtained by conjugation by $w_0$) yields an isomorphism of stacks between $\mathrm{Bun}_B$ and $\mathrm{Bun}_{B^-}$ (resp. $\overline{\mathrm{Bun}}_B$ and $\overline{\mathrm{Bun}}_{B^-}$). We denote by $\mathfrak{q}$ (resp. $\overline{\mathfrak{q}}$) the canonical projection $\mathrm{Bun}_{B^-}\to \mathrm{Bun}_{G}$ (resp. $\overline{\mathrm{Bun}}_{B^-}\to \mathrm{Bun}_{G}$), and by $\mathfrak{p}$  (resp. $\overline{\mathfrak{p}}$) the canonical projection $\mathrm{Bun}_{B^-}\to \mathrm{Bun}_{T}$ (resp. $\overline{\mathrm{Bun}}_{B^-}\to \mathrm{Bun}_{T}$).
	
	One can easily check that the connected components of $|\mathrm{Bun}_T|$ are parametrized by $\mathbf{Y}$: the connected component associated with $\mu$ consists in the $T$-bundles $\mathcal{F}_T$ such that the line bundle $\mathcal{L}_{\mathcal{F}_T}^\lambda$ has degree $-\langle\lambda,\mu\rangle$ for all $\lambda\in\mathbf{X}^+$. Likewise, we have a decomposition into connected components
 $$\mathrm{Bun}_B=\bigsqcup_{\mu\in\mathbf{Y}}\mathrm{Bun}_B^\mu, $$
 where $\mathrm{Bun}_B^\mu$ is the preimage through $\mathfrak{p}^+$ of the connected component of $\mathrm{Bun}_T$ associated with $\mu$. We will denote by $\overline{\mathrm{Bun}}_{B^-}^\theta$ the preimage of the connected component associated with $\theta\in \mathbf{Y}$ through $\overline{\mathfrak{p}}$.
	
	Next, we define the stack $\overline{\mathrm{Bun}}_{N^-}$, which is a ``compactification" of the canonical morphism $\mathrm{Bun}_{N^-}\to \mathrm{Bun}_{G}$ (more precisely, there is a canonical action of $T$ on $\overline{\mathrm{Bun}}_{N^-}$ such that the induced morphism $\overline{\mathrm{Bun}}_{N^-}/T\to\mathrm{Bun}_G$ is representatble and proper, cf. \cite[Lemma 2.2.4]{FGV}). By definition, $\overline{\mathrm{Bun}}_{N^-}$ sits in the Cartesian diagram 
 \begin{equation*}
			\xymatrix{
				\overline{\mathrm{Bun}}_{N^-}\ar[d] \ar[r]&\overline{\mathrm{Bun}}_{B^-}\ar[d]^{\overline{\mathfrak{q}}}\\
				pt\ar[r]&\mathrm{Bun}_{T},}
		\end{equation*}
  where the bottom horizontal arrow is the inclusion of the trivial bundle. Let us spell out this definition. As before, we first need to define a larger space $\overline{\mathrm{Bun}}_{N^-}'(G)$. It classifies the data of a principal $G$-bundle $\mathcal{F}_G\in\mathrm{Bun}_G(S)$, together with a collection of morphisms of sheaves
	$$\kappa^-:=(\kappa^{\lambda,-}:\mathcal{V}_{\mathcal{F}_G}^\lambda\to \mathcal{O}_{X\times S})_{\lambda\in \mathbf{X}^+}$$
	satisfying the Pl\"ucker relations, and such that the dual of $\kappa^{\lambda,-}|_{X\times s}$ is injective for every geometric point $s$ (for technical reasons which will appear in §\ref{section Zastava spaces}, we consider here dual morphisms in the definition). One can also define the analogue stack $\overline{\mathrm{Bun}}_{{\tilde{N}^-}}'(\tilde{G})$ associated with $\tilde{G}$, which is an algebraic stack locally of finite type over $\mathbb{F}$. Consider the natural morphism of stacks $\overline{\mathrm{Bun}}_{\tilde{N}^-}'(\tilde{G})\to\overline{\mathrm{Bun}}_{N^-}'(G)$, which is a closed embedding (\cite[Lemma 4.1.2]{ABBGM}), and define $\overline{\mathrm{Bun}}_{N^-}$ as the image of the latter morphism.

	Now, we introduce some variations of $\overline{\mathrm{Bun}}_{N^-}$ (as found in \cite[§4.1]{ABBGM}). Recall that we have fixed a closed point $x\in X$. For any $\nu\in \mathbf{Y}$, the stack $_{\leq\nu}\overline{\mathrm{Bun}}_{N^-}'(G)$ (respectively $_{\nu}\overline{\mathrm{Bun}}_{N^-}'$) classifies the data of a principal $G$-bundle $\mathcal{F}_G$, together with a collection of morphisms of sheaves
	$$\kappa_\nu^-:=(\kappa^{\lambda,-}_\nu:\mathcal{V}_{\mathcal{F}_G}^\lambda\to \mathcal{O}_{X\times S}((x\times S)\cdot \langle\nu,\lambda\rangle))_{\lambda\in \mathbf{X}^+} $$
	satisfying the Pl\"ucker relations (resp. and such that the divisor of zeros of $\kappa^{\lambda,-}_\nu$ does not contain $x$ for all $\lambda$). The stack $_{\nu}\overline{\mathrm{Bun}}_{N^-}'$ is thus open in $_{\leq \nu}\overline{\mathrm{Bun}}_{N^-}'(G)$. We also define $_{\nu}\mathrm{Bun}_{N^-}'$ to be the smooth open substack of $_{\nu}\overline{\mathrm{Bun}}_{N^-}'$, where the morphisms of sheaves are required to be morphisms of vector bundles. Next, choose a cocharacter $\nu'\in X_*(\tilde{T})$ which is sent to $\nu$ via the isogeny $\varphi$. Then one can define the version $_{\leq\nu'}\overline{\mathrm{Bun}}_{\tilde{N}^-}'(\tilde{G})$ attached to $\tilde{G}$ and $\nu'$, and the same arguments as the ones used to prove Proposition \ref{prop closed immersion} show that the natural morphism $_{\leq\nu'}\overline{\mathrm{Bun}}_{\tilde{N}^-}'(\tilde{G})\to {_{\leq\nu}\overline{\mathrm{Bun}}_{N^-}'}(G)$ is a closed embedding. Moreover, its image is easily seen not to depend on the choice of $\nu'$. So we let $_{\leq\nu}\overline{\mathrm{Bun}}_{N^-}$ be the image of the latter morphism, and put 
	$$_{\nu}\overline{\mathrm{Bun}}_{N^-}:={_{\leq\nu}\overline{\mathrm{Bun}}_{N^-}}(G)\cap {_{\nu}\overline{\mathrm{Bun}}_{N^-}'},\quad _{\nu}\mathrm{Bun}_{N^-}:={_{\leq\nu}\overline{\mathrm{Bun}}_{N^-}}(G)\cap{_{\nu}\mathrm{Bun}_{N^-}'}.$$
By construction we have 
\begin{equation}\label{eq strat}
    _{\nu}\overline{\mathrm{Bun}}_{N^-}={_{\leq\nu}\overline{\mathrm{Bun}}_{N^-}}-\bigcup_{\nu'<\nu}{_{\leq\nu'}\overline{\mathrm{Bun}}_{N^-}}.
\end{equation}
Note that all the previous stacks admit a natural action of $T$: an element $t$ multiplies each $\kappa_\nu^{\lambda,-}$ by $\lambda(t)$, for all $\nu\in\mathbf{Y}$, $\lambda\in\mathbf{X}^+$.

  \begin{prop}\label{prop open substack}
     For any $\nu\in\mathbf{Y}$, there is a canonical morphism ${_{\leq\nu}\overline{\mathrm{Bun}}_{N^-}}\to \overline{\mathrm{Bun}}_{B^-}$ which is smooth.
     \end{prop}
     \begin{proof}
 By construction it is equivalent to construct a canonical morphism $_{\leq\nu}\overline{\mathrm{Bun}}_{\tilde{N}^-}'(\tilde{G})\to \overline{\mathrm{Bun}}_{\tilde{B}^-}'(\tilde{G})$ and to check that it is smooth. First recall that there is an isomorphism of stacks 
     $$\mathrm{Bun}_{\tilde{T}}\simeq (pt/\tilde{T})\times \mathrm{Pic}(\mathbb{P}^1)^r, $$
  where $r:=\mathrm{rank}(\tilde{T})$, $pt/\tilde{T}$ is the stack of principal $\tilde{T}$-bundles, and $\mathrm{Pic}(\mathbb{P}^1)$ is the Picard scheme of $\mathbb{P}^1$. Now, we have a Cartesian square (see \cite[(2.4)]{FGV})
     \begin{equation*}
			\xymatrix{
				_{\leq\nu}\overline{\mathrm{Bun}}_{\tilde{N}^-}'(\tilde{G})\ar[d] \ar[r]&\overline{\mathrm{Bun}}_{\tilde{B}^-}'(\tilde{G})\ar[d]^{\overline{\mathfrak{q}}}\\
				pt\ar[r]&\mathrm{Bun}_{\tilde{T}},}
		\end{equation*}
  where the bottom arrow is determined by the maps 
  $$pt\to pt/\tilde{T},\quad pt\to\{\nu\}\subset\mathbf{Y}\simeq \mathrm{Pic}(\mathbb{P}^1)^r.$$ 
  But $pt\to pt/\tilde{T}$ is smooth and $\mathrm{Pic}(\mathbb{P}^1)\simeq\mathbb{Z}$ is discrete, so the bottom arrow is smooth.
 \end{proof}
 We will also need the following fact in the sequel, which follows from  \cite[Corollary 3.2.4]{FGV} and says that $_{\nu}\mathrm{Bun}_{N^-}$ is topologically ``contractible".
 \begin{prop}\label{prop contractible}
     For any $\nu\in\mathbf{Y}$, the cohomology group
     $$H^k(_{\nu}\mathrm{Bun}_{N^-},\underline{\mathbb{k}}_{_{\nu}\mathrm{Bun}_{N^-}}) $$
     vanishes when $k\neq0$, and is isomorphic to $\mathbb{k}$ when $k=0$.
 \end{prop}
 \begin{proof}
     By \cite[Corollary 3.2.4]{FGV}, there exists a unipotent group $N'$ and a $N'$-bundle $\mathcal{N}$ over the stack $_{\nu}\mathrm{Bun}_{N^-}$. Since $N'$ is unipotent, the forgetful functor
     $$\mathrm{D}({_{\nu}\mathrm{Bun}_{N^-}})\simeq \mathrm{D}_{N'}(\mathcal{N})\to  \mathrm{D}(\mathcal{N})$$
     is fully faithful. Therefore we have an isomorphism
     $$H^k({_{\nu}\mathrm{Bun}_{N^-}})\simeq H^k(\mathcal{N})$$
     for all $k$. By \textit{loc. cit.}, the stack $\mathcal{N}$ is isomorphic to a tower of affine spaces. This concludes the proof.
 \end{proof}
	For $\nu_1\leq \nu_2\in \mathbf{Y}$, we have a natural closed embedding 
	$$_{\leq\nu_1}\overline{\mathrm{Bun}}_{N^-}'(G)\hookrightarrow {_{\leq\nu_2}\overline{\mathrm{Bun}}_{N^-}'}(G),$$
	which induces a closed embedding 
	$$_{\leq\nu_1}\overline{\mathrm{Bun}}_{N^-}\hookrightarrow {_{\leq\nu_2}\overline{\mathrm{Bun}}_{N^-}}.$$
	We define ${_{\infty}\overline{\mathrm{Bun}}_{N^-}}$ as the direct limit (in the category of ind-algebraic stacks over $\mathbb{F}$) 
	$$\varinjlim_{\nu\in \mathbf{Y}} \left({_{\leq \nu}\overline{\mathrm{Bun}}_{N^-}}\right).$$
	We denote by $\overline{i}_{\leq\nu}$, resp. $\overline{i}_{\nu}$, resp. $i_\nu$, the embedding of $_{\leq\nu}\overline{\mathrm{Bun}}_{N^-}$, resp. $_{\nu}\overline{\mathrm{Bun}}_{N^-}$, resp. $_{\nu}\mathrm{Bun}_{N^-}$, into ${_{\infty}\overline{\mathrm{Bun}}_{N^-}}$.
	
	We will denote by ${^1\mathrm{Bun}}_G$ the stack sending a test scheme $S$ to the groupoid consisting in the data of a principal $G$-bundle $\mathcal{F}_G$ over $X\times S$, together with a structure of level $1$ at $x$ for $\mathcal{F}_G$. Recall that a structure of level one at $x$ is an isomorphism of sheaves $\mathcal{F}_G|_{x\times S}\simeq \mathcal{F}^0_G|_{x\times S}$. We will denote by ${^1_{\infty}\overline{\mathrm{Bun}}_{N^-}}$ the fibered product $_{\infty}\overline{\mathrm{Bun}}_{N^-}\times_{\mathrm{Bun}_G} {^1\mathrm{Bun}}_G$. Similarly, we will let ${^1_{\leq\nu}\overline{\mathrm{Bun}}_{N^-}},{^1_{\nu}\overline{\mathrm{Bun}}_{N^-}},{^1_{\nu}\mathrm{Bun}}_{N^-}$ be the corresponding stacks obtained by base change, and still denote by $\overline{i}_{\leq\nu},\overline{i}_{\nu},i_\nu$ the inclusion of these stacks in ${^1_{\infty}\overline{\mathrm{Bun}}_{N^-}}$. In the sequel, we will consider the diagonal $T$-action on these stacks, induced by the previously defined $T$-action together with the action induced by modifying the structure of level $1$.  
 \begin{rem}
     In \cite{ABBGM}, there are versions of the above stacks which are labelled by a non-negative integer $k$, to indicate a structure of level $k$ at $x$ for $\mathcal{F}_G$. Here, we only treat the cases where $k\in\{0,1\}$. In the sequel we will sometimes adopt the notation ${^0_{\infty}\overline{\mathrm{Bun}}_{N^-}}:={_{\infty}\overline{\mathrm{Bun}}_{N^-}}$.

 \end{rem}
	\subsection{Colored effective divisors}\label{section colored divisor} Given $\mu=\sum_{i\in \check{\mathfrak{R}}_s}n_ii\in\mathbf{Y}^{pos}$, we consider as in \cite{finkelberg1997semiinfinite} the space $X^\mu:=\prod_{i\in \check{\mathfrak{R}}_s}X^{(n_i)}$ of colored effective divisors of multidegree $\mu$, where $X^{(n_i)}$ denotes the $n_i$-th symmetric power of $X$. We will write an element $D=(D_i)_{i\in \check{\mathfrak{R}}_s}\in X^\mu$ as $D=\sum_{i\in \check{\mathfrak{R}}_s}i\cdot D_i$ and, for any $\lambda\in\Lambda$, define the divisor $\langle\lambda,D\rangle:=\sum_{i\in \check{\mathfrak{R}}_s}\langle\lambda,i\rangle D_i$. A partition of $\mu$ is a decomposition $\mu=\gamma_1+\cdots+\gamma_m$ into a sum of non-zero elements $\gamma_1,\cdots,\gamma_m$ of $\mathbf{Y}^{pos}$, which can be seen as an element $\{\{\gamma_1,\cdots,\gamma_m\}\}$ of the $m$-th symmetric power of $\mathbf{Y}^{pos}$. We will denote by $\mathfrak{A}(\mu)$ the set of partitions of $\mu$.
	
	For each $\Gamma=\{\{\gamma_1,\cdots,\gamma_m\}\}\in \mathfrak{A}(\mu)$, we will define the locally closed stratum $X^\mu_\Gamma\subset X^\mu$. First, we define $'X^\mu_\Gamma$ as the subset of $X^{\gamma_1}\times \cdots\times X^{\gamma_m}$ consisting of tuples $(x_1,\cdots,x_m)$ such that
	\begin{enumerate}
		\item for each $k$, all the components of $x_k$ in $X^{\gamma_k}$ are equal, meaning that if one writes $\gamma_k=\sum_{i\in \check{\mathfrak{R}}_s}n'_ii$, $x_k=(x_k^j)_{j\in I}$, then there exists some $y_k\in X$ such that $$x_k^j=(\underbrace{y_k,\cdots,y_k}_{n'_i})\in X^{(n'_i)}~\forall i\in \check{\mathfrak{R}}_s.$$
		\item we have $y_i\neq y_j$ for $i\neq j$.
	\end{enumerate}
	The locally closed subvariety $X^\mu_\Gamma\subset X^\mu$ is then defined as the image of the canonical morphism ${'X^\mu_\Gamma}\to X^\mu$, obtained by quotienting via the action of the symmetric group. We thus obtain a stratification $X^\mu=\sqcup_{\Gamma\in\mathfrak{A}(\mu)}X^\mu_\Gamma$.
	
	We also need to introduce the moduli space description of $X^\mu$ (cf. \cite[§2.1]{BBGM}). By definition, $X^\mu$ classifies the data of a principal $T$-bundle $\mathcal{F}_T$, together with a collection of non-zero morphisms of sheaves
	$d_\lambda:\mathcal{L}_{\mathcal{F}_T}^\lambda\to \mathcal{O}_{X\times S}$ for all $\lambda\in \mathbf{X}^+$, such that for every geometric point $s$ of $S$, the divisor of zeros of $d_\lambda|_{X\times s}$ has degree $\langle\lambda,\mu\rangle$. Similarly we define ${_{\leq\nu} X^\mu}$, where this time the second data is a collection of  non-zero morphisms of sheaves $\mathcal{L}_{\mathcal{F}_T}^\lambda\to \mathcal{O}_{X\times s}(x\cdot\langle\lambda,\nu\rangle)$, such that the divisor of zeros of the restriction to $X\times s$ has degree $\langle\lambda,\mu\rangle$. We then define the ind-scheme $_\infty X^\mu$ as a direct limit.

	\subsection{Zastava spaces}\label{section Zastava spaces}
	For $\mu\in\mathbf{Y}$, we define various Zastava spaces of degree $\mu$, as displayed in \cite[§4.1.4]{ABBGM}. We let $_{\leq\nu} Z^\mu$ (resp. $^1_{\leq\nu} Z^\mu$) be the open substack of
	$$_{\leq\nu} \overline{\mathrm{Bun}}_{N^-}\times_{\mathrm{Bun}_G}\mathrm{Bun}_B^\mu \quad(\text{resp.}~^1_{\leq\nu} \overline{\mathrm{Bun}}_{N^-}\times_{\mathrm{Bun}_G}\mathrm{Bun}_B^\mu)$$
	defined by the condition that the reductions to $N^-$ and $B$ are transversal. This means that for all $\lambda\in\mathbf{X}^+$, the composition $$\kappa^{\lambda,-}_\nu\circ\kappa^\lambda:\mathcal{L}^\lambda_{\mathcal{F}_T}\to \mathcal{O}_X(x\cdot\langle\nu,\lambda\rangle)$$ is non-zero. We also define $_{\nu} Z^\mu$ (resp. $^1_{\nu} Z^\mu$) by replacing $\leq\nu$ with $\nu$ in the previous definition, and put $ Z^\mu:={_{\leq0} Z^\mu}$. For $\nu_1\leq\nu_2$, we have a canonical closed embedding $_{\leq\nu_1} Z^\mu\hookrightarrow{_{\leq\nu_2} Z^\mu}$ (resp. $_{\leq\nu_1}^1 Z^\mu\hookrightarrow{_{\leq\nu_2}^1 Z^\mu}$), so that we can let $_{\infty} Z^\mu$ (resp. $_{\infty}^1 Z^\mu$) be the direct limit (in the category of ind-algebraic stacks over $\mathbb{F}$) 
	$$\varinjlim_{\nu\in \mathbf{Y}}{_{\leq\nu} Z^\mu}\quad (\text{resp}.~~\varinjlim_{\nu\in \mathbf{Y}}{_{\leq\nu}^1 Z^\mu}).$$ 
	In the sequel, we will often have statements involving $_{?}^k Z^\mu$, for $?\in\{\nu,\leq\nu,\infty\}$, $k\in\{0,1\}$, making the convenient identification $_{?}^0 Z^\mu:={_{?} Z^\mu}$.  Finally, we will denote by $^k\mathfrak{p}^\mu:{_\infty ^kZ}^\mu\to {_\infty^k\overline{\mathrm{Bun}}_{N^-}}$ the canonical projection. 
	
	One of the main features of these spaces is their factorization property, which we now explain.  First notice that, thanks to the moduli space desription of ${_{\leq\nu} X^\mu}$ (§\ref{section colored divisor}), we have canonical projections $_{\leq\nu}\mathfrak{s}^\mu$ and $_{\infty}\mathfrak{s}^\mu$ from $_{\leq\nu} Z^\mu$ and $_{\infty} Z^\mu$ to ${_{\leq\nu} X^\mu}$ and ${_{\infty} X^\mu}$ respectively. We will denote by $_{\infty}^1\mathfrak{s}^\mu:{_{\infty}^1 Z^\mu}\to {_{\infty} X^\mu}$ the composition of $_{\infty}\mathfrak{s}^\mu$ with the canonical projection $_{\infty}^1 Z^\mu\to {_{\infty} Z^\mu}$, and similarly for $_{\leq\nu}^1\mathfrak{s}^\mu$. We put $\overset{\circ}{X}:=X-\{x\}$, and let 
 $$\smash[t]{\overset{\circ}{X}}^\mu$$ 
 be the corresponding open subscheme of $X^\mu$. For $\mu_1,\mu_2\in \mathbf{Y}$, we denote by $$(\smash[t]{\overset{\circ}{X}}^{\mu_1}\times X^{\mu_2})_{\mathrm{disj}}$$ the open substack of the product $\smash[t]{\overset{\circ}{X}}^{\mu_1}\times X^{\mu_2}$ consisting of pairs $(x_1,x_2)$ such that $x_1$ and $x_2$ have no point in common. 
	\begin{prop}{{\cite[Lemma 4.1.5]{BBGM}}}\label{prop factorization zastava}
		For  $\mu_1+\mu_2=\mu$ and $k\in\{0,1\}$, there exists a canonical isomorphism of stacks
		\begin{equation}\label{facto}
			^k_\infty Z^\mu\times_{_\infty X^\mu}(\smash[t]{\overset{\circ}{X}}^{\mu_1}\times X^{\mu_2})_{\mathrm{disj}}\simeq (Z^{\mu_1}\times {^k_\infty Z^{\mu_2}})\times_{(X^{\mu_1}\times{_\infty X^{\mu_2}})}(\smash[t]{\overset{\circ}{X}}^{\mu_1}\times X^{\mu_2})_{\mathrm{disj}}.
		\end{equation}
	\end{prop}
	The following result is used implicitely in \cite{ABBGM}. 
	\begin{prop}{{\cite[§3.6]{BBGM}}}\label{zastava prop}
		Let $\nu\in\mathbf{Y}$, and $U$ be an open substack of finite type of ${^k_{\leq\nu}\overline{\mathrm{Bun}}_{N^-}}$ (resp. ${^k_\nu\overline{\mathrm{Bun}}_{N^-}}$). There exists $\mu\in\mathbf{Y}$ large enough so that $U$ belongs to the image of the canonical projection ${_{\leq\nu} ^kZ}^\mu\to {_{\leq\nu}^k\overline{\mathrm{Bun}}_{N^-}}$ (resp. ${_\nu ^kZ}^\mu\to {_\nu^k\overline{\mathrm{Bun}}_{N^-}}$).
	\end{prop}

	\subsection{Canonical torsors}\label{section torsors}
	Let $(\mathcal{F}_G,\kappa_0)$ be an object of $_0\overline{\mathrm{Bun}}_{N^-}(S)$. By definition, the restriction $(\mathcal{F}_G|_{\{x\}},\kappa_0|_{\{x\}})$ from $X$ to $\{x\}$ produces an $N^-$-bundle over $S$.  We introduce the stack $_0^1\mathcal{N}$ as classifying the triples $(\mathcal{F}_G,\kappa_0,\Phi)$, where $\Phi$ is a trivialization of the $N^-$-bundle  $(\mathcal{F}_G|_{\{x\}},\kappa_0|_{\{x\}})$. The canonical projection $^1_0\mathcal{N}\to {_0\overline{\mathrm{Bun}}_{N^-}}$ makes $^1_0\mathcal{N}$ an $N^-$-torsor over $_0\overline{\mathrm{Bun}}_{N^-}$. A similar construction (cf. \cite[§3.2.6]{FGV}) produces an $N^-$-torsor $^1_\nu\mathcal{N}$ over $_\nu\overline{\mathrm{Bun}}_{N^-}$. It follows directly from the constructions (and \cite[(27)]{ABBGM}) that we have a canonical isomorphism
	\begin{equation}\label{torsor twisted product}
		^1_\nu\overline{\mathrm{Bun}}_{N^-}\simeq G\times^{N^-}{^1_\nu\mathcal{N}}, 
	\end{equation}
 and that the $T$-action on $^1_\nu\overline{\mathrm{Bun}}_{N^-}$ corresponds to the diagonal action on the right-hand side. The above isomorphism induces morphisms
	\begin{equation}\label{evaluation map}
		\overline{\mathrm{ev}}_\nu:{^1_\nu\overline{\mathrm{Bun}}}_{N^-}\to G/N^{-},\quad \mathrm{ev}_\nu:{^1_\nu\mathrm{Bun}}_{N^-}\to G/N^{-}. 
	\end{equation}
where $\mathrm{ev}_\nu$ is obtained by restricting $\overline{\mathrm{ev}}_\nu$. 
 
 Let $\mu\in\mathbf{Y}$. The isomorphism \eqref{torsor twisted product} is compatible with the definition of the Zastava space $^1_{\nu} Z^\mu$. Namely, if we let $^1_\nu\widetilde{\mathcal{N}}$ be the open substack of
	$$^1_\nu\mathcal{N}\times_{\mathrm{Bun}_G}\mathrm{Bun}_B^\mu$$
	defined by the transversality condition (cf. §\ref{section Zastava spaces}), then we have an isomorphism
	\begin{equation}\label{torsor twisted product zastava}
		^1_{\nu} Z^\mu\simeq G\times^{N^-}{^1_\nu\widetilde{\mathcal{N}}}.
	\end{equation}
	
	Let $\overline{x}':=\{x'_1,\cdots,x'_m\}$ be a collection of points of $X-\{x\}$, and denote by $_\infty\overline{\mathrm{Bun}}_{N^-}^{\mathrm{n.z.}\overline{x}'}$ the open substack of $_\infty\overline{\mathrm{Bun}}_{N^-}$ defined by the condition that the divisor of zeros of the morphisms $\kappa^\lambda$ is disjoint from $\overline{x}'$, for all $\lambda\in \mathbf{X}^+$. For $i=1,\dots,m$, we denote by $t'_i$ a local coordinate on $X$ at $x'_i$. Using the same ideas as in the previous paragraph, one can then construct an $\prod_{i=1}^mN^-[[t'_i]]$-torsor $\mathcal{N}^{\overline{x}'}$ over  $_\infty\overline{\mathrm{Bun}}_{N^-}^{\mathrm{n.z.}\overline{x}'}$. Following the arguments of \cite[Lemma 3.2.7]{FGV}, one can show that the action of $\prod_{i=1}^mN^-[[t'_i]]$ on $\mathcal{N}^{\overline{x}'}$ extends to an action of $\prod_{i=1}^mN^-((t'_i))$.
	
	We will denote by $^1_\infty\overline{\mathrm{Bun}}_{N^-}^{\mathrm{n.z.}\overline{x}'}$, $^1\mathcal{N}^{\overline{x}'}$ the corresponding stacks obtained by base change through the map $^1\mathrm{Bun}_G\to \mathrm{Bun}_G$.
	
	\section{Perverse sheaves on the semi-infinite flag variety}\label{section Perverse sheaves on the semi-infinite flag variety}
	In this section, we recall the construction of certain categories of perverse sheaves on (some avatar of) the semi-infinite flag manifold. This construction comes from \cite{ABBGM} in the case where $\mathrm{char}(\mathbb{k})=0$, but the definitions also make sense for arbitrary characteristic.
	\subsection{Main definition}\label{section Main definition}
	Let $\mu\in\mathbf{Y}$ and $k\in\{0,1\}$. Notice that, by \cite[Lemma 3.7]{BBGM}, the morphism $\mathrm{Bun}_B^\mu\to \mathrm{Bun}_G$ is smooth whenever $\langle\alpha,\mu\rangle>-2$ for all positive root $\alpha$ (or more generally when $\langle\alpha,\mu\rangle>2g-2$ for all positive root $\alpha$ when $X$ is a general smooth projective curve of genus $g$). Thus,  the morphism $^k\mathfrak{p}^\mu:{_\infty ^kZ}^\mu\to {_\infty^k\overline{\mathrm{Bun}}_{N^-}}$ from §\ref{section Zastava spaces} is smooth when $\mu$ satisfies the previous condition.  For all $k\in\{0,1\}$, $\mu_1,\mu_2\in \mathbf{Y}^{pos}$ such that $\mu=\mu_1+\mu_2$, with $\langle\alpha,\mu_i\rangle>-2$, for $i=1,2$ and all positive root $\alpha$, we let $^ks_{1,2}$ (resp. $^kt_{1,2}$) denote the open inclusion of the left-hand side (resp. right-hand side) of \eqref{facto} into $^k_\infty Z^\mu$ (resp. $Z^{\mu_1}\times {^k_\infty Z^{\mu_2}}$). Also notice that, since the group $\prod_{i=1}^mN^-((t'_i))$ is an ind-pro-unipotent group scheme, one can define the equivariance condition for perverse sheaves on $^k\mathcal{N}^{\overline{x}'}$ with respect to the action of $\prod_{i=1}^mN^-((t'_i))$ (by writting the latter group as a direct limit of closed subgroup schemes, and imposing the equivariance condition for each subgroup, cf. for instance \cite[§1.2.1]{GAITSGORY2018789}). 
	
	We now define the category $\mathrm{Perv}_{G^k}(\mathcal{F}l^{\frac{\infty}{2}})$. By definition, it is the full subcategory of $\mathrm{Perv}_T({^k_\infty\overline{\mathrm{Bun}}_{N^-}})$ consisting of objects $\mathcal{F}$ satisfying the following two conditions:
	\begin{enumerate}
		\item (Factorization) For $\mu=\mu_1+\mu_2$ with $\langle\alpha,\mu_i\rangle>2g-2$ for all positive root $\alpha$, the complex $({^ks}_{1,2})^*(^k\mathfrak{p}^\mu)^*\mathcal{F}$ is isomorphic (up to a shift by the corresponding relative dimensions), under the isomorphism \eqref{facto}, to the complex
		$$({^kt}_{1,2})^*(\mathrm{IC}_{Z^{\mu_1}}\boxtimes  (^k\mathfrak{p}^{\mu_2})^*\mathcal{F}).$$
		\item (Equivariance) For all $m$-tuple $\overline{x}'$ of $X-\{x\}$, the pullback of $\mathcal{F}$ under the composition
		$$^k\mathcal{N}^{\overline{x}'}\to {_\infty^k\overline{\mathrm{Bun}}_{N^-}^{\mathrm{n.z.}\overline{x}'}}\to {_\infty^k\overline{\mathrm{Bun}}_{N^-}} $$
		is $\prod_{i=1}^mN^-((t'_i))$-equivariant.
	\end{enumerate}
 \begin{rem}
		Compared with the definition of \cite[§4.2.1]{ABBGM}, condition (3) is missing. When $k=1$ (which is our main case of interest), this condition is automatic (see Remark 4.2.2 in \textit{loc. cit.}). When $k=0$, we shall see that condition (3) is also automatic, since the obtained category turns out to be -- with or without condition (3) -- the Serre subcategory of $\mathrm{Perv}_T({_\infty\overline{\mathrm{Bun}}_{N^-}})$ generated by the IC-sheaves on each $_{\leq\nu}\overline{\mathrm{Bun}}_{N^-}$, for $\nu\in\mathbf{Y}$.  
	\end{rem}
We omit the proofs of the next two results, as one can copy the arguments of \textit{loc. cit.}.
	\begin{prop}{{\cite[Proposition 4.2.8]{ABBGM}}}
		The equivariance condition is implied by the factorization condition.
	\end{prop}
	
	\begin{prop}{{\cite[Corollary 4.2.9]{ABBGM}}}
		Let $k\in\{0,1\}$. As a subcategory of the category $\mathrm{Perv}_T({^k_\infty\overline{\mathrm{Bun}}_{N^-}})$, $\mathrm{Perv}_{G^k}(\mathcal{F}l^{\frac{\infty}{2}})$ is stable under extensions.
	\end{prop}
 
	\subsection{Simple objects} 
	We will now see that the structure of the category $\mathrm{Perv}_{G[[t]]}(\mathcal{F}l^{\frac{\infty}{2}})$ does \textit{not} depend on the characteristic of $\mathbb{k}$ (at least when Corollary \ref{coro stalks} holds). Let $\nu\in\mathbf{Y}$, and $\mathcal{F}\in\mathrm{Perv}_T(G/N^-)$ (the action of $T$ is induced by the right action on $G$ by translation). The isomorphism \eqref{torsor twisted product} allows us to consider the external twisted tensor product
	\begin{equation}\label{twisted}
		\mathcal{F}\widetilde{\boxtimes}\mathrm{IC}_{_{\nu}\overline{\mathrm{Bun}}_{N^-}}\in \mathrm{D}_T({^1_{\nu}\overline{\mathrm{Bun}}_{N^-}}),
	\end{equation}
	defined as the perversely shifted pullback of $\mathcal{F}\boxtimes\mathrm{IC}_{_{\nu}\overline{\mathrm{Bun}}_{N^-}}\in \mathrm{Perv}_T({G/N^-\times {_\nu\overline{\mathrm{Bun}}_{N^-}}})$ through the canonical projection $G\times^{N^-}{^1_\nu\mathcal{N}}\to G/N^-\times {_\nu\overline{\mathrm{Bun}}_{N^-}}$. In particular, the object $\eqref{twisted}$ is perverse.
	\begin{prop}{{\cite[Proposition 4.2.4]{ABBGM}}}\label{simple objects prop}
		\begin{enumerate}
			\item 	For any $\mathcal{F}\in\mathrm{Perv}_T(G/N^-)$ and $\nu\in\mathbf{Y}$, the object 
			$$(\overline{i}_\nu)_{!*}(\mathcal{F}\widetilde{\boxtimes}\mathrm{IC}_{_{\nu}\overline{\mathrm{Bun}}_{N^-}})\quad (\text{resp.}~ (\overline{i}_\nu)_{!*}(\mathrm{IC}_{_{\nu}\overline{\mathrm{Bun}}_{N^-}})) $$ is an object of $\mathrm{Perv}_{G^1}(\mathcal{F}l^{\frac{\infty}{2}})$ (resp. of $\mathrm{Perv}_{G[[t]]}(\mathcal{F}l^{\frac{\infty}{2}})$).
			\item  For any $\mathcal{G}\in \mathrm{Perv}_{G^1}(\mathcal{F}l^{\frac{\infty}{2}})$ (resp. $\mathcal{G}\in \mathrm{Perv}_{G[[t]]}(\mathcal{F}l^{\frac{\infty}{2}})$) and $\nu\in\mathbf{Y}$, the perverse cohomologies of $(\overline{i}_\nu)^*\mathcal{G}$ are of the form 
			$$\mathcal{F}\widetilde{\boxtimes}\mathrm{IC}_{_{\nu}\overline{\mathrm{Bun}}_{N^-}}\qquad \text{(resp.}~ \bigoplus_{i=1}^m\mathrm{IC}_{_{\nu}\overline{\mathrm{Bun}}_{N^-}})$$ 
			for some $\mathcal{F}\in\mathrm{Perv}_T(G/N^-)$ (resp. $m\geq1$).
		\end{enumerate}
		
	\end{prop}
	\begin{proof}
		Let $\mu,\nu\in\mathbf{Y}$, with $\langle\alpha,\mu\rangle>2g-2$ for all positive root $\alpha$. First notice that, by considering the base change of  $^k\mathfrak{p}^\mu:{_\infty ^kZ}^\mu\to {_\infty^k\overline{\mathrm{Bun}}_{N^-}}$ (resp. $^k\mathcal{N}^{\overline{x}'}\to {_\infty^k\overline{\mathrm{Bun}}_{N^-}}$) through ${_\nu^k\overline{\mathrm{Bun}}_{N^-}}\to {_\infty^k\overline{\mathrm{Bun}}_{N^-}}$, one can define the factorization and equivariance conditions for objects of $\mathrm{Perv}_T({_\nu^k\overline{\mathrm{Bun}}_{N^-}})$. The same observation applies for objects of $\mathrm{Perv}_{T}({_\nu^k\mathrm{Bun}}_{N^-})$ (this will be useful at the end of this proof). We denote by $\mathrm{Perv}_{G^1}({_\nu^1\overline{\mathrm{Bun}}_{N^-}})$ (resp. $\mathrm{Perv}_{G[[t]]}({_\nu\mathrm{Bun}}_{N^-})$) the category of objects satisfying these two conditions. Remark that the functor $(\overline{i}_\nu)_{!*}$ sends $\mathrm{Perv}_{G^1}({_\nu^1\overline{\mathrm{Bun}}_{N^-}})$ (resp. $\mathrm{Perv}_{G[[t]]}({_\nu\mathrm{Bun}}_{N^-})$) to $\mathrm{Perv}_{G^1}(\mathcal{F}l^{\frac{\infty}{2}})$ (resp. $\mathrm{Perv}_{G[[t]]}(\mathcal{F}l^{\frac{\infty}{2}})$). So we only need to check that  $\mathcal{F}\widetilde{\boxtimes}\mathrm{IC}_{_{\nu}\overline{\mathrm{Bun}}_{N^-}}$ (resp. $\mathrm{IC}_{_{\nu}\overline{\mathrm{Bun}}_{N^-}}$) satisfies the desired properties to prove the first point.
		
		The fact that $\mathcal{F}\widetilde{\boxtimes}\mathrm{IC}_{_{\nu}\overline{\mathrm{Bun}}_{N^-}}$ (resp. $\mathrm{IC}_{_{\nu}\overline{\mathrm{Bun}}_{N^-}}$) satisfies the equivariance condition is obvious. The factorization condition follows from the fact that the projection ${_\nu ^kZ}^\mu\to {_\nu^k\overline{\mathrm{Bun}}_{N^-}}$ is smooth by the assumption on $\mu$. Indeed, when $k=0$, we deduce that the pullback of  $\mathrm{IC}_{_{\nu}\overline{\mathrm{Bun}}_{N^-}}$ through the morphism ${_\nu Z}^\mu\to {_\nu\overline{\mathrm{Bun}}_{N^-}}$ is an $IC$-sheaf, so the factorization property follows from the fact that the external tensor product of two IC-sheaves is an IC-sheaf (cf. \cite[Lemma 3.3.14]{achar2021perverse}). Similarly when $k=1$, the pullback of $\mathcal{F}\widetilde{\boxtimes}\mathrm{IC}_{_{\nu}\overline{\mathrm{Bun}}_{N^-}}$ through ${_\nu ^1Z}^\mu\to {_\nu^1\overline{\mathrm{Bun}}_{N^-}}$ identifies with $\mathcal{F}\widetilde{\boxtimes}\mathrm{IC}_{{_\nu Z}^\mu}$, so that the factorization property again follows from the fact that any $IC$-sheaf verifies it.
		
		Now we prove the second point. The factorization and equivariance properties are compatible with respect to restriction to $_{\nu}\overline{\mathrm{Bun}}_{N^-}$. So it suffices to check that the functor $\mathcal{F}\mapsto \mathcal{F}\widetilde{\boxtimes}\mathrm{IC}_{_{\nu}\overline{\mathrm{Bun}}_{N^-}}$ defines equivalences of categories
		\begin{align*}
  \mathrm{Perv}(pt)&\xrightarrow{\sim}  \mathrm{Perv}_{G[[t]]}({_\nu\overline{\mathrm{Bun}}_{N^-}})\\
		    \mathrm{Perv}_T(G/N^-)&\xrightarrow{\sim}  \mathrm{Perv}_{G^1}({_\nu^1\overline{\mathrm{Bun}}_{N^-}}).
		\end{align*}
		We only deal with the second case, since the same arguments can be adapted for the first one.  First, we claim that every object $ \mathcal{F}'\in \mathrm{Perv}_{G^1}({_\nu^1\overline{\mathrm{Bun}}_{N^-}})$ is the Goresky-Macpherson extension of its restriction to ${_\nu^1\mathrm{Bun}}_{N^-}$. Indeed, assume that it is not, then we can find an open subscheme of finite type $U$ inside the support of $\mathcal{F}'$ such that either $*$ or $!$-restriction of $\mathcal{F}'|_{U}$ to $({_\nu^1\overline{\mathrm{Bun}}_{N^-}}- {_\nu^1\mathrm{Bun}}_{N^-})\cap U$ does \textit{not} lives in negative (resp. positive) perverse cohomological degrees (cf. for instance the proof of \cite[Lemma 3.3.3]{achar2021perverse}). But thanks to Proposition \ref{zastava prop}, we can find  $\mu_1\in\mathbf{Y}$ large enough so that $U$ belongs to the image of ${_\nu ^1Z}^{\mu_1}\to {_\nu^1\overline{\mathrm{Bun}}_{N^-}}$. Therefore, using the identification \eqref{facto}, we deduce that we can find $\mu_2\in\mathbf{Y}$ large enough (for instance $\mu_2=\mu_1$) such that, for $\mu=\mu_1+\mu_2$, either $*$ or $!$-restriction (of the pullback of $ \mathcal{F}'|_{U}$ through the map ${_\nu ^1Z}^{\mu_1+\mu_2}\to {_\nu^1\overline{\mathrm{Bun}}_{N^-}}$) on 
		$$((Z^{\mu_1}-\overset{\circ}{Z^{\mu_1}})\times {^1_\nu Z^{\mu_2}})\times_{(X^{\mu_1}\times{_\infty X^{\mu_2}})}(\smash[t]{\overset{\circ}{X}}^{\mu_1}\times X^{\mu_2})_{\mathrm{disj}}$$
		does not live in negative (resp. positive) perverse cohomological degrees. However, this contradicts the factorization property.
		
		So we are now reduced (thanks to the fact that Goresky-Macpherson extension is a fully faithful functor) to showing that the functor $\mathcal{F}\mapsto \mathcal{F}\widetilde{\boxtimes}\mathrm{IC}_{{_{\nu}\mathrm{Bun}}_{N^-}}$ induces an equivalence of categories
		$$\mathrm{Perv}_T(G/N^-)\xrightarrow{\sim}  \mathrm{Perv}_{G^1}({_\nu^1\mathrm{Bun}}_{N^-}).$$
		Since ${_\nu^1\mathrm{Bun}}_{N^-}$ is smooth, this functor coincides with the pullback via $\mathrm{ev}_\nu: {_\nu^1\mathrm{Bun}}_{N^-}\to G/N^-$, so that full faithfulness follows from the fact that $\mathrm{ev}_\nu$ is smooth surjective with connected fibers. To prove the essential surjectivity, the argument of \textit{loc. cit.} is the following. Let $\mathcal{F}'\in \mathrm{Perv}_{G^1}({_\nu^1\mathrm{Bun}}_{N^-})$. For any non-empty collection of points $\overline{x}'$ distinct from $x$, the pullback of $\mathcal{F}'$ through the map 
		$${_\nu^1\mathrm{Bun}}_{N^-}\times_{{_\nu^1\overline{\mathrm{Bun}}}_{N^-}} {^1\mathcal{N}^{\overline{x}'}}\to {_\nu^1\mathrm{Bun}}_{N^-}$$
		is $\prod_{i=1}^mN^-((t'_i))$-equivariant by definition. The assertion then follows from the fact that the latter ind-group-scheme acts transitively along the fibers of the map obtained by composition
		$${_\nu^1\mathrm{Bun}}_{N^-}\times_{{_\nu^1\overline{\mathrm{Bun}}}_{N^-}} {^1\mathcal{N}^{\overline{x}'}}\to G/N^-. $$
	\end{proof}
	\begin{rem}\label{remark G/B}
		Since the action of $T$ on $G/N^-$ is free, we have an equivalence of categories $\mathrm{Perv}_T(G/N^-)\simeq \mathrm{Perv}(G/B^-)$. So the previous proof has constructed an equivalence of categories
		$$\mathrm{Perv}(G/B^-)\to  \mathrm{Perv}_{G^1}({_\nu^1\overline{\mathrm{Bun}}}_{N^-}).$$
	\end{rem}
	As a corollary of the proof of the previous proposition, we obtain:
	\begin{coro}\label{coro constructible}
		Let $\mathcal{F}\in \mathrm{Perv}_{G[[t]]}(\mathcal{F}l^{\frac{\infty}{2}})$ and $\nu\in\mathbf{Y}$. The cohomology groups of $(i_\nu)^*\mathcal{F}$ are constant sheaves.
	\end{coro}
	\begin{proof}
	We will prove by induction on $m$ that the cohomology groups of ${^p\tau}^{\leq m}((i_\nu)^*\mathcal{F})$ are constant sheaves, starting at some negative enough integer such that ${^p\tau}^{\leq m}((i_\nu)^*\mathcal{F})=0$. So assume that there exists $m$ such that the cohomology groups of ${^p\tau}^{\leq m}((i_\nu)^*\mathcal{F})$ are constant sheaves. We have a distinguished triangle
		$${^p\tau}^{\leq m}((i_\nu)^*\mathcal{F})\to {^p\tau}^{\leq m+1}((i_\nu)^*\mathcal{F})\to {^pH}^{m+1}((i_\nu)^*\mathcal{F})\xrightarrow{+1}{}. $$
		By the second point of Proposition \ref{simple objects prop}, ${^pH}^{m+1}((i_\nu)^*\mathcal{F})$ is a constant sheaf. From this together with Proposition \ref{prop contractible}, we deduce that the cohomology groups of ${^p\tau}^{\leq m+1}((i_\nu)^*\mathcal{F})$ are constant sheaves.
	\end{proof}
	\begin{coro}{{\cite[Proposition 4.2.4]{ABBGM}}}\label{coro prop simple}
		Let $k\in\{0,1\}$.  All the simple objects of $\mathrm{Perv}_{G^k}(\mathcal{F}l^{\frac{\infty}{2}})$ can be written as in the first point of Proposition \ref{simple objects prop}, with $\mathcal{F}$ an irreducible object of $\mathrm{Perv}_T(G/N^{-})$.
	\end{coro}
 \begin{proof}
     We only treat the case $k=0$, since the arguments are easily adapted to the case $k=1$. By the first point of Proposition \ref{simple objects prop}, objects that are written in the form of \textit{loc. cit.} are indeed objects of $\mathrm{Perv}_{G[[t]]}(\mathcal{F}l^{\frac{\infty}{2}})$. Conversely, let $\mathcal{F}$ be a simple object of $\mathrm{Perv}_{G[[t]]}(\mathcal{F}l^{\frac{\infty}{2}})$, and denote by $Z$ its support. We deduce from \eqref{eq strat} that there exists a unique $\nu\in\mathbf{Y}$ such that ${_{\nu}\overline{\mathrm{Bun}}_{N^-}}\cap Z$ is open in $Z$. Since $\mathcal{F}$ is an IC-sheaf (because it is a simple object), we deduce that it must be isomorphic to $(\overline{i}_\nu)_{!*}((\overline{i}_\nu)^*\mathcal{F})$. But we know that
     $$(\overline{i}_\nu)^*\mathcal{F}\simeq\bigoplus_{i=1}^m\mathrm{IC}_{_{\nu}\overline{\mathrm{Bun}}_{N^-}}$$
     for some $m\geq 1$. Since $\mathcal{F}$ is simple we must have $m=1$.
 \end{proof}
	
	\subsection{Stalks of simple objects}\label{section Intersection cohomology}
	The goal of this subsection is to prove that, when $G=\mathrm{SL}_n$, the results of \cite{BBGM} concerning the description of the stalks of $\mathrm{IC}_{\overline{\mathrm{Bun}}_{B^-}}$ do not depend on the characteristic of $\mathbb{k}$.
	
	For $\mu\in \mathbf{Y}^{pos}$, we let $^\mu\mathrm{Bun}_{B^-}$ be a stack fibered over $X^\mu$, whose fiber over $D\in X^\mu$ consists of the triples $(\mathcal{F}_G,\mathcal{F}_T,\kappa)\in\overline{\mathrm{Bun}}_{B^-}$ such that the morphism 
	$$\kappa^\lambda:\mathcal{V}_{\mathcal{F}_G}\to \mathcal{L}^\lambda_{\mathcal{F}_T}$$
	factors through the canonical morphism $\mathcal{L}^\lambda_{\mathcal{F}_T}(-\langle \lambda,D\rangle)\to \mathcal{L}^\lambda_{\mathcal{F}_T}$, with the resulting morphism $\mathcal{V}_{\mathcal{F}_G}\to\mathcal{L}^\lambda_{\mathcal{F}_T}(-\langle \lambda,D\rangle)$ being a morphism of vector bundles. We also define a morphism
	$$\varphi^\mu:\mathrm{Bun}_{B^-}\times X^\mu\to  \overline{\mathrm{Bun}}_{B^-},$$
	sending $((\mathcal{F}_G,\mathcal{F}_T,\kappa),D)$ to $(\mathcal{F}_G,\mathcal{F}_T(D),\tilde{\kappa})$, where $\tilde{\kappa}^\lambda$ is defined as the composition 
	$$\mathcal{V}_{\mathcal{F}_G}\xrightarrow{\kappa^\lambda} \mathcal{L}^\lambda_{\mathcal{F}_T}\to \mathcal{L}^\lambda_{\mathcal{F}_T}(\langle\lambda,D\rangle). $$
	This morphism $\varphi^\mu$ is representable by a locally closed immersion (\cite[Proposition 6.1.2]{BG}), and its image coincides with $^\mu\mathrm{Bun}_{B^-}$. For $\Gamma\in\mathfrak{A}(\mu)$, we let 
	$$\varphi^\mu_\Gamma:\mathrm{Bun}_{B^-}\times X^\mu_\Gamma\to  \overline{\mathrm{Bun}}_{B^-}$$ denote the locally closed immersion obtained by composing $\varphi^\mu$ with the inclusion of $\mathrm{Bun}_{B^-}\times X^\mu_\Gamma$ into $\mathrm{Bun}_{B^-}\times X^\mu$. 
	Also notice that, for $\theta\in\mathbf{Y}^{pos}$ such that $\theta\leq \mu$, restriction of $\varphi^\mu$ yields a morphism 
	$$\varphi_{\mu,\theta}:{\mathrm{Bun}}_{B^-}^{\theta-\mu}\times X^\mu\to  \overline{\mathrm{Bun}}^\theta_{B^-}, $$
	which is representable by a locally closed immersion. By Proposition 6.1.3 of \textit{loc. cit.}, we obtain a stratification
	\begin{equation}\label{strat bun}
		\overline{\mathrm{Bun}}_{B^-}^\theta=\bigsqcup_{\mu\in\mathbf{Y}^{pos},~\mu\leq \theta,\Gamma \in\mathfrak{A}(\mu)}{\mathrm{Bun}}_{B^-}^{\theta-\mu}\times X_\Gamma^\mu.
	\end{equation} 
\begin{conj}\label{conj stalks independance}
    For any $i\in\mathbb{Z}$ and closed point $x\in \overline{\mathrm{Bun}}_{B^-}$, the dimension of the $\mathbb{k}$-vector space $\mathrm{H}^i((\mathrm{IC}_{\overline{\mathrm{Bun}}_{B^-}})_x)$ does not depend on the characteristic of $\mathbb{k}$.
\end{conj}
	
	\begin{lem}\label{lem kuznetsov}
		If $G=\mathrm{SL}_n$, then Conjecture \ref{conj stalks independance} holds.
	\end{lem}
	The proof of this lemma is postponed to §\ref{section Laumon}, and will follow from the facts that: 
	\begin{enumerate}
		\item The stack $\overline{\mathrm{Bun}}_{B^-}$ admits a small resolution of singularities $\pi$ when $G=\mathrm{SL}_n$.
		\item The cohomology of the fibers of $\pi$ is independent of the characteristic.
	\end{enumerate} 
	Indeed, the first point implies that the stalk of $\mathrm{IC}_{\overline{\mathrm{Bun}}_{B^-}}$ at a point coincides with the cohomology of the fiber of this point through $\pi$, and the second point allows to conclude.

	

	\begin{prop}\label{prop stacks of BunB}Assume that Conjecture \ref{conj stalks independance} holds. Let $\mu\in \mathbf{Y}^{pos}$ and $\Gamma\in\mathfrak{A}(\mu)$. Then the complex $(\varphi_{\Gamma}^{\mu})^!\mathrm{IC}_{\overline{\mathrm{Bun}}_{B^-}}\in \mathrm{D}({\mathrm{Bun}}_{B^-}\times X_\Gamma^\mu)$ lives in perverse cohomological degrees $\geq2$.

	\end{prop}
	
	\begin{proof}
		In the case where $\mathrm{char}(\mathbb{k})=0$, the description of $$(\varphi_{\Gamma}^{\mu})^*\mathrm{IC}_{\overline{\mathrm{Bun}}_{B^-}}\in \mathrm{D}({\mathrm{Bun}}_{B^-}\times X_\Gamma^\mu)$$ 
  follows from \cite[Theorem 7.3]{BBGM}. It shows in particular that $(\varphi_{\Gamma}^{\mu})^*\mathrm{IC}_{\overline{\mathrm{Bun}}_{B^-}}$ lives in perverse cohomological degrees $\leq -2$. Applying Verdier duality, we deduce that $(\varphi_{\Gamma}^{\mu})^!\mathrm{IC}_{\overline{\mathrm{Bun}}_{B^-}}$ lives in perverse cohomological degrees $\geq2$ when $\mathrm{char}(\mathbb{k})=0$. Thus, the case of positive characteristic follows from Lemma \ref{lem kuznetsov}.
	\end{proof}
	For all $\nu\in\mathbf{Y}$, $\mu\in \mathbf{Y}^{pos}$ and $\Gamma\in\mathfrak{A}(\mu)$, we can similarly define a locally closed immersion
	$${_\nu\varphi^\mu_\Gamma}:{_{\nu}\mathrm{Bun}}_{N^-}\times X^\mu_\Gamma\to  {_{\leq\nu}\overline{\mathrm{Bun}}}_{N^-},$$
	which is obtained by base change of the morphism $\varphi^\mu_\Gamma$ through the morphism ${_{\leq\nu}\overline{\mathrm{Bun}}}_{N^-}\to {\overline{\mathrm{Bun}}}_{B^-}$.
		\begin{coro}\label{coro stalks}Assume that Conjecture \ref{conj stalks independance} holds. Let $\nu\in\mathbf{Y}$, $\mu\in \mathbf{Y}^{pos}$ and $\Gamma\in\mathfrak{A}(\mu)$. Then the complex $({_\nu\varphi_{\Gamma}^{\mu}})^!\mathrm{IC}_{{_{\leq\nu}\overline{\mathrm{Bun}}}_{N^-}}\in \mathrm{D}({_{\nu}\mathrm{Bun}}_{N^-}\times X_\Gamma^\mu)$ lives in perverse cohomological degrees $\geq2$.
		\end{coro}
		\begin{proof}
			This follows from Proposition \ref{prop stacks of BunB} together with the fact that the morphism 
   $$_{\leq\nu}\overline{\mathrm{Bun}}_{N^-}\to\overline{\mathrm{Bun}}_{B^-}$$ 
   is smooth (Remark \ref{prop open substack}), so that the perversely-shifted pullback of the IC-sheaf is the IC-sheaf. 
	\end{proof}
 
	\subsection{Laumon's compactification}\label{section Laumon} Recall the classical fact that there is an equivalence of categories between the category of rank-$n$ vector bundle $\mathcal{M}$ over $X$, endowed with an isomorphism $\wedge^n\mathcal{M}\simeq\mathcal{O}_X$, and the category of principal $\mathrm{SL}_n$-bundles over $X$, defined on objects by sending a sheaf $\mathcal{M}$ to the relative spectrum of the dual sheaf $\mathcal{M}^\vee$.
	
	In this subsection, we prove Lemma \ref{lem kuznetsov}. When $G$ is $\mathrm{SL}_n$, Laumon introduced another compactification of the morphism $\mathrm{Bun}_{B^-}\to\mathrm{Bun}_G$. Namely, $^L\overline{\mathrm{Bun}}_{B^-}$ is the smooth algebraic stack classifying the data of a rank $n$ vector bundle $\mathcal{M}$ over $X$ together with an isomorphism $\wedge^n\mathcal{M}\simeq\mathcal{O}_X$ and a complete flag $0=\mathcal{M}_0\subset\mathcal{M}_1\subset\cdots\subset \mathcal{M}_n=\mathcal{M}$ of locally free subsheaves, where each inclusion $\mathcal{M}_i\subset\mathcal{M}_{i+1}$ is merely an inclusion of subsheaves, and not necessarily of vector bundles (cf. \cite[§3.3.1]{Lau}, the version where the quotients $\mathcal{M}_{i+1}/\mathcal{M}_i$ are vector bundles is isomorphic to $\mathrm{Bun}_{B^-}$). There is then a natural morphism $\pi:{^L\overline{\mathrm{Bun}}_{B^-}}\to \overline{\mathrm{Bun}}_{B^-}$, sending a complete flag $(\mathcal{M}_i)_{0\leq i\leq n}$ to the point 
 $$(\mathcal{F}_G:=(\mathcal{M},\wedge^n\mathcal{M}\simeq\mathcal{O}_X),\mathcal{F}_T,\kappa)$$ of $\overline{\mathrm{Bun}}_{B^-}$ defined as follows: $\mathcal{F}_T$ is the $T$-bundle $(\mathrm{det}(\mathcal{M}_{i+1}/\mathcal{M}_{i}))_{0\leq i\leq n-1}$, and $(\kappa^\lambda)_{\lambda\in\mathbf{X}^+}$ is the collection of morphisms determined by the inclusions of sheaves 
	$$\mathcal{L}^{\varpi_k}_{\mathcal{F}_T}:=\wedge^k\mathcal{M}_k\subset \wedge^k\mathcal{M}=\mathcal{V}^{\varpi_k}_\mathcal{M}\quad\forall~ 0\leq k\leq n.$$
	The connected components of $^L\overline{\mathrm{Bun}}_{B^-}$ are labelled by $\mathbf{Y}$. We will denote by $^L\overline{\mathrm{Bun}}_{B^-}^\theta$ the connected component associated with an element $\theta\in\mathbf{Y}$. By construction, $^L\overline{\mathrm{Bun}}_{B^-}^\theta$ is the pre-image of $\overline{\mathrm{Bun}}_{B^-}^\theta$ through $\pi$, and we will denote by 
	$$\pi_\theta:{^L\overline{\mathrm{Bun}}_{B^-}^\theta}\to \overline{\mathrm{Bun}}_{B^-}^\theta $$
	the restriction of $\pi$. Thanks to \cite[Lemme 4.2.3]{Lau}, the fibers of $\pi_\theta$ are projective smooth schemes over $\mathbb{F}$.  
 We need one last preliminary result.
 \begin{lem}\label{lem trivial open}
     The trivial bundle $\mathcal{F}^0_G$ is open inside $\mathrm{Bun}_G$.
 \end{lem}
 \begin{proof}
     Recall that we have assumed that $X=\mathbb{P}^1$. Thanks to \cite[Theorem 2.3.7]{zhu2016introduction}, we have an isomorphism of $\mathbb{F}$-stacks
     $$[G[t^{-1}]\backslash \mathrm{Gr}]\simeq \mathrm{Bun}_G, $$
     where we take the quotient stack on the left-hand side. The trivial bundle $\mathcal{F}^0_G$ corresponds to the $G[t^{-1}]$-orbit associated with the class of the neutral element in $\mathrm{Gr}$. Thanks to \cite[Proposition 2.3.3]{zhu2016introduction}, the latter orbit is open in $\mathrm{Gr}$.
 \end{proof}
Finally, recall the definition of a ''pseudoaffine" variety over $\mathbb{F}$ (cf. \cite[§2.3]{kuznetsov1996laumon}). Those are defined by induction on the dimension. A pseudo affine space of dimension one is $\mathbb{A}^1$. A variety $Y$ is pseudo affine if it admits a locally trivial fibration $Y\to B$, with pseudo affine fibers and pseudo affine $B$. We are now ready to prove Lemma \ref{lem kuznetsov}:
	\begin{proof}
		Let $\theta\in\mathbf{Y}^{pos}$. First, let us justify why the stalk of $\mathrm{IC}_{\overline{\mathrm{Bun}}_{B^-}^\theta}$ at a point $(\mathcal{F}_G,\mathcal{F}_T,\kappa)\in \overline{\mathrm{Bun}}_{B^-}^\theta$ depends only on the type of defect of that point, i.e. to which stratum of \eqref{strat bun} it belongs. Similarly to §\ref{section Zastava spaces} (and as in \cite{BBGM}), one can define the Zastava space
		$${'Z^{\mu-\theta}}\subset \overline{\mathrm{Bun}}_{B^-}^\theta\times_{\mathrm{Bun}_G}\mathrm{Bun}_B^\mu $$
		as an open substack of the right-hand side, for any $\mu\in\mathbf{Y}^{pos}$. The Zastava spaces defined in this way satisfy an analogous factorization property (see for instance \cite[§7]{BBGM}). The canonical projection ${'Z^{\mu-\theta}}\to \overline{\mathrm{Bun}}_{B^-}^\theta$ is smooth for $\mu$ large enough, and any open substack of finite type of $\overline{\mathrm{Bun}}_{B^-}^\theta$ belongs to its image for $\mu$ large enough (cf. {\cite[§3.6]{BBGM}}). Moreover, the stratification by type of defect \eqref{strat bun} of $\overline{\mathrm{Bun}}_{B^-}^\theta$ induces a stratification of 
		${'Z^{\mu-\theta}}$. Therefore, it is enough to prove that the dimension of the cohomology of the stalk of $\mathrm{IC}_{'Z^{\mu-\theta}}$ is constant on every stratum. Thanks to the factorization property of ${'Z^{\mu-\theta}}$, we can see that it is enough to check that the stalk of $\mathrm{IC}_{\overline{\mathrm{Bun}}_{B^-}^\theta}$ at a point of $\mathrm{Bun}_{B^-}^{\theta-\delta}$ does not depend on the point, for $0\leq\delta\leq\theta$. Here $\mathrm{Bun}_{B^-}^{\theta-\delta}$ is seen as included in $\mathrm{Bun}_{B^-}^{\theta-\delta}\times X^{\theta-\delta}_{\{\{\theta-\delta\}\}}$ via the inclusion of $\{x\}$ into $X^{\theta-\delta}_{\{\{\theta-\delta\}\}}$. But the same arguments as in the proof of Corollary \ref{coro constructible} show that the restriction of $\mathrm{IC}_{\overline{\mathrm{Bun}}_{B^-}^\theta}$ to $\mathrm{Bun}_{B^-}^{\theta-\delta}$ has constant stalks.
		
		Therefore it is enough to consider points of the form $\varphi:=(\mathcal{F}^0_G,\mathcal{F}_T,\kappa)\in \overline{\mathrm{Bun}}_{B^-}^\theta$, in which case the result follows from  \cite{kuznetsov1996laumon}. Indeed, in this case $\varphi$ belongs to the quasimap space $\mathcal{Q}^\theta:=(\overline{\mathfrak{p}}^\theta)^{-1}(\mathcal{F}^0_G)$, and $(\pi_\theta)^{-1}(\{\varphi\})$ coincides with the pre-image of $\{\varphi\}$ under the small resolution of singularities $^L\mathcal{Q}^\theta\to \mathcal{Q}^\theta$ from \textit{loc. cit.}, where $^L\mathcal{Q}^\theta$ is Laumon's compactification of the space of maps of degree $\theta$ (\cite[§1.4]{kuznetsov1996laumon}). Thus, the stalk of $\mathrm{IC}_{\mathcal{Q}^\theta}$ at $\varphi$ coincides with the the \'etale cohomology $H^\bullet((\pi_\theta)^{-1}(\{\varphi\}),\mathbf{k})$. The scheme of finite type $(\pi_\theta)^{-1}(\{\varphi\})$ is stratified by pseudoaffine spaces by \cite[Corollary 2.4.4]{kuznetsov1996laumon}, so that $H^\bullet((\pi_\theta)^{-1}(\{\varphi\}),\mathbf{k})$ is easily seen to be independent of $\mathrm{char}(\mathbb{k})$. So the dimension of the cohomology of the stalk of $\mathrm{IC}_{\mathcal{Q}^\theta}$ at $\varphi$ is independent of $\mathrm{char}(\mathbb{k})$. But $\mathcal{Q}^\theta$ is open in $\overline{\mathrm{Bun}}_{B^-}^\theta$ thanks to Lemma \ref{lem trivial open}, from which we deduce the the stalk of $\mathrm{IC}_{\overline{\mathrm{Bun}}_{B^-}^\theta}$ at $\varphi$ coincides with the stalk of $\mathrm{IC}_{\mathcal{Q}^\theta}$ at $\varphi$.
	\end{proof}
	\begin{rem}
		It is stated in \cite[§0.2.1]{BG} that the morphism $\pi$ is a small resolution of singularities, so that computing the stalks of $\mathrm{IC}_{\overline{\mathrm{Bun}}_{B^-}}$ amounts to computing the cohomology of the fibers of $\pi$. We have just explained why the cohomology of those fibers is independent of $\mathrm{char}(\mathbb{k})$. 
	\end{rem}
	\subsection{Semi-simplicity of the spherical category} 
	For $\nu\in\mathbf{Y}$, we put 
 $$\mathcal{IC}^{\frac{\infty}{2}}_\nu:=(\overline{i}_\nu)_{!*}(\mathrm{IC}_{_\nu\overline{\mathrm{Bun}}_{N^-}}).$$ By Proposition \ref{simple objects prop}, the set $\{\mathcal{IC}^{\frac{\infty}{2}}_\nu,~\nu\in \mathbf{Y}\}$ coincides with the set of isomorphism classes of simple objects of the category $\mathrm{Perv}_{G[[t]]}(\mathcal{F}l^{\frac{\infty}{2}})$.
	\begin{lem}\label{lem stalks}
		Assume that Conjecture \ref{conj stalks independance} holds. Let $\nu_1,\nu_2\in\mathbf{Y}$, with $\nu_1<\nu_2$. The complex $$i^!_{\nu_1}\mathcal{IC}^{\frac{\infty}{2}}_{\nu_2}\in \mathrm{D}({_{\nu_1}\mathrm{Bun}}_{N^-})$$ 
  lives in perverse cohomological degrees $\geq2$.
		
	\end{lem}
	\begin{proof}
		We want to apply Corollary \ref{coro stalks}. For that, we just need to notice that the inclusion $_{\nu_1}\mathrm{Bun}_{N^-}\subset {_{\leq\nu_2}\overline{\mathrm{Bun}}}_{N^-}$ factors through the locally closed inclusion
		$${_{\nu_1}\mathrm{Bun}}_{N^-}\times X^{\nu_2-\nu_1}_{\{\{\nu_2-\nu_1\}\}}\hookrightarrow  {_{\leq\nu_2}\overline{\mathrm{Bun}}}_{N^-}$$
		from §\ref{section Intersection cohomology}, where the inclusion ${_{\nu_1}\mathrm{Bun}}_{N^-}\hookrightarrow{_{\nu_1}\mathrm{Bun}}_{N^-}\times X^{\nu_2-\nu_1}_{\{\{\nu_2-\nu_1\}\}}$ is induced by the diagonal embedding of $\{x\}$ into $X^{\nu_2-\nu_1}_{\{\{\nu_2-\nu_1\}\}}$.
	\end{proof}
The previous Lemma allows us to take back the arguments of \cite{ABBGM}.
	\begin{thm}[{\cite[Proposition 4.3.2]{ABBGM}}]\label{thm ss}
		Assume that Conjecture \ref{conj stalks independance} holds. The category  $\mathrm{Perv}_{G[[t]]}(\mathcal{F}l^{\frac{\infty}{2}})$ is semi-simple.
	\end{thm}
	\begin{proof}
		Let $\nu_1,\nu_2\in\mathbf{Y}$. We want to prove that 
		$$\mathrm{Ext}^1_{_{\infty}\overline{\mathrm{Bun}}_{N^-}}(\mathcal{IC}^{\frac{\infty}{2}}_{\nu_1},\mathcal{IC}^{\frac{\infty}{2}}_{\nu_2})=0. $$
		Arguing as in the second case of \cite[Proof of Proposition 4.4]{baumann:hal-01491529}, we see that the support of one sheaf must be included in the other for the space $\mathrm{Ext}^1_{_{\infty}\overline{\mathrm{Bun}}_{N^-}}(\mathcal{IC}^{\frac{\infty}{2}}_{\nu_1},\mathcal{IC}^{\frac{\infty}{2}}_{\nu_2})$ to be non-trivial.	Applying Verdier duality, we may and will assume that $_{\leq\nu_1}\overline{\mathrm{Bun}}_{N^-}\subset {_{\leq\nu_2}\overline{\mathrm{Bun}}}_{N^-}$, which means that $\nu_1\leq \nu_2$.
		
		Next, consider the stack 
		$$\mathcal{X}:={_{\leq\nu_2}\overline{\mathrm{Bun}}}_{N^-}-({_{\leq\nu_1}\overline{\mathrm{Bun}}}_{N^-}-{_{\nu_1}\mathrm{Bun}}_{N^-}), $$
		and denote by $j:\mathcal{X}\hookrightarrow {_{\leq\nu_2}\overline{\mathrm{Bun}}}_{N^-}$ the inclusion. Since $\mathcal{X}$ contains the open substack ${_{\nu_2}\overline{\mathrm{Bun}}}_{N^-}$, the complex $\mathcal{IC}^{\frac{\infty}{2}}_{\nu_2}$ coincides with the image of the $\mathrm{IC}$-sheaf on $\mathcal{X}$ though the functor $j_{!*}$. On the other hand, if we let $j'$ denote the inclusion ${_{\nu_1}\mathrm{Bun}}_{N^-}\hookrightarrow \mathcal{X}$, then $\mathcal{IC}^{\frac{\infty}{2}}_{\nu_1}$ coincides with the image of $j'_{!*}(\mathrm{IC}_{{_{\nu_1}\mathrm{Bun}}_{N^-}})$ through $j_{!*}$, since $j_{!*}j'_{!*}\simeq (j\circ j')_{!*}$. Applying \cite[Lemma 6.1.4]{FGV}, we get an inclusion
		$$\mathrm{Ext}^1_{{_{\leq\nu_2}\overline{\mathrm{Bun}}}_{N^-}}(\mathcal{IC}^{\frac{\infty}{2}}_{\nu_1},\mathcal{IC}^{\frac{\infty}{2}}_{\nu_2}) \hookrightarrow \mathrm{Ext}^1_{\mathcal{X}}(j'_{!*}(\mathrm{IC}_{{_{\nu_1}\mathrm{Bun}}_{N^-}}),\mathrm{IC}_{\mathcal{X}}). $$
		Since ${_{\nu_1}\mathrm{Bun}}_{N^-}$ is closed in $\mathcal{X}$, the latter $\mathrm{Ext}^1$ is isomorphic to 
		\begin{equation}\label{ext1}
			\mathrm{Ext}^1_{{_{\nu_1}\mathrm{Bun}}_{N^-}}(\mathrm{IC}_{{_{\nu_1}\mathrm{Bun}}_{N^-}},i^!_{\nu_1}\mathcal{IC}^{\frac{\infty}{2}}_{\nu_2}). 
		\end{equation}
		If $\nu_1=\nu_2$, then the above space is $\{0\}$ since ${_{\nu_1}\mathrm{Bun}}_{N^-}$ is contractible (Proposition \ref{prop contractible}). If $\nu_1<\nu_2$, then the complex $i^!_{\nu_1}\mathcal{IC}^{\frac{\infty}{2}}_{\nu_2}$ lives in perverse cohomological degrees $\geq 2$ thanks to Lemma \ref{lem stalks}. 
	\end{proof}
 The next result will be useful in §\ref{section Iwahori-Whittaker for the semi-infinite flag variety}.
 \begin{coro}\label{coro kostant}
     Let $\nu,\nu''\in\mathbf{Y}$ with $\nu''\leq \nu$, $r\in\mathbb{Z}_{\geq0}$, and put $\mu_r:=(\ell^r-1)\zeta$. For $r$ large enough, there is an isomorphism of complexes of vector spaces
     $$\mathcal{IC}^{\frac{\infty}{2}}_{\nu}|_{{_{ \nu''}\mathrm{Bun}}_{N^-}} \simeq \mathcal{IC}^{\mu_r}|_{\mathrm{Gr}^{{\mu_r-(\nu-\nu'')}}},$$
     where $\mathcal{IC}^{\frac{\infty}{2}}_{\nu}|_{{_{ \nu''}\mathrm{Bun}}_{N^-}}$ (resp. $\mathcal{IC}^{\mu_r}|_{\mathrm{Gr}^{\mu_r-(\nu-\nu'')}}$) is meant to denote the stalk of $\mathcal{IC}^{\frac{\infty}{2}}_{\nu}$ (resp. of $\mathcal{IC}^{\mu_r}$) at any point of ${_{ \nu''}\mathrm{Bun}}_{N^-}$ (resp. of $\mathrm{Gr}^{\mu_r-(\nu-\nu'')}$)
 \end{coro}
 \begin{proof}
     Since the dimensions of the stalks of $\mathcal{IC}^{\frac{\infty}{2}}_{\nu}$ and $\mathcal{IC}^{\mu_r}$ do not depend on the characteristic of $\mathbb{k}$, the result follows from the end of the proof of \cite[Theorem 5.3.5]{ABBGM}. Let us briefly recall the arguments.

     On the one hand, the stalks of $\mathcal{IC}^{\frac{\infty}{2}}_{\nu}$ do not depend on $\mathrm{char}(\mathbb{k})$ thanks to Lemma \ref{lem kuznetsov}, so that the generating function of $\mathcal{IC}^{\frac{\infty}{2}}_{\nu}|_{{_{ \nu''}\mathrm{Bun}}_{N^-}}$ is given by the $q$-analogue of Kostant's partition function $\mathcal{K}^{\nu-\nu''}(t)$ associated with $\nu-\nu''$ thanks to \cite[§3]{feigin2semi} (or the main result of \cite{kuznetsov1996laumon}).

     On the other hand, the generating function of $\mathcal{IC}^{\mu_r}|_{\mathrm{Gr}^{\mu_r-(\nu-\nu'')}}$ does not depend on $\mathrm{char}(\mathbb{k})$ thanks to Lemma \ref{lem Steinberg weights}. The latter coincides with $\mathcal{K}^{\mu_r-(\mu_r-(\nu-\nu''))}(t)=\mathcal{K}^{\nu-\nu''}(t)$ when $r$ is large enough (so that $\mu_r$ is dominant enough) thanks to \cite[§§9,11]{lusztigsingularities}.
 \end{proof}
	\subsection{Iwahori-equivariant objects}\label{section Iwahori-equivariant objects} We now define the main category of interest (following \cite[§4.4.3]{ABBGM}). Let $\nu\in\mathbf{Y}$, and denote by ${^{I_{\mathrm{u}}}_\nu\overline{\mathrm{Bun}}}_{N^-}$ and ${^{I_{\mathrm{u}}}_\infty\overline{\mathrm{Bun}}}_{N^-}$ (resp. ${^{I}_\nu\overline{\mathrm{Bun}}}_{N^-}$ and ${^{I}_\infty\overline{\mathrm{Bun}}}_{N^-}$) the quotient stacsk of ${^1_\nu\overline{\mathrm{Bun}}}_{N^-}$ and  ${^1_\infty\overline{\mathrm{Bun}}}_{N^-}$ by $N$ (resp. by $B$). We let $\mathrm{Perv}_{I_{\mathrm{u}}}(\mathcal{F}l^{\frac{\infty}{2}})$ (resp. $\mathrm{Perv}_{I}(\mathcal{F}l^{\frac{\infty}{2}})$) be the subcategory of 
 $$\mathrm{Perv}({^{I_{\mathrm{u}}}_\infty\overline{\mathrm{Bun}}}_{N^-})\quad \text{(resp.}~ \mathrm{Perv}({^{I}_\infty\overline{\mathrm{Bun}}}_{N^-}))$$ 
 consisting of objects whose image through the forgetful functor to $ \mathrm{Perv}({^{1}_\infty\overline{\mathrm{Bun}}}_{N^-})$ belongs to $\mathrm{Perv}_{G^1}(\mathcal{F}l^{\frac{\infty}{2}})$. 
	
	By abuse of notation, we still denote by $\overline{\mathrm{ev}}_\nu:{^{I_{\mathrm{u}}}_\nu\overline{\mathrm{Bun}}}_{N^-}\to N\backslash G/B^-$ the morphism obtained by composing the map ${^{1}_\nu\overline{\mathrm{Bun}}}_{N^-}\to  G/N^-$ (cf. \eqref{evaluation map}) with the canonical projection $G/N^-\to G/B^-$, and then quotienting by $N$. For $w\in W$, we let 
 $$N_w:=N\backslash(N\cdot w\cdot B^-)/B^-$$ 
 be the Schubert cell associated with $w$. For $x\in W_{\mathrm{ext}}$, written as $x=wt_\nu$, we let ${^{I_{\mathrm{u}}}_{x}\overline{\mathrm{Bun}}}_{N^-}$ be the preimage of $N_w$ through $\overline{\mathrm{ev}}_\nu$, and $\overline{i}_x:{^{I_{\mathrm{u}}}_x\overline{\mathrm{Bun}}}_{N^-}\hookrightarrow{^{I_{\mathrm{u}}}_\infty\overline{\mathrm{Bun}}}_{N^-}$ denote the locally closed inclusion. Let $\mathrm{IC}_{w,G/B^-},~\tilde{\nabla}_w,~\tilde{\Delta}_w\in \mathrm{Perv}_N(G/B^-)$, denote respectively the Goresky-MacPherson extension of the constant sheaf on $N_w$, the costandard (i.e. $!$-extension of the constant sheaf) and standard (i.e. $*$-extension of the constant sheaf) objects associated with $w$.  Thanks to Proposition \ref{simple objects prop}, we see that the simple objects of $\mathrm{Perv}_{I_{\mathrm{u}}}(\mathcal{F}l^{\frac{\infty}{2}})$ consist of the intersection cohomology sheaves $\mathrm{IC}_x^{\frac{\infty}{2}}\in\mathrm{Perv}(^{I_{\mathrm{u}}}_\infty\overline{\mathrm{Bun}}_{N^-})$ associated with  $^{I_{\mathrm{u}}}_x\overline{\mathrm{Bun}}_{N^-}$, for $x\in W_{\mathrm{ext}}$. More precisely, we have
	$$\mathrm{IC}_x^{\frac{\infty}{2}}:=(\overline{i}_\nu)_{!*}\left(\mathrm{IC}_{w,G/B^-}\widetilde{\boxtimes}\mathrm{IC}_{{_\nu\overline{\mathrm{Bun}}}_{N^-}}\right),\quad x=wt_\nu.$$
 Consider the functor $\mathrm{F}:{^{I_{\mathrm{u}}}_\infty\overline{\mathrm{Bun}}}_{N^-}\to {_\infty\overline{\mathrm{Bun}}}_{N^-}$
 induced by forgetting the structure of level $1$, and put  
$$\mathrm{For}^{G[[t]]}_{I_\mathrm{u}}:=\mathrm{F}^*[\mathrm{dim}(G)-\mathrm{dim}(N)]:\mathrm{D}({_\infty\overline{\mathrm{Bun}}}_{N^-})\to\mathrm{D}({^{I_{\mathrm{u}}}_\infty\overline{\mathrm{Bun}}}_{N^-}).$$
Similarly, one defines the functor
$$\mathrm{For}^{G[[t]]}_{I}:\mathrm{D}({_\infty\overline{\mathrm{Bun}}}_{N^-})\to\mathrm{D}({^{I}_\infty\overline{\mathrm{Bun}}}_{N^-}). $$
The next result will be used profusely in the sequel.
\begin{prop}\label{prop forget}
    The functors $\mathrm{For}^{G[[t]]}_{I_\mathrm{u}}, \mathrm{For}^{G[[t]]}_{I}$ induce fully faithful functors
    $$\mathrm{For}^{G[[t]]}_{I_\mathrm{u}}:\mathrm{Perv}_{G[[t]]}(\mathcal{F}l^{\frac{\infty}{2}})\to \mathrm{Perv}_{I_\mathrm{u}}(\mathcal{F}l^{\frac{\infty}{2}}),\quad \mathrm{For}^{G[[t]]}_{I}:\mathrm{Perv}_{G[[t]]}(\mathcal{F}l^{\frac{\infty}{2}})\to \mathrm{Perv}_{I}(\mathcal{F}l^{\frac{\infty}{2}})$$
    which satisfy
    $$\mathrm{For}^{G[[t]]}_{I_\mathrm{u}}(\mathcal{IC}_\nu^\frac{\infty}{2})\simeq \mathrm{IC}_\nu^\frac{\infty}{2},\quad \mathrm{For}^{G[[t]]}_{I}(\mathcal{IC}_\nu^\frac{\infty}{2})\simeq \mathrm{IC}_\nu^\frac{\infty}{2}$$
    for all $\nu\in\mathbf{Y}$, where $\mathrm{IC}_\nu^\frac{\infty}{2}$ is seen as an object of $\mathrm{Perv}_{I}(\mathcal{F}l^{\frac{\infty}{2}})$ in the second isomorphism.
\end{prop}
\begin{proof}
    The morphism $\mathrm{F}$ is induced by the morphism $\mathrm{For}:{^1_\infty\overline{\mathrm{Bun}}}_{N^-}\to {_\infty\overline{\mathrm{Bun}}}_{N^-}$ forgetting the structure of level one. The latter is a $G$-torsor, so that $\mathrm{F}$ is smooth of relative dimension $\mathrm{dim}(G)-\mathrm{dim}(N)$. Since
    $$F^{-1}({_\nu\overline{\mathrm{Bun}}}_{N^-})\simeq {^{I_\mathrm{u}}_\nu\overline{\mathrm{Bun}}}_{N^-}$$
    for all $\nu\in\mathbf{Y}$ we have an isomorphism
    $$\mathrm{For}^{G[[t]]}_{I_\mathrm{u}}(\mathcal{IC}_\nu^\frac{\infty}{2})\simeq \mathrm{IC}({^{I_\mathrm{u}}_\nu\overline{\mathrm{Bun}}}_{N^-}).$$
    But $\mathrm{IC}({^{I_\mathrm{u}}_\nu\overline{\mathrm{Bun}}}_{N^-})\simeq\mathrm{IC}_\nu^\frac{\infty}{2}$ since the $N$-orbit of the trivial element is dense in $G/B^-$. The full faithfulness then follows from the semi-simplicity of $\mathrm{Perv}_{G[[t]]}(\mathcal{F}l^{\frac{\infty}{2}})$. The proof of the assertions concerning $\mathrm{For}^{G[[t]]}_{I}$ is similar.
\end{proof}
	We also let $\nabla_x^{\frac{\infty}{2}}$,~$\Delta_x^{\frac{\infty}{2}}$ denote the corresponding costandard and standard objects. Namely we have
	$$\nabla_x^{\frac{\infty}{2}}:= (\overline{i}_\nu)_{*}\left(\tilde{\nabla}_w\widetilde{\boxtimes}\mathrm{IC}_{{_\nu\overline{\mathrm{Bun}}}_{N^-}}\right),\quad  \Delta_x^{\frac{\infty}{2}}:= (\overline{i}_\nu)_{!}\left(\tilde{\Delta}_w\widetilde{\boxtimes}\mathrm{IC}_{{_\nu\overline{\mathrm{Bun}}}_{N^-}}\right).$$
Note that the objects $\nabla_x^{\frac{\infty}{2}}$,~$\Delta_x^{\frac{\infty}{2}}$ admit canonical lifts to the category $\mathrm{D}({^{I}_\infty\overline{\mathrm{Bun}}}_{N^-})$. We will denote the corresponding objects by $\nabla_x^{\frac{\infty}{2}}$,~$\Delta_x^{\frac{\infty}{2}}\in \mathrm{D}({^{I}_\infty\overline{\mathrm{Bun}}}_{N^-})$ for all $x\in W_\mathrm{ext}$. Thus we have 
$$\mathrm{For}^I_{I_\mathrm{u}}(\Delta_x^{\frac{\infty}{2}})\simeq \Delta_x^{\frac{\infty}{2}},\quad \mathrm{For}^I_{I_\mathrm{u}}(\nabla_x^{\frac{\infty}{2}})\simeq \nabla_x^{\frac{\infty}{2}} $$
(where $\mathrm{For}^I_{I_\mathrm{u}}:\mathrm{D}({^{I}_\infty\overline{\mathrm{Bun}}}_{N^-})\to \mathrm{D}({^{I_\mathrm{u}}_\infty\overline{\mathrm{Bun}}}_{N^-})$ denotes the forgetful functor).
 
By construction we have isomorphisms in $\mathrm{D}({^{I_\mathrm{u}}_\infty\overline{\mathrm{Bun}}}_{N^-})$:
$$\nabla_x^{\frac{\infty}{2}}\simeq (\overline{i}_x)_{*}(\mathrm{IC}_{{^{I_{\mathrm{u}}}_x\overline{\mathrm{Bun}}}_{N^-}}),\quad \Delta_x^{\frac{\infty}{2}}\simeq (\overline{i}_x)_{!}(\mathrm{IC}_{{^{I_{\mathrm{u}}}_x\overline{\mathrm{Bun}}}_{N^-}}).$$ 
Moreover, thanks to the proof of \cite[Corollary 4.4.5]{ABBGM}, we know that the morphism $\overline{i}_x$ is affine,  so that $\Delta_x$ and $\nabla_x$ are perverse. We can then reuse the arguments of \cite[Corollary 4.4.7]{ABBGM} to prove the following result.
	\begin{prop}\label{prop ext}
		For all $x,x'\in W_{\mathrm{ext}},~i\in\{0,1,2\}$, we have
		$$\mathrm{Ext}^i_{\mathrm{Perv}_{I_{\mathrm{u}}}(\mathcal{F}l^{\frac{\infty}{2}})}(\Delta_x^{\frac{\infty}{2}},\nabla_{x'}^{\frac{\infty}{2}})=\begin{cases} \mathbb{k}\quad\text{if}~i=0~\text{and}~x=x',\\0\quad\text{otherwise}.
		\end{cases}$$
	\end{prop}
 \begin{proof}
Let $\nu\in\mathbf{Y}$, $\mathcal{F}'\in \mathrm{Perv}_{G^1}({^1_\nu\overline{\mathrm{Bun}}_{N^-}})$, and $\mathcal{F}\in \mathrm{Perv}_{G^1}({^1_\infty\overline{\mathrm{Bun}}_{N^-}})$.  We first claim that, if $(\overline{i}_\nu)_!(\mathcal{F}')$ is a perverse sheaf, then the canonical morphism
     \begin{equation}\label{eq ext 1}
    \mathrm{Ext}^1_{\mathrm{D}_T({^1_\infty\overline{\mathrm{Bun}}_{N^-}})}((\overline{i}_\nu)_!(\mathcal{F}'),\mathcal{F})\to  \mathrm{Ext}^1_{\mathrm{D}_T({^1_\infty\overline{\mathrm{Bun}}_{N^-}})}((i_\nu)_!(\mathcal{F}'|_{{^1_\nu\mathrm{Bun}}_{N^-}}),\mathcal{F})
     \end{equation}
     is an isomorphism. This follows from the same arguments as in \cite[Proposition 4.2.14]{ABBGM}. The five lemma then implies that the analogue morphism \eqref{eq ext 1} for $\mathrm{Ext}^2$ is injective.

     But we have:
     \begin{align*}
         \mathrm{Ext}^i_{\mathrm{D}_T({^{I_\mathrm{u}}_\infty\overline{\mathrm{Bun}}_{N^-}})}((i_x)_!(\underline{\mathbb{k}}[\ell(x)]),\nabla_{x'}^{\frac{\infty}{2}})&\simeq  \mathrm{Ext}^i_{\mathrm{D}_T({^{I_\mathrm{u}}_x\mathrm{Bun}}_{N^-}})(\underline{\mathbb{k}}[\ell(x)],(i_x)^!\nabla_{x'}^{\frac{\infty}{2}})\\
         &\simeq H^{i-\ell(x)}_T({^{I_\mathrm{u}}_x\mathrm{Bun}}_{N^-},(i_x)^!\nabla_{x'}^{\frac{\infty}{2}}).\\
         &\simeq \begin{cases} \mathbb{k}\quad\text{if}~i=0~\text{and}~x=x',\\0\quad\text{otherwise},\end{cases}
     \end{align*}
     where the last isomorphism is due to the fact that, by definition of ${^{I_\mathrm{u}}_x\mathrm{Bun}}_{N^-}$ together with Proposition \ref{prop contractible}, the $k$-th cohomology group of the quotient $T\backslash{^{I_\mathrm{u}}_x\mathrm{Bun}}_{N^-}$ vanishes is when  $k\neq 0$, and is isomorphic to $\mathbb{k}$ otherwise. 
 \end{proof}
 \begin{coro}\label{coro simple socle IC}
     For any $w\in W_\mathrm{ext}$, the object $\nabla^{\frac{\infty}{2}}_w$ has a simple socle, isomorphic to $\mathrm{IC}_w^{\frac{\infty}{2}}$.
 \end{coro}
 \begin{proof}
     Let $w'\in W_\mathrm{ext}$ be such that we have a non-zero morphism $\mathrm{IC}_{w'}^{\frac{\infty}{2}}\to \nabla^{\frac{\infty}{2}}_w$. Composing this morphism with the surjective morphism $\Delta^{\frac{\infty}{2}}_{w'}\to \mathrm{IC}_{w'}^{\frac{\infty}{2}}$, we obtain a non-zero morphism $\Delta^{\frac{\infty}{2}}_{w'}\to \nabla^{\frac{\infty}{2}}_w$. By Proposition \ref{prop ext}, we must have $w'=w$.
 \end{proof}
 \begin{rem}\label{rem construct}
     The same comments as in \cite[Remark 4.4.7]{ABBGM} apply here. Namely, one deduces from Proposition \ref{prop ext} that the Ext group between a standard and costandard object vanishes for all $i>0$. Moreover, for any $\mathcal{F}\in \mathrm{Perv}_{I_\mathrm{u}}(\mathcal{F}l^{\frac{\infty}{2}})$, we have an isomorphism
     $$\mathrm{Ext}^i_{\mathrm{Perv}_{I_{\mathrm{u}}}(\mathcal{F}l^{\frac{\infty}{2}})}(\Delta_x^{\frac{\infty}{2}},\mathcal{F})\simeq H^i({^{I_\mathrm{u}}_x\mathrm{Bun}}_{N^-},i_x^!\mathcal{F}),$$
     and one deduces from the second point of Proposition \ref{simple objects prop} that the cohomology groups of $i_x^!\mathcal{F}$ are local systems. Since ${^{I_\mathrm{u}}_x\mathrm{Bun}}_{N^-}$ is contractible according to \textit{loc. cit.}, those cohomology groups are in fact constant sheaves. 
 \end{rem}
	\begin{rem}\label{rem conv flag}   
		Notice that the group $G$ acts on ${^1_\infty\overline{\mathrm{Bun}}}_{N^-}$. This allows us to define a convolution product 
		$$(-)\star(-):\mathrm{D}_{N}(G/B)\times\mathrm{D}({^{I}_\infty\overline{\mathrm{Bun}}}_{N^-})\to \mathrm{D}({^{I_{\mathrm{u}}}_\infty\overline{\mathrm{Bun}}}_{N^-}). $$
		Namely, for $(\mathcal{F},\mathcal{G})\in \mathrm{D}_{N}(G/B)\times\mathrm{D}({^{I}_\infty\overline{\mathrm{Bun}}}_{N^-})$, one can form the twisted external tensor product $\mathcal{F}\widetilde{\boxtimes}\mathcal{G}\in\mathrm{D}_{N}(G\times^{B}{^{1}_\infty\overline{\mathrm{Bun}}}_{N^-})$, and define $\mathcal{F}\star\mathcal{G}$ to be the pushforward of $\mathcal{F}\widetilde{\boxtimes}\mathcal{G}$ through the morphism $G\times^{B}{^{1}_\infty\overline{\mathrm{Bun}}}_{N^-}\to {^{1}_\infty\overline{\mathrm{Bun}}}_{N^-}$ induced by the action of $G$. One can then easily check that, for all $w_1,w_2\in W$, $\lambda\in\mathbf{Y}$ such that $\ell(w_1w_2w_0)=\ell(w_1)+\ell(w_2w_0)$, we have 
		\begin{equation}\label{add length}
			\mathcal{N}_{w_1}\star\nabla_{w_2t_\lambda}^{\frac{\infty}{2}}\simeq \nabla_{w_1w_2t_\lambda}^{\frac{\infty}{2}}\qquad  \mathcal{D}_{w_1}\star\Delta_{w_2t_\lambda}^{\frac{\infty}{2}}\simeq \Delta_{w_1w_2t_\lambda}^{\frac{\infty}{2}},
		\end{equation}
  where $\mathcal{N}_{w_1}$ , resp. $\mathcal{D}_{w_1}$, is seen as the costandard object (resp. standard object) associated with $w$ in $\mathrm{D}_N(G/B)$ through the fully faithful functor $\mathrm{D}_N(G/B)\to \mathrm{D}_{I_\mathrm{u}}(\mathrm{Fl})$ induced by pullback through the canonical projection $\mathrm{Fl}\to G/B$.
	\end{rem}
	\section{Convolution}
	\subsection{Global affine Grassmannian }
	Recall (§\ref{section The affine Grassmannian and the affine flag variety}) that the affine Grassmaninan $\mathrm{Gr}$ is an ind-scheme which can be defined as the fppf-quotient $G((t))/G[[t]]$. From this definition, one can show that for any test $\mathbb{F}$-algebra $R$, $\mathrm{Gr}(\mathrm{Spec}(R))$ is the set of isomorphism classes of pairs $(\mathcal{F}_G,\beta)$, where $\mathcal{F}_G$ is a $G$-torsor over $\mathrm{Spec}(R[[t]])$ and $\beta$ a trivialization over $\mathrm{Spec}(R((t)))$.  We will need the global interpretation of $\mathrm{Gr}$ -- due to Beauville and Laszlo \cite{BL} -- which can be found for instance in \cite[§1.4]{zhu2016introduction}. Namely, the result of \textit{loc. cit.} shows that $\mathrm{Gr}$ classifies the data of $(\mathcal{F}_G,\beta)$, where $\mathcal{F}_G\in\mathrm{Bun}_G(S)$ and $\beta$ is a trivialization on $(X-\{x\})\times S$. 
	
	For each coweight $\lambda\in\mathbf{Y}$, one can give the following description of the Schubert variety $\overline{\mathrm{Gr}}^\lambda$, as in \cite[§3.2.1]{BG}. It is the closed subscheme of $\mathrm{Gr}$ classifying those $(\mathcal{F}_G,\beta)$ such that, for each $\mu\in\mathbf{X}^+$ and $G$-module $V$ whose weights are $\leq\mu$, we have
	$$\mathcal{V}_{\mathcal{F}^0_G}(-\langle\lambda,\mu\rangle\cdot(x\times S))\subset \mathcal{V}_{\mathcal{F}_G} $$
	(or in other words, the meromorphic map $\beta_\mathcal{V}:\mathcal{V}_{\mathcal{F}^0_G}\to \mathcal{V}_{\mathcal{F}_G}$ induced by $\beta$ has a pole of order at most $\langle\lambda,\mu\rangle$ at $x$). Notice that, by passing to the dual representation, the above condition is equivalent to $\mathcal{V}_{\mathcal{F}_G}\subset \mathcal{V}_{\mathcal{F}^0_G}(-\langle w_0(\lambda),\mu\rangle\cdot(x\times S))$ for all $V$. Each Schubert variety is a finite dimensional projective scheme which is stable under the $G[[t]]$-action. For all $\lambda,\lambda'\in \mathbf{Y}$, we have
	$$\overline{\mathrm{Gr}}^\lambda\subset\overline{\mathrm{Gr}}^{\lambda'}\Longleftrightarrow \lambda\leq\lambda' $$
	and we put $\mathrm{Gr}^\lambda:=\overline{\mathrm{Gr}}^\lambda-(\cup_{\mu<\lambda}\overline{\mathrm{Gr}}^\mu)$. Then $\mathrm{Gr}^\lambda(\mathbb{F})$ coincides with $G[[t]]\cdot [\lambda]$.
	\subsection{Hecke stacks} Next, we denote by $\mathcal{H}_{G}$ the Hecke stack for $G$ at $x$. It classifies the data $(\mathcal{F}_1,\mathcal{F}_2,\beta)$ of two bundles $\mathcal{F}_1,\mathcal{F}_2\in\mathrm{Bun}_G(S)$ together with an isomorphism $$\beta:\mathcal{F}_1|_{(X-\{x\})\times S}\simeq \mathcal{F}_2|_{(X-\{x\})\times S}.$$
	Recall that $^{1}\mathrm{Bun}_G$ was introduced at the end of §\ref{section Drinfeld's compactification and variants}, and let $^{I_\mathrm{u}}\mathrm{Bun}_G$ denote the quotient stack of $^{1}\mathrm{Bun}_G$ by $N$. For each coweight $\lambda\in\mathbf{Y}$, we will denote by $\overline{\mathcal{H}}_{G}^\lambda$ the closed substack of $\mathcal{H}_{G}$ classifying those $(\mathcal{F}_1,\mathcal{F}_2,\beta)$ such that, for each $\mu\in\Lambda_+$ and $G$-module $V$ whose weights are $\leq\mu$, we have
	$$\mathcal{V}_{\mathcal{F}_1}(-\langle\lambda,\mu\rangle\cdot(x\times S))\subset \mathcal{V}_{\mathcal{F}_2}\subset  \mathcal{V}_{\mathcal{F}_1}(-\langle w_0(\lambda),\mu\rangle\cdot(x\times S)).$$
	As above, the second inclusion is implied by the first one.
	
	Consider the diagram
	$$\mathrm{Bun}_G\xleftarrow{\overset{\leftarrow}{h}}\mathcal{H}_{G}\xrightarrow{\overset{\rightarrow}{h}}\mathrm{Bun}_G $$
	defined by $\overset{\leftarrow}{h}(\mathcal{F}_1,\mathcal{F}_2)=\mathcal{F}_1,\overset{\rightarrow}{h}(\mathcal{F}_1,\mathcal{F}_2)=\mathcal{F}_2$. Let $\mathfrak{G}$ denote the canonical $G[[t]]$-torsor over $\mathrm{Bun}_G$ associated with $x$. By definition, $\mathfrak{G}$ classifies the data of a $G$-bundle $\mathcal{F}_G$ over $X$ together with a trivialization $\mathcal{F}_G|_{D_x}\simeq\mathcal{F}^0_G|_{D_x}$. The projection $\overset{\leftarrow}{h}$ (and also $\overset{\rightarrow}{h}$) realizes $\mathcal{H}_{G}$ as a fibration over $\mathrm{Bun}_G$ with typical fiber $\mathrm{Gr}$ (cf. \cite[§3.2.4]{BG}):
	$$\mathcal{H}_{G}\simeq \mathrm{Gr}\times^{G[[t]]}\mathfrak{G}, $$
 where $G[[t]]$ acts on the right on $\mathfrak{G}$, and by left translation on $\mathrm{Gr}$.
 
	Under this identification we have, $\overline{\mathcal{H}}_{G}^\lambda\simeq \overline{\mathrm{Gr}}^\lambda\times^{G[[t]]}\mathfrak{G}$ (and $\overline{\mathcal{H}}_{G}^{-w_0(\lambda)}\simeq \overline{\mathrm{Gr}}^\lambda\times^{G[[t]]}\mathfrak{G}$ if we project via $\overset{\rightarrow}{h}$). 
	
	For $k_1,k_2\in\{0,1\}$, we let $^{k_1,k_2}\mathcal{H}_{G}$ denote the base change of $\mathcal{H}_{G}$ through the morphism
	$$^{k_1}\mathrm{Bun}_G\times{^{k_2}\mathrm{Bun}_G }\to \mathrm{Bun}_G\times\mathrm{Bun}_G $$
 (where $^{0}\mathrm{Bun}_G:=\mathrm{Bun}_G$). We still denote by $\overset{\leftarrow}{h}$ and $\overset{\rightarrow}{h}$ the canonical projections from $^{k_1,k_2}\mathcal{H}_{G}$ to $^{k_1}\mathrm{Bun}_G$ and ${^{k_2}\mathrm{Bun}_G }$ respectively. If we let $\mathfrak{G}^k_x$ denote the canonical $G^k$-torsor over $^k\mathrm{Bun}_G$ associated with $x$ for $k\in\{0,1\}$, then the projection $\overset{\rightarrow}{h}$ realizes $^{k_1,k_2}\mathcal{H}_{G}$ as a fibration over ${^{k_2}\mathrm{Bun}_G }$ with typical fiber $G((t))/G^{k_1}$ (\cite[§5.1.1]{ABBGM}). This means that we get a canonical isomorphism
	\begin{equation}\label{iden 1}
		^{k_1,k_2}\mathcal{H}_{G}\simeq G((t))/G^{k_1}\times^{G^{k_2}}\mathfrak{G}^{k_2},
	\end{equation}
where $\mathfrak{G}^0:=\mathfrak{G}$. Similarly, let us denote by $^{I_{\mathrm{u}},0}\mathcal{H}_{G}$ the Hecke stack obtained by replacing $^{k_1}\mathrm{Bun}_G$ with $^{I_{\mathrm{u}}}\mathrm{Bun}_G$ in the previous construction, and letting $k_2=0$. Via $\overset{\rightarrow}{h}$, we then get a canonical isomorphism
	\begin{equation}\label{iden 2}
		^{I_{\mathrm{u}},0}\mathcal{H}_{G}\simeq G((t))/I_{\mathrm{u}}\times^{G[[t]]}\mathfrak{G}.
	\end{equation}
	Moreover, we obtain the following isomorphism via $\overset{\leftarrow}{h}$:
	\begin{equation}\label{iden 3}
		^{I_{\mathrm{u}},0}\mathcal{H}_{G}\simeq \mathrm{Gr}\times^{G[[t]]}{^{I_\mathrm{u}}\mathfrak{G}},
	\end{equation}
	where ${^{I_\mathrm{u}}\mathfrak{G}}$ is the canonical $G[[t]]$-torsor over $^{I_{\mathrm{u}}}\mathrm{Bun}_G$.
	
	We then consider the following diagram, in which both squares are Cartesian
	\begin{equation*}
		\xymatrix{
			^{k_1}_\infty\overline{\mathrm{Bun}}_{N^-}\ar[d]_{\mathfrak{p}} &^{k_1,k_2}\mathcal{H}_{G,N^{-}}\ar[l]^{\overset{\leftarrow}{h'}} \ar[r]^{\overset{\rightarrow}{h'}}\ar[d] & ^{k_2}_\infty\overline{\mathrm{Bun}}_{N^-}\ar[d]_{\mathfrak{p}} \\
			^{k_1}\mathrm{Bun}_{G}&^{k_1,k_2}\mathcal{H}_{G}\ar[l]^{\overset{\leftarrow}{h}} \ar[r]^{\overset{\rightarrow}{h}}& ^{k_2}\mathrm{Bun}_{G}.
		}
	\end{equation*}
	Similarly we have the following diagram
	\begin{equation*}
		\xymatrix{
			^{I_{\mathrm{u}}}_\infty\overline{\mathrm{Bun}}_{N^-}\ar[d]_{\mathfrak{p}} &^{I_{\mathrm{u}},0}\mathcal{H}_{G,N^{-}}\ar[l]^{\overset{\leftarrow}{h'}} \ar[r]^{\overset{\rightarrow}{h'}}\ar[d] & _\infty\overline{\mathrm{Bun}}_{N^-}\ar[d]_{\mathfrak{p}} \\
			^{I_{\mathrm{u}}}\mathrm{Bun}_{G}&^{I_{\mathrm{u}},0}\mathcal{H}_{G}\ar[l]^{\overset{\leftarrow}{h}} \ar[r]^{\overset{\rightarrow}{h}}& \mathrm{Bun}_{G}.
		}
	\end{equation*}
	Now, the previous diagrams together with the identifications \eqref{iden 1} and \eqref{iden 2} yield the following canonical isomorphisms:
	\begin{equation}\label{ident conv prod}
		^{0,0}\mathcal{H}_{G,N^{-}}\simeq G((t))/G[[t]]\times^{G[[t]]}\widetilde{\mathfrak{G}},\qquad ^{I_{\mathrm{u}},0}\mathcal{H}_{G,N^{-}}\simeq G((t))/I_{\mathrm{u}}\times^{G[[t]]}\widetilde{\mathfrak{G}}, 
	\end{equation}
	where $\widetilde{\mathfrak{G}}$ denotes the canonical $G[[t]]$-torsor over $_\infty\overline{\mathrm{Bun}}_{N^-}$. The identification \eqref{iden 3} also yield the isomorphism
	$$^{I_{\mathrm{u}},0}\mathcal{H}_{G,N^{-}}\simeq \mathrm{Gr}\times^{G[[t]]}{^{I_\mathrm{u}}\widetilde{\mathfrak{G}}}, $$
where ${^{I_\mathrm{u}}\widetilde{\mathfrak{G}}}$ denotes the canonical $G[[t]]$-torsor over $^{I_\mathrm{u}}_\infty\overline{\mathrm{Bun}}_{N^-}$.	For any $\lambda\in \mathbf{Y}$, we define the locally closed substack 
$$^{I_{\mathrm{u}},0}\mathcal{H}_{G,N^{-}}^\lambda:=\mathrm{Gr}^\lambda\times^{G[[t]]}{^{I_\mathrm{u}}\widetilde{\mathfrak{G}}}.$$
	\subsection{Definition of convolution}\label{section Definition of convolution}  
	
	Recall (\cite[§5.1.1]{ABBGM}) that the inversion map $\mathrm{sw}:G((t))\to G((t)),~g\mapsto g^{^-1}$ induces an equivalence of categories
	$$\mathcal{S}\mapsto\mathrm{sw}^*(\mathcal{S}):\mathrm{D}_{G^{k_1}}(G((t))/G^{k_2})\to \mathrm{D}_{G^{k_2}}(G((t))/G^{k_1}). $$
	
	We are now ready to define the convolution product for any $\mathcal{S}\in \mathrm{D}_{G[[t]]}(\mathrm{Gr})$:
	$$\mathcal{F}\mapsto \mathcal{S}\star\mathcal{F}: \mathrm{D}(_\infty\overline{\mathrm{Bun}}_{N^-})\to \mathrm{D}(_\infty\overline{\mathrm{Bun}}_{N^-}),~\mathcal{F}\mapsto (\overset{\leftarrow}{h'})_*(\mathrm{sw}^*(\mathcal{S})\widetilde{\boxtimes}\mathcal{F}),$$
	where $\mathrm{sw}^*(\mathcal{S})\widetilde{\boxtimes}\mathcal{F}$ is seen as an object of $\mathrm{D}(^{0,0}\mathcal{H}_{G,N^{-}})$ thanks to \eqref{ident conv prod}. And respectively for any $\mathcal{S}\in \mathrm{D}_{I_\mathrm{u}}(\mathrm{Gr})$:
	$$\mathcal{F}\mapsto \mathcal{S}\star\mathcal{F}:\mathrm{D}({_\infty\overline{\mathrm{Bun}}_{N^-}})\to \mathrm{D}(^{I_\mathrm{u}}_\infty\overline{\mathrm{Bun}}_{N^-}),~\mathcal{F}\mapsto (\overset{\leftarrow}{h'})_*(\mathrm{sw}^*(\mathcal{S})\widetilde{\boxtimes}\mathcal{F}).$$
	Similarly, one can define convolution functors
	\begin{align*}
 &\mathrm{D}_I(\mathrm{Gr})\times\mathrm{D}({_\infty\overline{\mathrm{Bun}}}_{N^-})\to  \mathrm{D}({^{I}_\infty\overline{\mathrm{Bun}}}_{N^-}),~(\mathcal{S},\mathcal{F})\mapsto\mathcal{S}\star\mathcal{F},\\
	    &\mathrm{D}_{I_\mathrm{u}}(\mathrm{Fl})\times\mathrm{D}({^{I_\mathrm{u}}_\infty\overline{\mathrm{Bun}}}_{N^-})\to  \mathrm{D}({^{I_\mathrm{u}}_\infty\overline{\mathrm{Bun}}}_{N^-}),~(\mathcal{S},\mathcal{F})\mapsto\mathcal{S}\star\mathcal{F},\\
     &\mathrm{D}_{I}(\mathrm{Fl})\times\mathrm{D}({^{I}_\infty\overline{\mathrm{Bun}}}_{N^-})\to  \mathrm{D}({^{I}_\infty\overline{\mathrm{Bun}}}_{N^-}),~(\mathcal{S},\mathcal{F})\mapsto\mathcal{S}\star\mathcal{F}.
	\end{align*}
The next result follows easily from the constructions.
\begin{lem}\label{lem conv for}
    For any $\mathcal{F}\in\mathrm{D}_{I}(\mathrm{Fl})$, $\mathcal{G}\in\mathrm{D}({_\infty\overline{\mathrm{Bun}}}_{N^-})$, there exists a canonical isomorphism in $\mathrm{D}({^{I}_\infty\overline{\mathrm{Bun}}}_{N^-})$:
    $$\mathcal{F}\star\mathrm{For}_I^{G[[t]]}(\mathcal{G})\simeq \pi_*(\mathcal{F})\star\mathcal{G}, $$
    where we recall that $\pi:\mathrm{Fl}\to\mathrm{Gr}$ is the canonical projection.
\end{lem}
	In the cases we consider, the convolution product commutes with Verdier duality.
	\begin{prop}
		The morphism $\overset{\leftarrow}{h'}$ from $^{0,0}\mathcal{H}_{G,N^{-}}$ (resp. $^{I_{\mathrm{u}},0}\mathcal{H}_{G,N^{-}}$) to $_\infty\overline{\mathrm{Bun}}_{N^-}$ (resp. $^{I_{\mathrm{u}}}_\infty\overline{\mathrm{Bun}}_{N^-}$)  is ind-proper.
	\end{prop}
	\begin{proof}
		This follows from the fact that $\overset{\leftarrow}{h'}$ realizes $^{0,0}\mathcal{H}_{G,N^{-}}$ (resp. $^{I_{\mathrm{u}},0}\mathcal{H}_{G,N^{-}}$) as a fibration over $_\infty\overline{\mathrm{Bun}}_{N^-}$ (resp. $^{I_{\mathrm{u}}}_\infty\overline{\mathrm{Bun}}_{N^-}$) with typical fiber an ind-projective ind-scheme (namely the affine Grassmannian in the first case, and the affine flag variety in the second case).
	\end{proof}
	\begin{coro}\label{verdier coro}
		For any $\mathcal{F}\in\mathrm{D}_{G[[t]]}(\mathcal{F}l^{\frac{\infty}{2}})$ and $\mathcal{S}\in \mathrm{D}_{G[[t]]}(\mathrm{Gr})$ (resp. $\mathcal{S}\in  \mathrm{D}_{I_\mathrm{u}}(\mathrm{Gr})$), we have 
		$$\mathbb{D}(\mathcal{S}\star\mathcal{F})=\mathbb{D}(\mathcal{S})\star\mathbb{D}(\mathcal{F}).$$
	\end{coro}
	\subsection{Exactness of convolution}
	We will explain the proof of the following key property.
	\begin{thm}{{\cite[Theorem 5.2.2]{ABBGM}}}\label{thm exact}
		For any $\nu\in\mathbf{Y}$, the convolution functor 
		$$\mathcal{S}\mapsto \mathcal{S}\star\mathcal{IC}^{\frac{\infty}{2}}_\nu $$
		from $\mathrm{D}_{G[[t]]}(\mathrm{Gr})$ (resp.   $\mathrm{D}_{I_\mathrm{u}}(\mathrm{Gr})$) to $\mathrm{D}_{G[[t]]}(\mathcal{F}l^{\frac{\infty}{2}})$ (resp. $\mathrm{D}_{I_\mathrm{u}}(\mathcal{F}l^{\frac{\infty}{2}})$) is $t$-exact.
	\end{thm}
	\begin{proof}
		Thanks to Corollary \ref{verdier coro}, it is enough to prove that the convolution functor is right $t$-exact. Let $\mathcal{S}\in {^p\mathrm{D}_{G[[t]]}(\mathrm{Gr})^{\leq0}}$ (resp. $\mathcal{S}\in {^p\mathrm{D}_{I_\mathrm{u}}(\mathrm{Gr})^{\leq0}}$) and $\mu\in \mathbf{Y}$. We need to prove that $\overline{i}_\mu^*(\mathcal{S}\star \mathcal{IC}^{\frac{\infty}{2}}_\nu ) $ lives in non-positive cohomological degrees. 
		
		The preimage $(\overset{\leftarrow}{h'})^{-1}({_\mu\overline{\mathrm{Bun}}_{N^-}})$ decomposes into the locally closed substacks
		$$X_{\mu',\lambda}:=(\overset{\leftarrow}{h'})^{-1}({_\mu\overline{\mathrm{Bun}}_{N^-}})\bigcap (\overset{\rightarrow}{h'})^{-1}({_{\mu'}\overline{\mathrm{Bun}}_{N^-}})\bigcap {^{?,0}\mathcal{H}^\lambda_{G,N^-}},\quad\mu'\in\mathbf{Y},\lambda\in\mathbf{Y}^+, $$
		where $?=0$ (resp. $?=I_\mathrm{u}$). In order to conclude, it will be enough to show that:
		\begin{enumerate}
			\item The dimension of the fibers of $\overset{\leftarrow}{h'}|_{X_{\mu',\lambda}}$ is bounded by $\langle \mu'-\mu+\lambda,\rho\rangle$. 
			\item The restriction of $\mathrm{sw}^*(\mathcal{S})\widetilde{\boxtimes}\mathrm{IC}_\nu$ to $X_{\mu',\lambda}$ lives in perverse cohomological degrees less than $ -\langle \mu'-\mu+\lambda,\rho\rangle$.
		\end{enumerate}
		For the first point, the argument is that $\overset{\leftarrow}{h'}|_{X_{\mu',\lambda}}$ realizes $X_{\mu',\lambda}$ as the following fibration over ${^?_\mu\overline{\mathrm{Bun}}_{N^-}}$:  
		$$(\mathrm{Gr}^{\lambda}\cap N^-((t))\cdot(\mu-\mu'))\times^{N^-[[t]]} {^?_{\mu}\mathcal{N}},$$
		where, in the case $?=I_\mathrm{u}$, the stack ${^{I_\mathrm{u}}_{\mu}\mathcal{N}}$ is the $N^-[[t]]$-torsor ${^{1}_{\mu}\mathcal{N}}/N$ over ${^{I_\mathrm{u}}_\mu\overline{\mathrm{Bun}}_{N^-}}$.
		
		Let us explain the proof of the second point.
		Let $p_?$ be the identity $\mathrm{Gr}\to\mathrm{Gr}$ when $?=0$, and the canonical projection $G((t))/I_{\mathrm{u}}\to\mathrm{Gr}$ when $?=I_\mathrm{u}$. Projecting via $\overset{\rightarrow}{h'}$, the locally closed substack $X_{\mu',\lambda}$  identifies with
		$$p_?^{-1}(\mathrm{Gr}^{-w_0(\lambda)}\cap N^-((t))\cdot(\mu'-\mu))\times^{N^-[[t]]} {_{\mu'}\mathcal{N}},$$
		so that the restriction of $\mathrm{sw}^*(\mathcal{S})\widetilde{\boxtimes}\mathrm{IC}_\nu$ to $X_{\mu',\lambda}$ identifies with
		$$\mathrm{sw}^*(\mathcal{S})|_{p_?^{-1}(\mathrm{Gr}^{-w_0(\lambda)}\cap N^-((t))\cdot(\mu'-\mu))}\widetilde{\boxtimes}\mathcal{IC}^{\frac{\infty}{2}}_\nu|_{{_{\mu'}\overline{\mathrm{Bun}}_{N^-}}}. $$
		It is a property of intersection cohomology that the complex $\mathcal{IC}^{\frac{\infty}{2}}_\nu|_{{_{\mu'}\overline{\mathrm{Bun}}_{N^-}}}$ lives in negative (and in particular non-positive) perverse cohomological degrees. So we only need to check that the restriction of $\mathrm{sw}^*(\mathcal{S})$ to ${p_?^{-1}(\mathrm{Gr}^{-w_0(\lambda)}\cap N^-((t))\cdot(\mu'-\mu))}$ lives in perverse cohomological degrees $\leq -\langle \mu'-\mu+\lambda,\rho\rangle$. Since $*$-restriction is right $t$-exact, the complex $\mathrm{sw}^*(\mathcal{S})|_{p_?^{-1}(\mathrm{Gr}^{-w_0(\lambda)})}$ lives in non-positive perverse cohomological degrees. When $?=0$, the complex $\mathrm{sw}^*(\mathcal{S})|_{\mathrm{Gr}^{-w_0(\lambda)}}$ is constant since it is $G[[t]]$-equivariant and $\mathrm{Gr}^{-w_0(\lambda)}$ is an orbit. So $\mathrm{sw}^*(\mathcal{S})|_{\mathrm{Gr}^{-w_0(\lambda)}\cap N^-((t))\cdot(\mu'-\mu)}$ lives in perverse cohomological degrees less than
		$$-\mathrm{codim}(\mathrm{Gr}^{-w_0(\lambda)}\cap N^-((t))\cdot(\mu'-\mu),\mathrm{Gr}^{-w_0(\lambda)})\leq -\langle \mu'-\mu+\lambda,\rho\rangle. $$
		Finally, assume that $?=I_\mathrm{u}$. Let $H$ be the stabilizer of $t^{-w_0(\lambda)}$ in $G[[t]]$ and  $n$ be a large enough integer such that we have an isomorphism $G_n/H\simeq \mathrm{Gr}^{-w_0(\lambda)}$ (where $G_n$ was defined in §\ref{section Group arc schemes}). Then we have an isomorphism of $\mathbb{F}$-schemes
		$$p_{I_\mathrm{u}}^{-1}(\mathrm{Gr}^{-w_0(\lambda)})\simeq  G_n\times^{H}G[[t]]/I_\mathrm{u}\simeq G_n\times^{H}G/U.$$
		Since $\mathrm{sw}^*(\mathcal{S})|_{p_{I_\mathrm{u}}^{-1}(\mathrm{Gr}^{-w_0(\lambda)})}$ is $G[[t]]$-equivariant, we deduce that it is of the form 
  $$\underline{\mathbb{k}}_{\mathrm{Gr}^{-w_0(\lambda)}}[d]\widetilde{\boxtimes} \mathcal{F}$$ 
  for some $\mathcal{F}\in\mathrm{D}(G/U)$ and integer $d$. Since $\mathrm{sw}^*(\mathcal{S})|_{p_?^{-1}(\mathrm{Gr}^{-w_0(\lambda)})}$ lives in non-positive perverse cohomological degrees, we must have $d\geq \mathrm{dim}(\mathrm{Gr}^{-w_0(\lambda)})$ and $\mathcal{F}\in{^p\mathrm{D}}(G/U)^{\leq0}$. The identification $$\mathrm{sw}^*(\mathcal{S})|_{p_{I_\mathrm{u}}^{-1}(\mathrm{Gr}^{-w_0(\lambda)}\cap N^-((t))\cdot(\mu'-\mu))}\simeq \underline{\mathbb{k}}_{\mathrm{Gr}^{-w_0(\lambda)}\cap N^-((t))\cdot(\mu'-\mu))}[d]\widetilde{\boxtimes} \mathcal{F}$$ then allows to conclude.
	\end{proof}
	\subsection{Convolution in the spherical case} From here and until the end, we assume that Conjecture \ref{conj stalks independance} holds.  The present subsection -- which is copied from \cite[§5.2.3]{ABBGM} -- settles some results which will be essential when defining the functor from the Hecke category to $\mathrm{Perv}_{I_\mathrm{u}}(\mathcal{F}l^\frac{\infty}{2})$. For any $V\in\mathrm{Rep}(G^\vee_\mathbb{k})$ and coweight $\mu\in\mathbf{Y}$, we will denote by $V(\mu)$ the associated $\mu$-weight space. For any morphism of $\mathbb{F}$-schemes of finite type $i:X\to Y$  and $\mathcal{F}\in\mathrm{D}(Y)$, we put
 $$H^k_Y(X,\mathcal{F}):=H^k(Y,i^!\mathcal{F}) $$
 for all $k\in \mathbb{Z}$. Recall that,  by \cite{Mirkovic2004GeometricLD}, we have a canonical isomorphism of complexes of vector spaces for any $\mathcal{G}\in\mathrm{Perv}_{G^\vee_\mathbb{k}[[t]]}(\mathrm{Gr})$
	\begin{equation}\label{can iso}
		H_{N^-((t))\cdot \mu}^\bullet (\mathrm{Gr} ,\mathcal{G})\simeq H_{N^-((t))\cdot \mu}^{\langle{2\rho,\mu\rangle}}(\mathrm{Gr},\mathcal{G})\simeq \mathrm{Sat}(\mathcal{G})(\mu).
	\end{equation}
	
	\begin{prop}{{\cite[Proposition 5.2.4]{ABBGM}}}\label{prop hecke geometric} For any $\mathcal{G}\in\mathrm{Perv}_{G[[t]]}(\mathrm{Gr})$ and $\nu\in\mathbf{Y}$, there exist canonical isomorphisms
		$$\mathcal{G}\star\mathcal{IC}^{\frac{\infty}{2}}_\nu\simeq\bigoplus_{\mu\in\mathbf{Y}}\mathcal{IC}^{\frac{\infty}{2}}_{\nu+\mu}\otimes_\mathbb{k}\mathrm{Sat}(\mathrm{sw}^*\mathcal{G})(-\mu)\simeq\bigoplus_{\mu\in\mathbf{Y}}\mathcal{IC}^{\frac{\infty}{2}}_{\nu+\mu}\otimes_\mathbb{k}\mathrm{Sat}(\mathcal{G})(w_0(\mu)). $$
		Moreover, for any $\mathcal{V},\mathcal{U}\in \mathrm{Perv}_{G[[t]]}(\mathrm{Gr})$, the following diagram commutes
		\begin{equation}\label{diagram prop hecke}
			\xymatrix{
				(\mathcal{U}\star\mathcal{V})\star\mathcal{IC}^{\frac{\infty}{2}}_\nu\ar[d]^{\sim} \ar[r]^{\sim}&\oplus_{\mu'}(\mathcal{U}\star\mathcal{IC}^{\frac{\infty}{2}}_{\nu+\mu'})\otimes \mathrm{Sat}(\mathcal{V})(\mu')\ar[d]^{\sim}  \\
				\oplus_{\mu}\mathcal{IC}^{\frac{\infty}{2}}_{\nu+\mu}\otimes(\mathrm{Sat}(\mathcal{U})\otimes \mathrm{Sat}(\mathcal{V}))(\mu')\ar[r]^-{\sim}&\oplus_{\mu'+\mu''}\mathcal{IC}^{\frac{\infty}{2}}_{\nu+\mu'+\mu''}\otimes \mathrm{Sat}(\mathcal{U})(\mu'')\otimes \mathrm{Sat}(\mathcal{V})(\mu').}
		\end{equation}
		
	\end{prop}

	\begin{proof}
		 Since the convolution product is $t$-exact (Theorem \ref{thm exact}) and the category $\mathrm{Perv}_{G[[t]]}(\mathcal{F}l^\frac{\infty}{2})$ is semi-simple (Theorem \ref{thm ss}), we know that there exist vector spaces $L_\mu$ such that 
     \begin{equation}\label{hom space}
			\mathcal{G}\star\mathcal{IC}^{\frac{\infty}{2}}_\nu\simeq \bigoplus_{\mu\in\mathbf{Y}}\mathcal{IC}^{\frac{\infty}{2}}_{\nu+\mu}\otimes_\mathbb{k}L_\mu.  
		\end{equation}
Thanks to Lemma \ref{lem stalks} (or more easily by the usual theory of perverse sheaves), we have an isomorphism of vector spaces	between $L_\mu$ and	
		\begin{equation}\label{hom space 2}
			\mathrm{Hom}(\mathrm{IC}_{_{\nu+\mu}\mathrm{Bun}_{N^-}},(i_{\nu+\mu})^!(\mathcal{G}\star\mathcal{IC}^{\frac{\infty}{2}}_\nu)). 
		\end{equation}
		Let $\lambda$ be a dominant cocharacter such that the support of $\mathrm{sw}^*(\mathcal{G})$ is included in $\overline{\mathrm{Gr}}^{-w_0(\lambda)}$. By the proper base change theorem, we have an isomorphism $$(i_{\nu+\mu})^!(\mathcal{G}\star\mathcal{IC}^{\frac{\infty}{2}}_\nu)\simeq (\overset{\leftarrow}{h'})_*(\tilde{i}_{\nu+\mu})^!(\mathrm{sw}^*(\mathcal{G})\widetilde{\boxtimes}\mathcal{IC}^{\frac{\infty}{2}}_\nu), $$
		where $\tilde{i}^{\nu+\mu}$ is the locally closed inclusion 
		$$Y:=(\overset{\leftarrow}{h'})^{-1}({_{\nu+\mu}\mathrm{Bun}}_{N^-})\bigcap (\overset{\rightarrow}{h'})^{-1}({_{\leq \nu}\overline{\mathrm{Bun}}_{N^-}})\bigcap {^{0,0}\overline{\mathcal{H}}^\lambda_{G,N^-}}\subset  (\overset{\rightarrow}{h'})^{-1}({_{\leq \nu}\overline{\mathrm{Bun}}_{N^-}})\bigcap {^{0,0}\overline{\mathcal{H}}^\lambda_{G,N^-}}$$
		(by abuse of notation, we still denote by $\overset{\leftarrow}{h'}$ the restriction of $\overset{\leftarrow}{h'}$ to $Y\subset {^{0,0}\mathcal{H}_{G,N^{-}}}$).
		
		For $\beta\leq \lambda$, $\beta\in\mathbf{Y}^+$, denote by $Y_\beta$ the locally closed substack of $Y$ defined by replacing ${^{0,0}\overline{\mathcal{H}}^\lambda_{G,N^-}}$ with ${^{0,0}\mathcal{H}}^\beta_{G,N^-}$ in the definition of $Y$, by $j_\beta:Y_\beta\hookrightarrow {^{0,0}\mathcal{H}_{G,N^{-}}}$ the locally closed inclusion, and by $(\overset{\leftarrow}{h'_\beta})_*$ the restriction of $(\overset{\leftarrow}{h'})_*$ to $Y_\beta$. Then we have the finite stratification $Y=\cup_\beta Y_\beta$, and projecting via $\overset{\rightarrow}{h'}$ induces an isomorphism
		\begin{equation*}\label{iden Y_b}
			Y_\beta\simeq \left(\mathrm{Gr}^{-w_0(\beta)}\cap N^-((t))\cdot (-\mu)\right)\times^{N^-[[t]]}{_\nu\overset{\circ}{\mathcal{N}}},
		\end{equation*}
		where ${_\nu\overset{\circ}{\mathcal{N}}}$ denotes the preimage of ${_{\nu}\mathrm{Bun}_{N^-}}$ under the canonical projection $_\nu\mathcal{N}\to {_{ \nu}\overline{\mathrm{Bun}}_{N^-}}$. Let us denote by $j:\mathrm{Gr}^{-w_0(\beta)}\cap N^-((t))\cdot (-\mu)\hookrightarrow\mathrm{Gr}$ the locally closed embedding. Now, notice that for all $\beta$, we have the isomorphism
		\begin{align*}
			(\overset{\leftarrow}{h'_\beta})_*(j_{\beta})^!(\mathrm{sw}^*(\mathcal{G})\widetilde{\boxtimes}\mathcal{IC}^{\frac{\infty}{2}}_\nu)
			&\simeq  (\overset{\leftarrow}{h'_\beta})_* (j^!(\mathrm{sw}^*(\mathcal{G}))\widetilde{\boxtimes} i_\nu^!\mathcal{IC}^{\frac{\infty}{2}}_\nu).
		\end{align*}
		By adjunction, we have the isomorphism 
		\begin{equation}\label{equ split}
			\begin{split}
				&\mathrm{Hom}(\mathrm{IC}_{_{\nu+\mu}\mathrm{Bun}_{N^-}},(\overset{\leftarrow}{h'_\beta})_* (j^!(\mathrm{sw}^*(\mathcal{G}))\widetilde{\boxtimes} i_\nu^!\mathcal{IC}^{\frac{\infty}{2}}_\nu))\\&\simeq \mathrm{Hom}(\mathrm (\overset{\leftarrow}{h'_\beta})^*(\mathrm{IC}_{_{\nu+\mu}\mathrm{Bun}_{N^-}}), (j^!(\mathrm{sw}^*(\mathcal{G}))\widetilde{\boxtimes} i_\nu^!\mathcal{IC}^{\frac{\infty}{2}}_\nu))
			\end{split}
		\end{equation}
		Since $_{\nu+\mu}\mathrm{Bun}_{N^-}$ is smooth, $\mathrm{IC}_{_{\nu+\mu}\mathrm{Bun}_{N^-}}$ is the constant perverse sheaf of rank one. Moreover the relative dimension of $_{\nu+\mu}\mathrm{Bun}_{N^-}$ compared with $_{\nu}\mathrm{Bun}_{N^-}$ is $2\langle\rho,\mu\rangle$. Therefore, the pullback $(\overset{\leftarrow}{h'_\beta})^*(\mathrm{IC}_{_{\nu+\mu}\mathrm{Bun}_{N^-}})$ is isomorphic to $\underline{\mathbb{k}}_{Y'_\beta}[2\langle\rho,\mu\rangle]\widetilde{\boxtimes}\mathrm{IC}_{_{\nu}\mathrm{Bun}_{N^-}}$, where we have put $Y'_\beta:=\mathrm{Gr}^{-w_0(\beta)}\cap N^-((t))\cdot (-\mu)$. Thus, \eqref{equ split} coincides with
		$$\mathrm{Hom}(\underline{\mathbb{k}}_{Y'_\beta}[2\langle\rho,\mu\rangle],j^!(\mathrm{sw}^*(\mathcal{G})))\simeq H_{Y'_\beta}^{-2\langle\rho,\mu\rangle}(\mathrm{Gr},\mathrm{sw}^*(\mathcal{G})). $$
		Now, using distinguished triangles associated with open and closed immersions together with an induction on the number of strata of $Y$, we deduce that \eqref{hom space 2} coincides with
		$$H_{N^{-}((t))\cdot(-\mu)}^{-2\langle\rho,\mu\rangle}(\mathrm{Gr},\mathrm{sw}^*(\mathcal{G})). $$
		Finally, the above vector space is canonically isomorphic to $\mathrm{Sat}(\mathrm{sw}^*(\mathcal{G}))(-\mu)$ thanks to \eqref{can iso}. By Proposition \ref{prop Chevalley inv}, the latter vector space is isomorphic to $\mathrm{Sat}(\mathcal{G})(w_0(\mu))$.

  The commutativity of the diagrams follows from the fact that the geometric Satake equivalence is a tensor functor.
	\end{proof}
 \subsection{Action of the regular perverse sheaf}\label{section Action of the regular perverse sheaf}
We define the $\mathbf{Y}$-graded object of $\mathrm{Perv}_{G[[t]]}(\mathcal{F}l^{\frac{\infty}{2}})$:
$$\mathcal{IC}^{\frac{\infty}{2}}:=\bigoplus_{\lambda\in\mathbf{Y}}\mathcal{IC}^{\frac{\infty}{2}}_\lambda.$$
For any $\lambda,\mu\in\mathbf{Y}$, we will define a morphism 
 $$H_{\mu,\lambda}:\mathcal{R}_{\mu}\star \mathcal{IC}^{\frac{\infty}{2}}_{\lambda}\to\mathcal{IC}^{\frac{\infty}{2}}_{\lambda+\mu}.  $$
 Recall (§\ref{section The regular perverse sheaf}) that $\mathcal{R}_\mu:=\varinjlim_{\delta}  \mathcal{I}^{w_0(\mu)+\delta}_*\star\mathcal{I}^{-w_0(\delta)}_*\in\mathrm{Ind}(\mathrm{Perv}_{G[[t]]}(\mathrm{Gr})).$ For any $\delta\in \mathbf{Y}^+\cap(-w_0(\mu)+\mathbf{Y}^+)$, we define
 \begin{align*}
     H_{\mu,\lambda}^\delta:(\mathcal{I}^{w_0(\mu)+\delta}_*\star\mathcal{I}^{-w_0(\delta)}_*)\star\mathcal{IC}^{\frac{\infty}{2}}_\lambda&\xrightarrow[\text{Prop.\ref{prop hecke geometric}}]{\sim}\bigoplus_{\mu'}\mathcal{I}^{w_0(\mu)+\delta}_*\star\mathcal{IC}^{\frac{\infty}{2}}_{\lambda+\mu'}\otimes_\mathbb{k}\mathrm{Sat}(\mathcal{I}^{-w_0(\delta)}_*)(w_0(\mu')) \\
     &\to\mathcal{I}^{w_0(\mu)+\delta}_*\star\mathcal{IC}^{\frac{\infty}{2}}_{\lambda-w_0(\delta)} \\
     &\xrightarrow[\text{Prop.\ref{prop hecke geometric}}]{\sim}\bigoplus_{\mu'}\mathcal{IC}^{\frac{\infty}{2}}_{\lambda-w_0(\delta)+\mu'}\otimes_\mathbb{k}\mathrm{Sat}(\mathcal{I}^{w_0(\mu)+\delta}_*)(w_0(\mu'))\\
     &\to\mathcal{IC}^{\frac{\infty}{2}}_{\lambda+\mu},
 \end{align*}
 where the second (resp. last) morphism is obtained by projecting onto the $1$-dimensional $(-\delta)$-weight space (resp. $(w_0(\mu)+\delta)$-weight space), identified with $\mathbb{k}$ via the weight vector $\dot{w_0}\cdot v_{-w_0(\delta)}$ (resp. $v_{w_0(\mu)+\delta}$). Thanks to the commutativity of the diagram \eqref{diagram prop hecke}, the collection of morphisms $(H_{\mu,\lambda}^\delta)_{\delta\in \mathbf{Y}^+\cap(-w_0(\mu)+\mathbf{Y}^+)}$ is compatible with the transition morphisms in the limit defining $\mathcal{R}_\mu$. So we obtain the desired morphism $H_{\mu,\lambda}$. Moreover, the commutativity of  \eqref{diagram prop hecke} implies that the morphisms $(H_{\mu,\lambda})_{\mu,\lambda\in\mathbf{Y}}$ satisfy obvious associativity and unit axioms. Namely:
 \begin{prop}\label{prop action IC}
    The morphisms $(H_{\mu,\lambda})_{\mu,\lambda\in\mathbf{Y}}$ endow $\mathcal{IC}^{\frac{\infty}{2}}$ with a structure of left $\mathcal{R}$-module.
\end{prop}
 
	\subsection{Action of convolution on standard and simple objects}
	\begin{prop}{{\cite[Proposition 5.4.2]{ABBGM}}}\label{Prop Action of convolution on standard objects}
		For $\lambda\in\mathbf{Y}^+$ and $\nu\in\mathbf{Y}$, there are canonical isomorphisms
		\begin{equation*}
		    \mathcal{D}_\lambda\star\Delta^{\frac{\infty}{2}}_\nu\simeq\Delta^{\frac{\infty}{2}}_{\nu+\lambda},\quad \mathcal{N}_\lambda\star\nabla^{\frac{\infty}{2}}_\nu\simeq\nabla^{\frac{\infty}{2}}_{\nu+\lambda}.
		\end{equation*}
	\end{prop}
	\begin{proof}
  The second isomorphism will follow from the first one by applying Verdier duality togehter with Corollary \ref{verdier coro}. So we only need to prove that
 $$ \mathcal{D}_\lambda\star\Delta^{\frac{\infty}{2}}_\nu\simeq\Delta^{\frac{\infty}{2}}_{\nu+\lambda}.$$
 
		One can check (same proof as in \cite[Proposition 5.1.5]{ABBGM}) that the perverse cohomologies of $\mathcal{D}_\lambda\star\Delta^{\frac{\infty}{2}}_\nu$ are objects of $\mathrm{Perv}_{I}(\mathcal{F}l^{\frac{\infty}{2}})$. Moreover, we can deduce from the second point of Proposition \ref{simple objects prop} that, for all $w\in W,~\delta\in\mathbf{Y}$, the perverse cohomologies of the restriction of an object of $\mathrm{Perv}_{I}(\mathcal{F}l^{\frac{\infty}{2}})$ to $_{wt_\delta}^{I}\mathrm{Bun}_{N^-}$ is of the form $\mathcal{F}\widetilde{\boxtimes}\mathrm{IC}_{{_\nu\overline{\mathrm{Bun}}}_{N^-}}$, for some $\mathcal{F}\in \mathrm{Perv}_N(G/B^-)$. Thus, to prove that $\mathcal{D}_\lambda\star\Delta^{\frac{\infty}{2}}_\nu\simeq\Delta^{\frac{\infty}{2}}_{\nu+\lambda}$, it suffices to check that the stalk of $\mathcal{D}_\lambda\star\Delta^{\frac{\infty}{2}}_\nu$ at any point of $_{wt_\delta}^{I}\overline{\mathrm{Bun}}_{N^-}$ is zero when $wt_\delta\neq t_{\nu+\lambda}$, and is the (perversly shifted) constant sheaf otherwise. The computation of the latter stalk can then be conducted as in the proof of \cite[Proposition 5.4.2]{ABBGM}, where the authors show that it amounts to computing the cohomology of an empty set when $wt_\delta\neq t_{\nu+\lambda}$, and of a point scheme otherwise.
	\end{proof}

	Next, we exhibit the commutativity of certain diagrams (cf. \cite[(49),(50)]{ABBGM}), which will turn out to be useful in §\ref{section equivalence}. Below, those diagrams will involve the following canonical morphisms.  For any $\mu,\lambda\in\mathbf{Y}^+,~\nu\in\mathbf{Y}$, we have:
 \begin{equation}\label{eq can}
     \nabla_{\lambda}\star\mathcal{IC}^{\frac{\infty}{2}}_\nu\xrightarrow{\sim}\mathcal{N}_\lambda\star\mathrm{IC}^{\frac{\infty}{2}}_\nu\to\mathcal{N}_\lambda\star\nabla^{\frac{\infty}{2}}_\nu,
 \end{equation}
 where the first isomorphism follows from Lemma \ref{lem conv for}. Moreover, we will denote by 
$$b_\lambda:\mathcal{I}^{\lambda}_*\to \nabla_{\lambda}$$
the morphism established in \eqref{eq convo}, and by
$$a_\nu:\mathcal{IC}^{\frac{\infty}{2}}_\nu\to \nabla_{\nu}^{\frac{\infty}{2}} $$
the canonical morphism in $\mathrm{Perv}_{I_\mathrm{u}}(\mathcal{F}l^{\frac{\infty}{2}})$, where we have identified $\mathcal{IC}^{\frac{\infty}{2}}_\nu$ with $\mathrm{IC}^{\frac{\infty}{2}}_\nu$ through the functor $\mathrm{For}_{I_\mathrm{u}}^{G[[t]]}$ (cf. Proposition \ref{prop forget}).
 \begin{prop}\label{prop diagram 1}
     For any $\mu,\lambda\in\mathbf{Y}^+,~\nu\in\mathbf{Y}$, the following diagrams  are commutative:
	\begin{equation}\label{diagram 50}
		\xymatrix{
			 \mathcal{I}^{\lambda}_*\star\mathcal{IC}^{\frac{\infty}{2}}_\nu\ar[d]^{\sim}_{\text{Prop.\ref{prop hecke geometric}}} \ar[r]^{b_\lambda\star\mathrm{id}}&\nabla_{\lambda}\star\mathcal{IC}^{\frac{\infty}{2}}_\nu\ar[r]^{\eqref{eq can}}&\mathcal{N}_\lambda\star\nabla^{\frac{\infty}{2}}_\nu\ar[d]^{\sim}_{\text{Prop. \ref{Prop Action of convolution on standard objects}}}  \\
			\oplus_{\nu'}\mathcal{IC}^{\frac{\infty}{2}}_{\nu+\nu'}\otimes \mathrm{Sat}(\mathcal{I}^{\lambda}_*)(w_0(\nu')) \ar[r]&\mathcal{IC}^{\frac{\infty}{2}}_{\nu+\lambda}\ar[r]_{a_{\nu+\lambda}}&\nabla_{\nu+\lambda}^{\frac{\infty}{2}}}
	\end{equation}
	\begin{equation}\label{diagram 51}
		\xymatrix{
			\nabla_{\lambda}\star\mathcal{IC}^{\frac{\infty}{2}}_\nu\ar[d]^{\eqref{eq can}} \ar[r]&\nabla_{\lambda}\star\mathcal{I}^{\mu}_*\star\mathcal{I}^{-w_0(\mu)}_*\star\mathcal{IC}^{\frac{\infty}{2}}_\nu\ar[d]^{\text{Prop.\ref{prop hecke geometric}}} \\
			\mathcal{N}_{\lambda}\star\nabla_\nu^{\frac{\infty}{2}}\ar[d]^{\sim} & \nabla_{\lambda}\star\mathcal{I}^{\mu}_*\star\mathcal{IC}^{\frac{\infty}{2}}_{\nu-\mu}\ar[d]^{(\mathrm{id}\star b_\mu)\star\mathrm{id}}\\
			\nabla_{\lambda+\nu}^{\frac{\infty}{2}} &\nabla_{\lambda+\mu}\star\mathcal{IC}^{\frac{\infty}{2}}_{\nu-\mu},\ar[l]}
	\end{equation}
 where the bottom arrow $\nabla_{\lambda+\mu}\star\mathcal{IC}^{\frac{\infty}{2}}_{\nu-\mu}\to \nabla_{\lambda+\nu}^{\frac{\infty}{2}}$ is the composition 
 $$\nabla_{\lambda+\mu}\star\mathcal{IC}^{\frac{\infty}{2}}_{\nu-\mu}\xrightarrow{\eqref{eq can}}\mathcal{N}_{\lambda+\mu}\star\nabla_{\nu-\mu}^{\frac{\infty}{2}}\xrightarrow[\text{Prop.}~\ref{Prop Action of convolution on standard objects}]{\sim}  \nabla_{\lambda+\nu}^{\frac{\infty}{2}}. $$
 \end{prop}
 \begin{proof}
     Notice that, for any $\nu\in\mathbf{Y}$, the space of morphisms between $\mathrm{IC}^{\frac{\infty}{2}}_{\nu}$ and $\nabla_{\nu}^{\frac{\infty}{2}}$ is one dimensional and canonical. The fact that the two morphisms morphisms $\mathcal{I}^{\lambda}_*\star\mathcal{IC}^{\frac{\infty}{2}}_\nu\to\nabla_{\nu+\lambda}^{\frac{\infty}{2}}$  (resp. $\nabla_\lambda\star\mathcal{IC}^{\frac{\infty}{2}}_\nu\to\nabla_{\nu+\lambda}^{\frac{\infty}{2}}$ ) considered in the first (resp. second) diagram coincide then follows from the construction of the isosmorphisms of propositions \ref{prop hecke geometric} and \ref{Prop Action of convolution on standard objects}.
 \end{proof}

	From Proposition \ref{Prop Action of convolution on standard objects} we deduce (by replacing $B$ with $B^-$ in the arguments) that we have a canonical isomorphism for $\lambda\in\mathbf{Y}^+,~\nu'\in\mathbf{Y}$
	$$\mathcal {D}_{-\lambda}\star \Delta_{w_0t_{\nu'}}^{\frac{\infty}{2}}\simeq \Delta_{w_0t_{\nu'-w_0(\lambda)}}^{\frac{\infty}{2}}.$$
Thanks to the first point of Lemma \ref{lem conv standard}, we deduce that there is a canonical isomorphism
	\begin{equation}\label{eq standard times costandard}
		\mathcal {N}_{\lambda}\star \Delta_{w_0t_{\nu'}}^{\frac{\infty}{2}}\simeq \Delta_{w_0t_{\nu'+w_0(\lambda)}}^{\frac{\infty}{2}}.
	\end{equation}
At the present moment, we are not able to prove the following result (which follows from \cite[Lemma 5.4.5]{ABBGM} in the case of $\mathbb{Q}_\ell$-sheaves), but hope to come back to this question in future work.
 \begin{conj}\label{conj commut diagram}
     For any $\nu\in\mathbf{Y}$, there exists a non-zero morphism 
		$$\widetilde{a}_\nu:\mathrm{IC}^{\frac{\infty}{2}}_{\nu}\to \Delta_{w_0t_{\nu+2\zeta}}^{\frac{\infty}{2}} $$ such that the following diagrams commute for all  $\nu\in\mathbf{Y},~\lambda,\mu\in\mathbf{Y}^+$:
	\begin{equation}\label{eq 512}
		\xymatrix{
			\mathcal{I}^{\lambda}_*\star\mathcal{IC}^{\frac{\infty}{2}}_\nu\ar[d]^{\sim}_{\text{Prop.\ref{prop hecke geometric}}} \ar[r]^{b_\lambda\star\mathrm{id}}&\nabla_\lambda\star\mathcal{IC}^{\frac{\infty}{2}}_\nu\ar[r]& \mathcal{N}_\lambda\star\Delta_{w_0t_{\nu+2\zeta}}^{\frac{\infty}{2}}\ar[d]^{\sim}_{\text{Prop. \ref{Prop Action of convolution on standard objects}}}  \\
			\oplus_{\nu'}\mathcal{IC}^{\frac{\infty}{2}}_{\nu+\nu'}\otimes \mathrm{Sat}(\mathcal{I}^{\lambda}_*)(w_0(\nu')) \ar[r]&\mathcal{IC}^{\frac{\infty}{2}}_{\nu+w_0(\lambda)}\ar[r]_{\widetilde{a}_{\nu+w_0(\lambda)}}&\Delta_{w_0t_{\nu+2\zeta+w_0(\lambda)}}^{\frac{\infty}{2}}}
	\end{equation}
	\begin{equation}
		\xymatrix{
			\nabla_\lambda\star\mathcal{IC}^{\frac{\infty}{2}}_\nu\ar[d] \ar[r]&\nabla_\lambda\star\mathcal{I}^{\mu}_*\star\mathcal{I}^{-w_0(\mu)}_*\star\mathcal{IC}^{\frac{\infty}{2}}_\nu\ar[d]^{\text{Prop.\ref{prop hecke geometric}}} \\
			\mathcal{N}_\lambda\star\Delta_{w_0t_{\nu+2\zeta}}^{\frac{\infty}{2}}\ar[d]^{\sim}_{\eqref{eq standard times costandard}} & \nabla_\lambda\star\mathcal{I}^{\mu}_*\star\mathcal{IC}^{\frac{\infty}{2}}_{\nu-\mu}\ar[d]\\
			\Delta_{w_0t_{\nu+2\zeta+w_0(\lambda)}}^{\frac{\infty}{2}} &\nabla_{\lambda+\mu}\star\mathcal{IC}^{\frac{\infty}{2}}_{\nu-\mu}.\ar[l]}
	\end{equation}
 \end{conj}
	
	\begin{thm}{{\cite[Theorem 5.3.2]{ABBGM}}}\label{thm conv simples}
		For any $w\in W^{\mathrm{res}}_\mathrm{ext}$ and $\nu\in\mathbf{Y}$, we have
		$$\mathrm{L}_w\star\mathcal{IC}^{\frac{\infty}{2}}_\nu\simeq\mathrm{IC}^{\frac{\infty}{2}}_{wt_{\nu}}. $$
	\end{thm}
 \begin{proof}
     By definition we have 
     $$\mathrm{L}_w\star\mathcal{IC}^{\frac{\infty}{2}}_\nu:=(\overset{\leftarrow}{h'})_*(\mathrm{sw}^*(\mathrm{L}_w)\widetilde{\boxtimes}\mathcal{IC}^{\frac{\infty}{2}}_\nu),$$
     where $\overset{\leftarrow}{h'}:{^{I,0}\mathcal{H}_{G,N^{-}}}\to {^{I}_\infty\overline{\mathrm{Bun}}_{N^-}}$. Denote by $X$ (resp. $Y$) the closed substack of ${^{I,0}\mathcal{H}_{G,N^{-}}}$ (resp. of ${^{I}_\infty\overline{\mathrm{Bun}}_{N^-}}$) defined as the support of $\mathrm{sw}^*(\mathrm{L}_w)\widetilde{\boxtimes}\mathcal{IC}^{\frac{\infty}{2}}_\nu$ (resp. of $\mathrm{IC}^{\frac{\infty}{2}}_{wt_{\nu}}$). Consider the stratifications  
     $$X=\bigsqcup_{(u,\mu)\in W_\mathrm{ext}\times \mathbf{Y}}\mathbb{O}_u\times^{N^-[[t]]}{^I_\mu\mathcal{N}},\quad  Y=\bigsqcup_{vt_\lambda\in W_\mathrm{ext}}{^{I}_{vt_\lambda}\overline{\mathrm{Bun}}_{N^-}},$$
     where $\mathbb{O}_u\subset \mathrm{Fl}$ denotes the $G[[t]]$ -orbit associated with an element $u\in W_\mathrm{ext}$. By the proof of {\cite[Theorem 5.3.2]{ABBGM}}, the morphism 
     $$\overset{\leftarrow}{h'}|_{\overset{\leftarrow}{h'}^{-1}(Y)}:X\to Y$$ is stratified small with respect to the dense open subset ${^{I}_{wt_\nu}\overline{\mathrm{Bun}}_{N^-}}\subset Y$. Since $\overset{\leftarrow}{h'}$ is moreover proper and $\mathrm{sw}^*(\mathrm{L}_w)\widetilde{\boxtimes}\mathcal{IC}^{\frac{\infty}{2}}_\nu$ coincides with the IC-sheaf on $X$, we deduce that $\mathrm{L}_w\star\mathcal{IC}^{\frac{\infty}{2}}_\nu$ coincides with the IC-sheaf on $Y$ (see for instance \cite[Proposition 3.8.10]{achar2021perverse}).
 \end{proof}
	\section{Iwahori-Whittaker objects}

	\subsection{Iwahori-Whittaker for the semi-infinite flag variety}\label{section Iwahori-Whittaker for the semi-infinite flag variety}	For any $\mathbb{F}$-stack locally of finite type $\mathcal{X}$ acted on by $N^-$, the category $\mathrm{D}_{(N^-,\psi)}(\mathcal{X})$ is defined as the full subcategory of $\mathrm{D}(\mathcal{X})$ consisting of objects $\mathcal{F}$ such that there exists an isomorphism $a^*\mathcal{F}\simeq \psi^*\mathcal{L}_{\mathrm{AS}}\boxtimes \mathcal{F}$, where $a:N^-\times \mathcal{X}\to \mathcal{X}$ is the morphism defining the action (and recall that $\psi:N^-\to\mathbb{G}_\mathrm{a}$ was defined in §\ref{section Iwahori-Whittaker model for the Satake equivalence}). This category is endowed with a natural perverse $t$-structure, whose heart will be denoted by $\mathrm{Perv}_{(N^-,\psi)}(\mathcal{X})$, and which is sent to $\mathrm{Perv}(\mathcal{X})$ by the forgetful functor $\mathrm{For}_{(N^-,\psi)}:\mathrm{D}_{(N^-,\psi)}(\mathcal{X})\to \mathrm{D}(\mathcal{X})$.
	
	First consider the category $\mathrm{Perv}_{(N^-,\psi)}(G/B^-)$. It is a classical fact that the only $N^-$-orbit supporting a non-zero Whittaker local system is the open dense orbit, i.e. the orbit associated with $w_0$. Thus, the category $\mathrm{Perv}_{(N^-,\psi)}(G/B^-)$ contains a unique simple object, denoted by $\psi_{G/B^-}$, defined as the Goresky-Macpherson extension of the rank-one Whittaker local system on this orbit. Next, we define the category $\mathrm{Perv}_{(I^-_\mathrm{u},\psi)}(\mathcal{F}l^{\frac{\infty}{2}})$ as the full subcategory of $\mathrm{Perv}_{(N^-,\psi)}({_\infty^1\overline{\mathrm{Bun}}_{N^-}})$ consisting of $T$-equivariant perverse sheaves satisfying the factorization and equivariance conditions (cf. §\ref{section Main definition}). The same arguments as the one used in the proof of Proposition \ref{simple objects prop} show that the simple objects of $\mathrm{Perv}_{(I^-_\mathrm{u},\psi)}(\mathcal{F}l^{\frac{\infty}{2}})$ are of the form
	$$\mathrm{IC}_\nu^\psi:=(\overline{i}_\nu)_{!*}\left(\psi_{G/B^-}\widetilde{\boxtimes}\mathrm{IC}_{{_\nu\overline{\mathrm{Bun}}_{N^-}}}\right),~\nu\in\mathbf{Y}. $$
	We have the convolution functor
	$$\mathrm{Perv}_{\mathcal{IW}}(\mathrm{Gr})\times \mathrm{Perv}_{G[[t]]}(\mathcal{F}l^{\frac{\infty}{2}})\to  \mathrm{Perv}_{(I^-_\mathrm{u},\psi)}(\mathcal{F}l^{\frac{\infty}{2}}).$$
Recall (§\ref{subsection the group}) that we have picked $\zeta\in\mathbf{Y}$ such that $\langle\alpha,\zeta\rangle=1$ for all $\alpha\in\mathfrak{R}_s$. The proof of the next result is the same as that of Theorem \ref{thm conv simples}.
	\begin{prop}{{\cite[Theorem 5.3.5]{ABBGM}}}\label{conv whitt}
		For any $\nu\in\mathbf{Y}$, we have
		$$\mathrm{L}^\mathcal{IW}_{\zeta}\star\mathcal{IC}^{\frac{\infty}{2}}_\nu\simeq  \mathrm{IC}^\psi_{\nu+\zeta}.$$
	\end{prop}
	For any $\nu\in\mathbf{Y}$, we have the corresponding standard and costandard objects
	$$\nabla_\nu^\psi:= (\overline{i}_\nu)_{*}\left(\psi_{G/B^-}\widetilde{\boxtimes}\mathrm{IC}_{{_\nu\overline{\mathrm{Bun}}_{N^-}}}\right),\quad  \Delta_\nu^\psi:= (\overline{i}_\nu)_{!}\left(\psi_{G/B^-}\widetilde{\boxtimes}\mathrm{IC}_{{_\nu\overline{\mathrm{Bun}}_{N^-}}}\right).$$
	By the proof of \cite[Corollary 4.4.5]{ABBGM}, we know that the inclusion of ${_{\nu}^{I_\mathrm{u}}\overline{\mathrm{Bun}}_{N^-}}$ into ${_{\leq\nu}^{I_\mathrm{u}}\overline{\mathrm{Bun}}_{N^-}}$ is affine, so that both objects above are perverse sheaves. They are then easily seen to be objects of $\mathrm{Perv}_{(I^-_\mathrm{u},\psi)}(\mathcal{F}l^{\frac{\infty}{2}})$.
	\begin{thm}{{\cite[Theorem 4.4.10]{ABBGM}}}\label{prop clean IC}
		For any $\nu\in\mathbf{Y}$, the canonical morphisms
		$$\Delta_\nu^\psi\to \mathrm{IC}_\nu^\psi\to\nabla_\nu^\psi  $$
		are isomorphisms.
	\end{thm}
	\begin{proof}
		By the arguments of \textit{loc. cit.}, we are reduced to showing that the Euler characteristic of the stalk of $\mathrm{IC}_{\nu+\zeta}^\psi$ at any point of the locally closed substack $\mathrm{ev}_{\nu'}^{-1}(N^-\cdot w_0)$, for $\nu'<\nu+\zeta$, is zero. By Proposition \ref{conv whitt}, this is equivalent to computing the stalk of $\mathrm{L}^\mathcal{IW}_{\zeta}\star\mathcal{IC}^{\frac{\infty}{2}}_{\nu}$. Using similar arguments as in the proof of Proposition \ref{prop hecke geometric}, this stalk identifies with the extension of a finite number of certain complexes of $\mathbb{k}$-vector spaces $K_{\nu''}$, $\nu''<\nu$. Each $K_{\nu''}$ is defined as the stalk of the direct image under
		$$\overset{\leftarrow}{h'}:(\overset{\leftarrow}{h'})^{-1}({^1_{\nu'}\overline{\mathrm{Bun}}}_{N^-})\bigcap (\overset{\rightarrow}{h'})^{-1}({_{ \nu''}\overline{\mathrm{Bun}}}_{N^-})\bigcap {^{1,0}\overline{\mathcal{H}}^{\zeta}_{G,N^-}}\to {^1_{\nu'}\overline{\mathrm{Bun}}}_{N^-}$$
		of the restriction of $\mathrm{L}^\mathcal{IW}_{\zeta}\widetilde{\boxtimes}\mathcal{IC}^{\frac{\infty}{2}}_{\nu}$ to the above locally closed substack. Moreover, it follows from the constructions that the projection $\overset{\rightarrow}{h'}$ identifies that same locally closed substack with 
		$$\left(N((t))\cdot w_0(\nu'-\nu'')\cap I^-\cdot(w_0(\zeta))\right)\times^{N^-[[t]]}{_{\nu''}{\mathcal{N}}}.$$ 
		Therefore 
		\begin{align*}
		    K_{\nu''}&\simeq H^\bullet_c(N((t))\cdot w_0(\nu'-\nu'')\cap I^-\cdot(w_0(\zeta)),\mathcal{L}\otimes_\mathbb{k} \mathcal{IC}^{\frac{\infty}{2}}_{\nu}|_{{_{ \nu''}\mathrm{Bun}}_{N^-}})\\
      &\simeq H^\bullet_c(N((t))\cdot w_0(\nu'-\nu'')\cap I^-\cdot(w_0(\zeta)),\mathcal{L})\otimes_\mathbb{k} \mathcal{IC}^{\frac{\infty}{2}}_{\nu}|_{{_{ \nu''}\mathrm{Bun}}_{N^-}}
		\end{align*}
		where $\mathcal{L}$ denotes the rank-one Iwahori-Whittaker local system on $I^-\cdot( w_0(\zeta))$, and $\mathcal{IC}^{\frac{\infty}{2}}_{\nu}|_{{_{ \nu''}\mathrm{Bun}}_{N^-}}$ is meant to be the stalk of $\mathcal{IC}^{\frac{\infty}{2}}_{\nu}$ at any point of ${_{ \nu''}\mathrm{Bun}}_{N^-}$. 

  On the other hand, let us put $\mu_r:=(\ell^r-1)\zeta$ for any integer $r\geq0$. The canonical morphisms
  $$\mathcal{I}^{\mu_r}_!\to\mathcal{IC}^{\mu_r} \to\mathcal{I}^{\mu_r}_*$$
  are isomorphisms for any $r$ thanks to Lemma \ref{lem Steinberg weights}, so that (by applying the functor $\Delta^{\mathcal{IW}}_{\zeta}\star(-)$) the canonical morphisms
$$\Delta^{\mathcal{IW}}_{\mu_r+\zeta}\to\mathrm{L}_{\mu_r+\zeta}^\mathcal{IW}\to\nabla^{\mathcal{IW}}_{\mu_r+\zeta}$$
are also isomorphisms. In particular, the IC-sheaf $\mathrm{L}_{\mu_r+\zeta}^\mathcal{IW}\in\mathrm{Perv}_{\mathcal{IW}}(\mathrm{Gr})$ is clean. 

Notice that the map 
$$f:\nu''\mapsto \mu_r-(\nu-\nu'') $$
induces a bijection between the sets $\{\nu'',~\nu''\leq \nu\}$ and $\{\mu'',~\mu''\leq \mu_r\}$. 
Put $\mu'_r:=f(\nu')$. The stalk of $\mathrm{L}_{\mu_r+\zeta}^\mathcal{IW}\simeq \mathrm{L}_{\zeta}^\mathcal{IW}\star \mathcal{IC}^{\mu_r}$ at $[t^{\mu'_r}]\in Y_{\mu'_r}$ can be written as an extension of certain complexes $\widetilde{K}_{\mu''}$, for $\mu''\leq\mu_r$ (only a finite number of complexes will be non-zero). Each $\widetilde{K}_{\mu''}$ is defined as the stalk at $[t^{\mu'_r}]$ of the direct image under the multiplication morphism
\begin{equation}\label{eq multi morphism}
    m:I^-\cdot(w_0(\zeta))\times^{G((t))} \mathrm{Gr}^{\mu''}\to  \mathrm{Gr}
\end{equation}
of the restriction of $\mathrm{L}_{\zeta}^\mathcal{IW}\widetilde{\boxtimes}\mathcal{IC}^{\mu_r}$ to $I^-\cdot(w_0\zeta)\times^{G((t))} \mathrm{Gr}^{\mu''}$. By \cite[Remark 5.3.3]{ABBGM}, the preimage of $[t^{\mu'_r}]$ by the morphism \eqref{eq multi morphism} is isomorphic to $N((t))\cdot w_0(\mu'_r-\mu'')\cap I^-\cdot(w_0(\zeta))$ when $r$ is large enough (so that $\mu_r$ is dominant enough). Thus
\begin{align*}
		    \widetilde{K}_{\mu''}&\simeq H^\bullet_c(N((t))\cdot w_0(\mu'_r-\mu'')\cap I^-\cdot(w_0(\zeta)),\mathcal{L}\otimes_\mathbb{k} \mathcal{IC}^{\mu_r}|_{\mathrm{Gr}^{\mu''}})\\
      &\simeq H^\bullet_c(N((t))\cdot w_0(\mu'_r-\mu'')\cap I^-\cdot(w_0(\zeta)),\mathcal{L})\otimes_\mathbb{k} \mathcal{IC}^{\mu_r}|_{\mathrm{Gr}^{\mu''}}
		\end{align*}
 where $\mathcal{IC}^{\mu_r}|_{\mathrm{Gr}^{\mu''}}$ is meant to be the stalk of $\mathcal{IC}^{\mu_r}$ at any point of $\mathrm{Gr}^{\mu''}$. 

Since $\mu_r-f(\nu'')=\nu-\nu''$ by construction, we can take $r$ large enough so that 
$$\mathcal{IC}^{\frac{\infty}{2}}_{\nu}|_{{_{ \nu''}\mathrm{Bun}}_{N^-}} \simeq \mathcal{IC}^{\mu_r}|_{\mathrm{Gr}^{f(\nu'')}}$$
for all $\nu''<\nu$ such that $K_{\nu''}$ is non-zero by Corollary \ref{coro kostant}. Moreover $\mu'_r-f(\nu'')=\nu'-\nu''$ by definition of $f$. So we get an isomorphism
$$K_{\nu''}\simeq \widetilde{K}_{f(\nu'')}\quad\forall \nu''\leq\nu. $$

So the Euler characteristic of the stalk of $\mathrm{IC}_{\nu+\zeta}^\psi$ at any point of $\mathrm{ev}_{\nu'}^{-1}(N^-\cdot w_0)$ coincides with the Euler characteristic of the stalk of $\mathrm{L}_{\mu_r+\zeta}^\mathcal{IW}$ at $[t^{\mu'_r}]$. But $\mu'_r<\mu_r+\zeta$, so that the latter stalk is zero by cleanness. This concludes the proof.
	\end{proof}
	\subsection{Partial integrability and consequences}
	We denote by
	\begin{equation}\label{av 1}
		\mathrm{av}_{!,(N^-,\psi)},~\mathrm{av}_{*,(N^-,\psi)}:\mathrm{D}({_{ \infty}^1\overline{\mathrm{Bun}}}_{N^-})\to \mathrm{D}_{(I^-_\mathrm{u},\psi)}({_{ \infty}^1\overline{\mathrm{Bun}}}_{N^-}) 
	\end{equation}
	the functors which are respectively left and right adjoint to the forgetful functor. They are defined by
 $$\mathrm{av}_{!,(N^-,\psi)}(\mathcal{F}):=a_!(\psi^*\mathcal{L}_{\mathrm{AS}}\boxtimes\mathcal{F})[2\mathrm{dim}(N^-)],\quad  \mathrm{av}_{*,(N^-,\psi)}(\mathcal{F}):=a_*(\psi^*\mathcal{L}_{\mathrm{AS}}\boxtimes\mathcal{F}),$$
 where $a:N^-\times {_{ \infty}^1\overline{\mathrm{Bun}}}_{N^-}\to {_{ \infty}^1\overline{\mathrm{Bun}}}_{N^-}$ is the action morphism.
 Moreover we denote by
	$$\mathrm{Av}_{!,(I_\mathrm{u},\psi)},~\mathrm{Av}_{*,(I_\mathrm{u},\psi)}:\mathrm{D}({_{ \infty}^{I_\mathrm{u}}\overline{\mathrm{Bun}}}_{N^-})\to \mathrm{D}_{(I^-_\mathrm{u},\psi)}({_{ \infty}^1\overline{\mathrm{Bun}}}_{N^-}) $$
	the functors obtained by composing \eqref{av 1} with the forgetful functor $\mathrm{D}({_{ \infty}^{I_\mathrm{u}}\overline{\mathrm{Bun}}}_{N^-})\to \mathrm{D}({_{ \infty}^1\overline{\mathrm{Bun}}}_{N^-}) $. 

 Let us denote by 
 $$\mathrm{For}^I_{I_\mathrm{u}}:\mathrm{D}({_{ \infty}^{I}\overline{\mathrm{Bun}}}_{N^-})\to \mathrm{D}({_{ \infty}^{I_\mathrm{u}}\overline{\mathrm{Bun}}}_{N^-})$$
 the forgetful functor. From the definitions, we have isomorphisms of functors from $\mathrm{D}({_{ \infty}^{I}\overline{\mathrm{Bun}}}_{N^-})$ to $ \mathrm{D}_{(I^-_\mathrm{u},\psi)}({_{ \infty}^1\overline{\mathrm{Bun}}}_{N^-})$:
  \begin{equation}\label{eq averaging}
      \mathrm{Av}_{!,(I_\mathrm{u},\psi)}[-2\mathrm{dim}(N^-)]\circ\mathrm{For}^I_{I_\mathrm{u}}\simeq \Delta_\psi\star(-),\quad \mathrm{Av}_{*,(I_\mathrm{u},\psi)}\circ\mathrm{For}^I_{I_\mathrm{u}}\simeq \nabla_\psi\star(-) 
  \end{equation}
     where $\Delta_\psi\in \mathrm{Perv}_{(N^-,\psi)}(G/B)$ (resp. $\nabla_\psi$ ) is the $!$-extension (resp. $*$-extension) of the rank one Whittaker local system on the dense $N^-$-orbit in $G/B$  (and the convolution product $\star$ is defined as in Remark \ref{rem conv flag}).
	\begin{lem}{{\cite[Lemma 4.4.12]{ABBGM}}}\label{lem averaging}
		The canonical morphism of functors
		$$\mathrm{Av}_{!,(I_\mathrm{u},\psi)}[-\mathrm{dim}(N^-)]\to \mathrm{Av}_{*,(I_\mathrm{u},\psi)}[\mathrm{dim}(N^-)]$$
		is an isomorphism. Moreover, the resulting functor
		$$\mathrm{D}({_{ \infty}^{I_\mathrm{u}}\overline{\mathrm{Bun}}}_{N^-})\to  \mathrm{D}_{(I^-_\mathrm{u},\psi)}({_{ \infty}^1\overline{\mathrm{Bun}}}_{N^-})$$
		is $t$-exact.
	\end{lem}
 We will denote by 
	$$\mathrm{Av}_{(I_\mathrm{u},\psi)}:\mathrm{D}({_{ \infty}^{I_\mathrm{u}}\overline{\mathrm{Bun}}}_{N^-})\to  \mathrm{D}_{(I^-_\mathrm{u},\psi)}({_{ \infty}^1\overline{\mathrm{Bun}}}_{N^-}) $$
	the obtained $t$-exact functor.
 \begin{proof}
     First, notice that the category $\mathrm{D}({_{ \infty}^{I_\mathrm{u}}\overline{\mathrm{Bun}}}_{N^-})$ is generated (as a triangulated category) by the essential image of the forgetful functor 
     $$\mathrm{For}^I_{I_\mathrm{u}}:\mathrm{D}({_{ \infty}^{I}\overline{\mathrm{Bun}}}_{N^-})\to \mathrm{D}({_{ \infty}^{I_\mathrm{u}}\overline{\mathrm{Bun}}}_{N^-}).$$
   Thus, the first point can be proved for the corresponding functors composed with $\mathrm{For}^I_{I_\mathrm{u}}$. In this case, we have the isomorphisms of functors \eqref{eq averaging}, so that the morphism 
     $$\mathrm{Av}_{!,(I_\mathrm{u},\psi)}[-\mathrm{dim}(N^-)]\circ\mathrm{For}^I_{I_\mathrm{u}}\to \mathrm{Av}_{*,(I_\mathrm{u},\psi)}[\mathrm{dim}(N^-)]\circ\mathrm{For}^I_{I_\mathrm{u}} $$
     is induced by the canonical morphism $\Delta_\psi\to\nabla_\psi$. But the latter morphism is an isomorphism, as both objects are easily seen to be isomorphic to the simple perverse sheaf $\psi_{G/B}$.

     Finally, we deduce that the resulting functor is $t$-exact. Indeed, $\mathrm{av}_{*,(N^-,\psi)}$ is right $t$-exact since it is defined as the $*$-pushforward along an affine morphism \cite[Th\'eor\`eme 4.1.1]{beilinson2018faisceaux}, and $\mathrm{av}_{!,(N^-,\psi)}$ is left $t$-exact since it is defined as the $!$-pushforward along an affine morphism \cite[Corollaire 4.1.2]{beilinson2018faisceaux}.
 \end{proof}

	An object of $\mathrm{Perv}_{I_\mathrm{u}}(\mathcal{F}l^{\frac{\infty}{2}})$ is called \textit{partially integrable} if all of its irreducible sub-quotients are of the form $\mathrm{IC}^{\frac{\infty}{2}}_{wt_\nu}$, with $\nu\in\mathbf{Y}$ and $w\neq w_0$. We let $^f\mathrm{Perv}_{I_\mathrm{u}}(\mathcal{F}l^{\frac{\infty}{2}})$ denote the resulting Serre quotient category obtained by quotienting $\mathrm{Perv}_{I_\mathrm{u}}(\mathcal{F}l^{\frac{\infty}{2}})$ by the sub-category of partially integrable objects. The next result parallels Proposition \ref{Prop averaging simples}.
 \begin{prop}\label{prop averaging simples 2}
      For any $wt_\nu\in W_\mathrm{ext}$ we have isomorphisms
   $$\mathrm{Av}_{I_\mathrm{u},\psi}(\nabla^{\frac{\infty}{2}}_{wt_\nu})\simeq\nabla^{\psi}_{\nu},\quad \mathrm{Av}_{I_\mathrm{u},\psi}(\Delta^{\frac{\infty}{2}}_{wt_\nu})\simeq\Delta^{\psi}_{\nu}.$$
 \end{prop}
 \begin{proof}
 By \eqref{eq averaging}, we have 
 $$\mathrm{Av}_{I_\mathrm{u},\psi}(\nabla^{\frac{\infty}{2}}_{wt_\nu})\simeq \nabla_\psi\star \nabla^{\frac{\infty}{2}}_{wt_\nu},$$
 where $\nabla^{\frac{\infty}{2}}_{wt_\nu}$ is seen as an object of $\mathrm{D}({_{ \infty}^{I}\overline{\mathrm{Bun}}}_{N^-})$. Recall that, by definition 
 $$\nabla^{\frac{\infty}{2}}_{wt_\nu}=(\overline{i}_\nu)_{*}\left(\tilde{\nabla}_w\widetilde{\boxtimes}\mathrm{IC}_{{_\nu\overline{\mathrm{Bun}}}_{N^-}}\right).$$
 So one can check that 
 $$\nabla_\psi\star \nabla^{\frac{\infty}{2}}_{wt_\nu}\simeq (\overline{i}_\nu)_{*}\left((\nabla_\psi\star\tilde{\nabla}_w)\widetilde{\boxtimes}\mathrm{IC}_{{_\nu\overline{\mathrm{Bun}}}_{N^-}}\right). $$
 As in \cite[Lemma 4.4.8]{MR3003920}, we have an isomorphism  $\nabla_\psi\star\tilde{\nabla}_w\simeq \psi_{G/B}$, concluding the proof of the first isomorphism. The proof of the second one is similar.
 \end{proof}
 The proof of the next result is done exactly as in \cite[Proposition 4.4.13]{ABBGM}.
	\begin{prop}\label{prop functor average}
		\begin{enumerate}
			\item The functor 
			$$\mathrm{Av}_{(I_\mathrm{u},\psi)}:\mathrm{Perv}_{I_\mathrm{u}}(\mathcal{F}l^{\frac{\infty}{2}})\to  \mathrm{Perv}_{(I^-_\mathrm{u},\psi)}(\mathcal{F}l^{\frac{\infty}{2}})  $$
			factors through $^f\mathrm{Perv}_{I_\mathrm{u}}(\mathcal{F}l^{\frac{\infty}{2}})$.
			\item The resulting functor 
			$$^f\mathrm{Perv}_{I_\mathrm{u}}(\mathcal{F}l^{\frac{\infty}{2}})\to  \mathrm{Perv}_{(I^-_\mathrm{u},\psi)}(\mathcal{F}l^{\frac{\infty}{2}}) $$
			is exact and faithful.
		\end{enumerate}
	\end{prop}

	\section{The convolution functor}\label{section equivalence}
	\subsection{Construction of the functor}\label{Construction of the functor and first properties}
	Here we follow \cite[§6.1.1]{ABBGM}. We recall the construction of a functor 
	$$\mathrm{Conv}^{\mathrm{Hecke}}:\mathrm{Mod}^\mathbf{Y}_{I_\mathrm{u}}(\mathcal{R})\to\mathrm{Ind}(\mathrm{Perv}_{{I_\mathrm{u}}}(\mathcal{F}l^{\frac{\infty}{2}})), $$
	which will (conjecturally) restrict to an equivalence on the categories of finite length objects. Let $\mathcal{S}\in \mathrm{Mod}^\mathbf{Y}_{I_\mathrm{u}}(\mathcal{R})$. For any $\lambda,\mu\in\mathbf{Y}$, consider the morphisms 
	\begin{align*}
	    &\alpha_{\lambda,\mu}:\mathcal{S}_\lambda\star\mathcal{R}_{-\mu}\star\mathcal{IC}^{\frac{\infty}{2}}_{\mu-\lambda}\to  \mathcal{S}_{\lambda-\mu}\star\mathcal{IC}^{\frac{\infty}{2}}_{\mu-\lambda}\\
     &\beta_{\lambda,\mu}:\mathcal{S}_\lambda\star\mathcal{R}_{-\mu}\star\mathcal{IC}^{\frac{\infty}{2}}_{\mu-\lambda}\to  \mathcal{S}_{\lambda}\star\mathcal{IC}^{\frac{\infty}{2}}_{-\lambda},
	\end{align*}
	induced respectively by the right $\mathcal{R}$-module structure on $\mathcal{S}$ and the left $\mathcal{R}$-module structure on $\mathcal{IC}^{\frac{\infty}{2}}$ (cf. §\ref{section Action of the regular perverse sheaf}). 
 
 We then define $\mathrm{Conv}^{\mathrm{Hecke}}(\mathcal{S})$ as the coequalizer of the following diagram:
	$$  \bigoplus_{\lambda,\mu}\mathcal{S}_\lambda\star\mathcal{R}_{-\mu}\star\mathcal{IC}^{\frac{\infty}{2}}_{\mu-\lambda}\rightrightarrows\bigoplus_\nu \mathcal{S}_{\nu}\star\mathcal{IC}^{\frac{\infty}{2}}_{-\nu},$$
	where the two arrows correspond to $\alpha_{\lambda,\mu}$ and $\beta_{\lambda,\mu}$. 
	
	In other words, the object $\mathrm{Conv}^{\mathrm{Hecke}}(\mathcal{S})$ represents the covariant functor which, to any $\mathcal{F}\in \mathrm{Ind}(\mathrm{Perv}_{{I_\mathrm{u}}}(\mathcal{F}l^{\frac{\infty}{2}}))$ , associates the set of collections of morphisms $(\mathcal{S}_{\nu}\star\mathcal{IC}^{\frac{\infty}{2}}_{-\nu}\to \mathcal{F})_{\nu\in\mathbf{Y}}$ such that, for any $\lambda,\mu\in\mathbf{Y}$, the following diagram commutes:
	\begin{equation}\label{diagram coeq}
		\mathcal{S}_\lambda\star\mathcal{R}_{-\mu}\star\mathcal{IC}^{\frac{\infty}{2}}_{\mu-\lambda}\rightrightarrows\mathcal{F}, 
	\end{equation}
	where the top (resp. bottom) arrow is the composition of $\alpha_{\lambda,\mu}$ (resp. $\beta_{\lambda,\mu}$) with the morphism $ \mathcal{S}_{\lambda-\mu}\star\mathcal{IC}^{\frac{\infty}{2}}_{\mu-\lambda}\to \mathcal{F}$ (resp. $ \mathcal{S}_{\lambda}\star\mathcal{IC}^{\frac{\infty}{2}}_{-\lambda}\to \mathcal{F}$).
\subsection{First properties}\label{section first properties}
 Recall that $\Phi:\mathrm{Perv}_{I_\mathrm{u}}(\mathrm{Gr})\to \mathrm{mod}^\mathbf{Y}_{I_\mathrm{u}}(\mathcal{R})$ was defined in §\ref{section Modules over the regular perverse sheaf}.

	\begin{prop}{{\cite[Proposition 3.1.2]{ABBGM}}}\label{prop can isom}
		For any $\mathcal{S}\in\mathrm{Perv}_{I_\mathrm{u}}(\mathrm{Gr})$ and $\mu\in\mathbf{Y}$, we have a canonical isomorphism
		$$\mathrm{Conv}^{\mathrm{Hecke}}(\Phi(\mathcal{S})\langle\mu\rangle)\simeq \mathcal{S}\star\mathcal{IC}^{\frac{\infty}{2}}_{-\mu}. $$
	\end{prop}
	\begin{proof}
		We first construct a morphism of functors 
		\begin{equation}\label{functor hecke}
			\mathrm{Hom}_{\mathrm{Ind}(\mathrm{Perv}_{{I_\mathrm{u}}}(\mathcal{F}l^{\frac{\infty}{2}}))}\left(\mathrm{Conv}^{\mathrm{Hecke}}(\Phi(\mathcal{S})\langle\mu\rangle),-\right)\to  \mathrm{Hom}_{\mathrm{Ind}(\mathrm{Perv}_{{I_\mathrm{u}}}(\mathcal{F}l^{\frac{\infty}{2}}))}\left(\mathcal{S}\star\mathcal{IC}^{\frac{\infty}{2}}_{-\mu},-\right).
		\end{equation}
		For any morphism $\mathrm{Conv}^{\mathrm{Hecke}}(\Phi(\mathcal{S})\langle\mu\rangle)\to \mathcal{F}$, we take the component 
		$$(\Phi(\mathcal{S})\langle\mu\rangle)_{\mu}\star \mathcal{IC}^{\frac{\infty}{2}}_{-\mu}\to \mathcal{F}.$$
		Composing with the canonical morphism $\mathcal{S}\to\mathcal{S}\star\mathcal{R}_0=(\Phi(\mathcal{S})\langle\mu\rangle)_{\mu}$, we obtain the desired map $\mathcal{S}\star \mathcal{IC}^{\frac{\infty}{2}}_{-\mu}\to \mathcal{F}$.
		
		Next, we construct an inverse of $\eqref{functor hecke}$. Consider a morphism $\varphi:\mathcal{S}\star \mathcal{IC}^{\frac{\infty}{2}}_{-\mu}\to \mathcal{F}$. For any $\lambda\in\mathbf{Y}$, we construct a morphism $(\Phi(\mathcal{S})\langle\mu\rangle)_{\lambda}\star \mathcal{IC}^{\frac{\infty}{2}}_{-\lambda}\to \mathcal{F}$ in the following way:
		$$(\Phi(\mathcal{S})\langle\mu\rangle)_{\lambda}\star \mathcal{IC}^{\frac{\infty}{2}}_{-\lambda}=\mathcal{S}\star\mathcal{R}_{\lambda-\mu}\star \mathcal{IC}^{\frac{\infty}{2}}_{-\lambda}\xrightarrow{\mathrm{id}\star H_{\lambda-\mu,-\lambda}}\mathcal{S}\star\mathcal{IC}^{\frac{\infty}{2}}_{-\mu}\xrightarrow{\varphi} \mathcal{F}. $$
		
		The fact that this collection of morphisms makes the diagram \eqref{diagram coeq} commute follows from the associativity property of the action of $\mathcal{R}$ on $\mathcal{IC}^{\frac{\infty}{2}}$ (Proposition \ref{prop action IC}). The fact that we indeed obtain an inverse of \eqref{functor hecke} is readily checked.
	\end{proof}
	\begin{coro}\label{coro simples}
		For any $w\in W_\mathrm{ext}$, we have an isomorphism
		$$\mathrm{Conv}^{\mathrm{Hecke}}(\widehat{\mathcal{L}}_w)\simeq\mathrm{IC}^{\frac{\infty}{2}}_w. $$
	\end{coro}
	\begin{proof}
		Let $w\in W_\mathrm{ext}$, and write it as $w=yt_\lambda$ for $(y,\lambda)\in W^{\mathrm{res}}_\mathrm{ext}\times\mathbf{Y}$. By construction, $\widehat{\mathcal{L}}_w=\Phi(\mathrm{L}_y)\langle-\lambda\rangle$, so that Proposition \ref{prop can isom} yields an isomorphism
		$$\mathrm{Conv}^{\mathrm{Hecke}}(\widehat{\mathcal{L}}_w)\simeq \mathrm{L}_y\star \mathcal{IC}^{\frac{\infty}{2}}_{\lambda}. $$
		By Theorem \ref{thm conv simples}, the right-hand side is isomorphic to $\mathrm{IC}^{\frac{\infty}{2}}_{yt_{\lambda}}$.
	\end{proof}
 	\begin{prop}{{\cite[Proposition 6.2.2]{ABBGM}}}\label{lem exact}
		The functor $\mathrm{Conv}^{\mathrm{Hecke}}$ is exact.
	\end{prop}
\begin{proof}
	By construction, the functor $\mathrm{Conv}^{\mathrm{Hecke}}$ is right exact. So it remains to check that it is left exact. Since $\mathrm{Mod}^\mathbf{Y}_{I_\mathrm{u}}(\mathcal{R})$ is the ind-completion of $\mathrm{mod}^\mathbf{Y}_{I_\mathrm{u}}(\mathcal{R})$, it suffices to check left exactness for the restriction of $\mathrm{Conv}^{\mathrm{Hecke}}$ to $\mathrm{mod}^\mathbf{Y}_{I_\mathrm{u}}(\mathcal{R})$. So let 
	$$0\to\mathcal{S}_1\xrightarrow{u}\mathcal{S}_2\to\mathcal{S}\to 0 $$
	be an exact sequence in  $\mathrm{mod}^\mathbf{Y}_{I_\mathrm{u}}(\mathcal{R})$. We have to prove that $\mathrm{Conv}^{\mathrm{Hecke}}(\mathcal{S}_1)\to \mathrm{Conv}^{\mathrm{Hecke}}(\mathcal{S}_2)$ is injective. By an inductive argument on the length of $\mathcal{S}_2$, it is sufficient to prove the result in the case where $\mathcal{S}$ is simple. In particular (thanks to Theorem \ref{thm simple g1t}), we may and will assume that there exist $\mathcal{F}\in\mathrm{Perv}_{I_\mathrm{u}}(\mathrm{Gr})$ and $\mu\in\mathbf{Y}$ such that $\mathcal{S}\simeq\Phi(\mathcal{F})\langle\mu\rangle$. Consider the morphism
 $$\mathcal{F}\to \mathcal{F}\star\mathcal{R}_0\simeq \mathcal{S}_{-\mu}.  $$
 Since $(\mathcal{S}_2)_{-\mu},~\mathcal{S}_{-\mu}$ are ind-objects of the category $\mathrm{Perv}_{I_\mathrm{u}}(\mathrm{Gr})$, there exists $\mathcal{F}'\in \mathrm{Perv}_{I_\mathrm{u}}(\mathrm{Gr})$ together with a morphism $\varphi:\mathcal{F}'\to (\mathcal{S}_2)_{-\mu}$ and a surjective morphism $\mathcal{F}'\to\mathcal{F}$ such that the following diagram commutes
 \begin{equation*}
			\xymatrix{
				\mathcal{F}'\ar[d] \ar[r]&\mathcal{F}\ar[d]  \\
				(\mathcal{S}_2)_{-\mu}\ar[r]&\mathcal{S}_{-\mu}.}
		\end{equation*}
  In particular, we obtain a morphism $a:\Phi(\mathcal{F}')\langle\mu\rangle\to \mathcal{S}_2$ by Proposition \ref{prop iso morphisms}. We put $\mathcal{F}'':=\mathrm{ker}(\mathcal{F}'\to\mathcal{F})$. We have a commutative diagram
  \begin{equation}
		\xymatrix{
   &&0\ar[d]&0\ar[d]& \\
			&&\Phi(\mathcal{F}'')\langle\mu\rangle\ar[d] \ar[r]^{\mathrm{id}}&\Phi(\mathcal{F}'')\langle\mu\rangle\ar[d]& \\
			0\ar[r]&\mathcal{S}_1\ar[d]^{\mathrm{id}}\ar[r]&\mathcal{S}_1\oplus(\Phi(\mathcal{F}')\langle\mu\rangle)\ar[r]\ar[d]^{u\oplus a}&\Phi(\mathcal{F}')\langle\mu\rangle\ar[d]\ar[r] &0\\ 0\ar[r]&\mathcal{S}_1\ar[r]&\mathcal{S}_2\ar[r]\ar[d]&\Phi(\mathcal{F})\langle\mu\rangle\ar[d]\ar[r] &0\\ 
			&&0&0.&}
	\end{equation}
 To prove the desired result, we only need to prove that the maps $ \mathrm{Conv}^{\mathrm{Hecke}}(\mathcal{S}_1)\to \mathrm{Conv}^{\mathrm{Hecke}}(\mathcal{S}_1\oplus(\Phi(\mathcal{F}')\langle\mu\rangle))$ and $ \mathrm{Conv}^{\mathrm{Hecke}}(\Phi(\mathcal{F}'')\langle\mu\rangle)\to \mathrm{Conv}^{\mathrm{Hecke}}(\mathcal{S}_1\oplus(\Phi(\mathcal{F}')\langle\mu\rangle))$ are injective. For the first one this is obvious, since  $\mathrm{Conv}^{\mathrm{Hecke}}$ transforms direct sums into direct sums. For the second one, it suffices to check that 
 $$ \mathrm{Conv}^{\mathrm{Hecke}}(\Phi(\mathcal{F}'')\langle\mu\rangle)\to \mathrm{Conv}^{\mathrm{Hecke}}(\Phi(\mathcal{F}')\langle\mu\rangle)$$ 
 is injective. But one deduces from Proposition \ref{prop can isom} that the latter morphism coincides with 
 $$\mathcal{F}''\star\mathcal{IC}^{\frac{\infty}{2}}_{\mu}\to\mathcal{F}'\star\mathcal{IC}^{\frac{\infty}{2}}_{\mu}, $$
 obtained by convolving the injection $\mathcal{F}''\to\mathcal{F}'$ with $\mathcal{IC}^{\frac{\infty}{2}}_{\mu}$. This morphism is injective thanks to exactness of convolution (Theorem \ref{thm exact}).
\end{proof}
 Notice that, thanks to Proposition \ref{lem exact} and Corollary \ref{coro simples}, we immediately obtain:
\begin{prop}
    The functor $\mathrm{Conv}^{\mathrm{Hecke}}$ sends finite length objects to finite length objects. 
\end{prop}
So restriction yields a functor
$$\mathrm{Conv}^{\mathrm{Hecke}}:\mathrm{mod}^\mathbf{Y}_{I_\mathrm{u}}(\mathcal{R})\to\mathrm{Perv}_{{I_\mathrm{u}}}(\mathcal{F}l^{\frac{\infty}{2}}). $$
Notice that, by replacing $\mathrm{Mod}^\mathbf{Y}_{I_\mathrm{u}}(\mathcal{R})$ with $\mathrm{Mod}^\mathbf{Y}_{\mathcal{IW}}(\mathcal{R})$ in the definition of $\mathrm{Conv}^{\mathrm{Hecke}}$, we obtain a functor 
$$\mathrm{Conv}^{\mathrm{Hecke}}_\mathcal{IW}:\mathrm{Mod}^\mathbf{Y}_{\mathcal{IW}}(\mathcal{R})\to\mathrm{Ind}(\mathrm{Perv}_{(I^-_\mathrm{u},\psi)}(\mathcal{F}l^{\frac{\infty}{2}})) $$
satisfying properties analogous to the previously stated results concerning $\mathrm{Conv}^{\mathrm{Hecke}}$. In particular, $\mathrm{Conv}^{\mathrm{Hecke}}_\mathcal{IW}$ sends simple objects to simple objects.
 \begin{prop}\label{prop commutes ave}
     The functor $\mathrm{Conv}^{\mathrm{Hecke}}$ commutes with averaging functors. Namely, for any $\mathcal{S}\in \mathrm{mod}^\mathbf{Y}_{I_\mathrm{u}}(\mathcal{R})$ we have a functorial isomorphism
     $$\mathrm{Av}_{(I_\mathrm{u},\psi)}(\mathrm{Conv}^{\mathrm{Hecke}}(\mathcal{S}))\simeq \mathrm{Conv}^{\mathrm{Hecke}}_\mathcal{IW}(\mathrm{Av}_{(I_\mathrm{u},\psi)}(\mathcal{S})).$$
 \end{prop}
 \begin{proof}
 First, one can easily check from the constructions of the averaging functors that, for any $\lambda\in\mathbf{Y}$, we have an isomorphism functorial in $\mathcal{S}\in\mathrm{Mod}^\mathbf{Y}_{I_\mathrm{u}}(\mathcal{R})$:
 $$\mathrm{Av}_{(I_\mathrm{u},\psi)}(\mathcal{S}\star\mathcal{IC}^\frac{\infty}{2}_\lambda)\simeq \mathrm{Av}_{(I_\mathrm{u},\psi)}(\mathcal{S})\star\mathcal{IC}^\frac{\infty}{2}_\lambda. $$
 The conclusion then follows from the fact that $\mathrm{Av}_{(I_\mathrm{u},\psi)}$ is an exact functor, so it preserves coequalizers.
 \end{proof}
 \begin{lem}\label{lem verdier commute}
	The functor $\mathrm{Conv}^{\mathrm{Hecke}}$ commutes with Verdier duality.
\end{lem}
\begin{proof}
    Let $\mathcal{S}\in \mathrm{mod}^\mathbf{Y}_{I_\mathrm{u}}(\mathcal{R})$. Recall the construction of the Verdier duality on $ \mathrm{mod}^\mathbf{Y}_{I_\mathrm{u}}(\mathcal{R})$ from subsection \ref{section Verdier}, and let $\alpha:\mathcal{S}_1\to \mathcal{S}_2$ be a morphism between two free graded $\mathcal{R}$-modules such that $\mathcal{S}$ is the cokernel of $\alpha$. By Remark \ref{rem duality}, $\mathbb{D}(\mathcal{S})$ coincides with the kernel of $\mathbb{D}(\alpha):\mathbb{D}(\mathcal{S}_2)\to \mathbb{D}(\mathcal{S}_1)$. Let us spell out how $\mathbb{D}(\alpha)$ was constructed. By additivity of $\mathbb{D}$, we may and will assume that $\mathcal{S}_1\simeq\Phi(\mathcal{F}_1)$ and $\mathcal{S}_2\simeq\Phi(\mathcal{F}_2)\langle\mu\rangle$, for some $\mathcal{F}_1,\mathcal{F}_2\in \mathrm{Perv}_{I_\mathrm{u}}(\mathrm{Gr})$, and $\mu\in\mathbf{Y}$. So (using Proposition \ref{prop iso morphisms}) $\alpha$ comes from a map
    $$\tilde{\alpha}:\mathcal{F}_1\to \mathcal{F}_2\star\mathcal{I}^{-w_0(\mu)+\nu}_*\star\mathcal{I}^{-w_0(\nu)}_*$$
    in $\mathrm{Perv}_{I_\mathrm{u}}(\mathrm{Gr})$, for some large enough $\nu\in \mathbf{Y}^+\cap(w_0(\mu)+\mathbf{Y}^+)$. Thus (up to taking a larger $\nu$), $\mathbb{D}(\alpha)$ comes from the dual morphism
    $$\beta:\mathbb{D}(\mathcal{F}_2)\to \mathbb{D}(\mathcal{F}_1)\star\mathcal{I}^{\mu-w_0(\nu)}_*\star\mathcal{I}^{\nu}_*.$$
    
Now, the object $\mathrm{Conv}^{\mathrm{Hecke}}(\mathbb{D}(\mathcal{S}))$ is the kernel of $\mathrm{Conv}^{\mathrm{Hecke}}(\mathbb{D}(\alpha))$ thanks to Proposition \ref{lem exact}.
    By Proposition \ref{prop can isom}, the latter morphism is a morphism between the objects $\mathbb{D}(\mathcal{F}_2)\star\mathcal{IC}^{\frac{\infty}{2}}_{-\mu}$ and $\mathbb{D}(\mathcal{F}_1)\star\mathcal{IC}^\frac{\infty}{2}_0$. It is constructed as the composition
    \begin{align*}
        \mathbb{D}(\mathcal{F}_2)\star\mathcal{IC}^{\frac{\infty}{2}}_{-\mu}&\xrightarrow{\beta\star\mathrm{id}}\mathbb{D}(\mathcal{F}_1)\star(\mathcal{I}^{\mu-w_0(\nu)}_*\star\mathcal{I}^{\nu}_*\star\mathcal{IC}^{\frac{\infty}{2}}_{-\mu})\\
        &\xrightarrow{\mathrm{id}\star H'}\mathbb{D}(\mathcal{F}_1)\star\mathcal{IC}^\frac{\infty}{2}_0, 
    \end{align*}
    where $H'$ is the image of $H^\nu_{-\mu,0}$ through the composition
    \begin{align*}
        \mathrm{Hom}((\mathcal{I}^{-w_0(\mu)+\nu}_*\star\mathcal{I}^{-w_0(\nu)}_*)\star\mathcal{IC}^\frac{\infty}{2}_0,\mathcal{IC}^{\frac{\infty}{2}}_{-\mu})&\xrightarrow[\sim]{\text{adjunction}}\mathrm{Hom}(\mathcal{IC}^\frac{\infty}{2}_0,(\mathcal{I}^{\mu-w_0(\nu)}_!\star\mathcal{I}^{\nu}_!)\star\mathcal{IC}^{\frac{\infty}{2}}_{-\mu})\\
        &\xrightarrow[\sim]{\mathbb{D}}\mathrm{Hom}((\mathcal{I}^{\mu-w_0(\nu)}_*\star\mathcal{I}^{\nu}_*)\star\mathcal{IC}^{\frac{\infty}{2}}_{-\mu},\mathcal{IC}^\frac{\infty}{2}_0).
    \end{align*}

    On the other hand, $\mathrm{Conv}^{\mathrm{Hecke}}(\mathcal{S})$ is the cokernel of $\mathrm{Conv}^{\mathrm{Hecke}}(\alpha)$ thanks to Proposition \ref{lem exact}. The latter morphism coincides with the composition
    \begin{align*}
        \mathcal{F}_1\star\mathcal{IC}^\frac{\infty}{2}_0&\xrightarrow{\tilde{\alpha}\star\mathrm{id}}\mathcal{F}_2\star(\mathcal{I}^{-w_0(\mu)+\nu}_*\star\mathcal{I}^{-w_0(\nu)}_*\star\mathcal{IC}^\frac{\infty}{2}_0)\\
        &\xrightarrow{\mathrm{id}\star H_{-\mu,0}^{\nu}} \mathcal{F}_2\star\mathcal{IC}^{\frac{\infty}{2}}_{-\mu}. 
    \end{align*}

    To conclude, we just need to notice that $\mathrm{Conv}^{\mathrm{Hecke}}(\mathbb{D}(\alpha))$ and $\mathrm{Conv}^{\mathrm{Hecke}}(\alpha)$ are transformed into one another by Verdier duality.
\end{proof}	

\begin{conj}\label{main conjecture}
		The functor 
		$$\mathrm{Conv}^{\mathrm{Hecke}}:\mathrm{mod}^\mathbf{Y}_{I_\mathrm{u}}(\mathcal{R})\to\mathrm{Perv}_{{I_\mathrm{u}}}(\mathcal{F}l^{\frac{\infty}{2}}) $$
		induces an equivalence of categories between $\mathrm{mod}^\mathbf{Y}_{I_\mathrm{u}}(\mathcal{R})$ and the subcategory of finite-length objects in $\mathrm{Perv}_{{I_\mathrm{u}}}(\mathcal{F}l^{\frac{\infty}{2}})$.
	\end{conj}
\section{Strategy of proof for Conjecture \ref{main conjecture}}\label{section Strategy of proof for the conjecture}
We give a strategy of proof for Conjecture \ref{main conjecture}, which follows \cite{ABBGM} and will work provided Conjecture \ref{conj commut diagram} holds.
\subsection{First steps}
The strategy is to prove that $\mathrm{Conv}^{\mathrm{Hecke}}$ respects the highest weight structures on both sides. 

\begin{lem}\label{lem commute convo}
    Let $\gamma\in\mathbf{Y}^+$, $\nu\in\mathbf{Y}$ and  $w,w'\in W$ such that $\ell(ww_0w')=\ell(w)+\ell(w_0w')$. We have canonical isomorphisms
    $$ \mathcal{N}_\gamma\star^I\mathrm{Conv}^{\mathrm{Hecke}}(\widehat{\mathcal{Z}}'_\nu)\simeq \mathrm{Conv}^{\mathrm{Hecke}}(\widehat{\mathcal{Z}}'_{\gamma+\nu}),\quad\mathcal{N}_w\star^I\mathrm{Conv}^{\mathrm{Hecke}}(\widehat{\mathcal{Z}}'_{w't_\nu})\simeq \mathrm{Conv}^{\mathrm{Hecke}}(\widehat{\mathcal{Z}}'_{ww't_\nu}).$$
\end{lem}
\begin{proof}
   Recall that  $\mathrm{Conv}^{\mathrm{Hecke}}(\widehat{\mathcal{Z}}'_{\gamma+\nu})$ represents the covariant functor which, to any $\mathcal{F}\in \mathrm{Ind}(\mathrm{Perv}_{{I_\mathrm{u}}}(\mathcal{F}l^{\frac{\infty}{2}}))$ , associates the set of collections of morphisms $$((\widehat{\mathcal{Z}}'_{\gamma+\nu})_\lambda\star\mathcal{IC}^{\frac{\infty}{2}}_{-\lambda}\to \mathcal{F})_{\nu\in\mathbf{Y}}$$ such that, for any $\lambda,\mu\in\mathbf{Y}$, the following diagram commutes (cf. \eqref{diagram coeq}):
	\begin{equation*}
		(\widehat{\mathcal{Z}}'_{\gamma+\nu})_\lambda\star\mathcal{R}_{-\mu}\star\mathcal{IC}^{\frac{\infty}{2}}_{\mu-\lambda}\rightrightarrows\mathcal{F}. 
	\end{equation*} 
 Applying Proposition \ref{prop convolution costandard} together with the first point of Lemma \ref{lem conv standard}, we obtain that $\mathrm{Conv}^{\mathrm{Hecke}}(\widehat{\mathcal{Z}}'_{\gamma+\nu})$ represents the covariant functor which, to any $\mathcal{F}\in \mathrm{Ind}(\mathrm{Perv}_{{I_\mathrm{u}}}(\mathcal{F}l^{\frac{\infty}{2}}))$ , associates the set of collections of morphisms $$((\widehat{\mathcal{Z}}'_{\nu})_\lambda\star\mathcal{IC}^{\frac{\infty}{2}}_{-\lambda}\to \mathcal{D}_{-\gamma}\star\mathcal{F})_{\nu\in\mathbf{Y}}$$ such that, for any $\lambda,\mu\in\mathbf{Y}$, the following diagram commutes (cf. \eqref{diagram coeq}):
	\begin{equation*}
		(\widehat{\mathcal{Z}}'_{\nu})_\lambda\star\mathcal{R}_{-\mu}\star\mathcal{IC}^{\frac{\infty}{2}}_{\mu-\lambda}\rightrightarrows\mathcal{D}_{-\gamma}\star\mathcal{F}. 
	\end{equation*} 
 By the first point of Lemma \ref{lem conv standard} again, this functor is represented by $ \mathcal{N}_\gamma\star^I\mathrm{Conv}^{\mathrm{Hecke}}(\widehat{\mathcal{Z}}'_\nu)$.

  The second isomorphism is proved in the same way, using that $\mathcal{N}_w\star^I\widehat{\mathcal{Z}}'_{w't_\nu}\simeq \widehat{\mathcal{Z}}'_{ww't_\nu}$ (Proposition \ref{prop convolution costandard 2}). 
\end{proof}
  We have the following key result.
\begin{prop}\label{proposition costandrad 0}
   Assume that Conjecture \ref{conj commut diagram} holds. There is an isomorphism
    \begin{equation}\label{eq costandrad 0}
        \mathrm{Conv}^{\mathrm{Hecke}}(\widehat{\mathcal{Z}}'_{0})\simeq \nabla_0^\frac{\infty}{2}.
    \end{equation}
\end{prop}

\subsection{Proof of Proposition \ref{proposition costandrad 0}} Recall (cf. Proposition \ref{prop head}) that, for any $w\in W_\mathrm{ext}$, the object $\widehat{\mathcal{Z}}_{w}$ has a simple head, isomorphic to $\widehat{\mathcal{L}}_{w}$. We start with a lemma (inspired by {\cite[Corollary 4.4.14]{ABBGM}}).
\begin{lem}\label{lem indecomposable} Let $\nu\in\mathbf{Y}$.
		\begin{enumerate}
			\item The kernel of 
   $$f:\mathrm{Conv}^{\mathrm{Hecke}}(\widehat{\mathcal{Z}}_{w_0t_\nu}\to\widehat{\mathcal{L}}_{w_0t_\nu})$$ 
   is partially integrable.
			\item For any $w\in W$, $\mathrm{IC}_{w_0t_\nu}^{\frac{\infty}{2}}$ is the only non-partially integrable irreducible sub-quotient of $\mathrm{Conv}^{\mathrm{Hecke}}(\widehat{\mathcal{Z}}_{wt_\nu})$.
			\item The object $\mathrm{IC}_{w_0t_\nu}^{\frac{\infty}{2}}$ is the head of $\mathrm{Conv}^{\mathrm{Hecke}}(\widehat{\mathcal{Z}}'_{\nu})$ and the socle of $\mathrm{Conv}^{\mathrm{Hecke}}(\widehat{\mathcal{Z}}_{\nu})$. In particular, $\mathrm{Conv}^{\mathrm{Hecke}}(\widehat{\mathcal{Z}}'_{\nu})$ and $\mathrm{Conv}^{\mathrm{Hecke}}(\widehat{\mathcal{Z}}_{\nu})$ are indecomposable.
		\end{enumerate}
	\end{lem}
 \begin{proof}
 \begin{enumerate}
			\item By exactness of the functor $\mathrm{Av}_{(I_\mathrm{u},\psi)}$ (second point of Proposition \ref{prop functor average}), it suffices to show that the image of $f$ through $\mathrm{Av}_{(I_\mathrm{u},\psi)}$ is an isomorphism. But we have
   $$\mathrm{Av}_{(I_\mathrm{u},\psi)}(f)\simeq \mathrm{Conv}^{\mathrm{Hecke}}_\mathcal{IW}(\mathrm{Av}_{(I_\mathrm{u},\psi)}(\widehat{\mathcal{Z}}_{w_0t_\nu}\to\widehat{\mathcal{L}}_{w_0t_\nu})) $$
   thanks to Proposition \ref{prop commutes ave}, and the morphism $\mathrm{Av}_{(I_\mathrm{u},\psi)}(\widehat{\mathcal{Z}}_{w_0t_\nu}\to\widehat{\mathcal{L}}_{w_0t_\nu})$ is easily seen to be an isomorphism thanks to Proposition \ref{prop semi-simple IW}.
			\item Let $w'\in W$ be such that $w=w'w_0$. We have a canonical morphism 
   $\mathcal{D}_0\to \mathcal{D}_{w'}$, whose image through $\mathrm{Av}_{(I_\mathrm{u},\psi)}$ is an isomorphism. Convolving it with $\mathrm{id}:\mathrm{Conv}^{\mathrm{Hecke}}(\widehat{\mathcal{Z}}_{w_0t_\nu})\to \mathrm{Conv}^{\mathrm{Hecke}}(\widehat{\mathcal{Z}}_{w_0t_\nu})$ we obtain a morphism
   \begin{align*}
       \mathrm{Conv}^{\mathrm{Hecke}}(\widehat{\mathcal{Z}}_{w_0t_\nu})\to \mathcal{D}_{w'}\star\mathrm{Conv}^{\mathrm{Hecke}}(\widehat{\mathcal{Z}}_{w_0t_\nu})&\overset{\text{Lem. \ref{lem commute convo}}}{\simeq} \mathrm{Conv}^{\mathrm{Hecke}}(\mathcal{D}_{w'}\star\widehat{\mathcal{Z}}_{w_0t_\nu})\\
       &\overset{\text{Prop. \ref{prop convolution costandard 2}}}{\simeq}\mathrm{Conv}^{\mathrm{Hecke}}(\widehat{\mathcal{Z}}_{wt_\nu}).
   \end{align*}
   The image of the above map through $\mathrm{Av}_{(I_\mathrm{u},\psi)}$ is an isomorphism. So we conclude thanks to the first point and the second point of Proposition \ref{prop functor average}.
			\item Since Verdier duality swaps $\mathrm{Conv}^{\mathrm{Hecke}}(\widehat{\mathcal{Z}}'_{\nu})$ and $\mathrm{Conv}^{\mathrm{Hecke}}(\widehat{\mathcal{Z}}_{\nu})$ (Lemma \ref{lem verdier commute}) and preserves $\mathrm{IC}_{w_0t_\nu}^{\frac{\infty}{2}}$, it suffices to show that $\mathrm{IC}_{w_0t_\nu}^{\frac{\infty}{2}}$ is the head of $\mathrm{Conv}^{\mathrm{Hecke}}(\widehat{\mathcal{Z}}'_{\nu})$. In view of the second point, it suffices to show that, if $\mathrm{Conv}^{\mathrm{Hecke}}(\widehat{\mathcal{Z}}'_{\nu})$ maps to an irreducible object $\mathcal{S}$, then $\mathcal{S}$ must be non-partially integrable. Here we repeat the argument of \cite[Proposition 2.3.2]{ABBGM}. So assume that $\mathrm{Conv}^{\mathrm{Hecke}}(\widehat{\mathcal{Z}}'_{\nu})$ maps to a partially integrable irreducible object $\mathcal{S}\in \mathrm{Perv}_{I_\mathrm{u}}(\mathcal{F}l^{\frac{\infty}{2}})$. Since $\mathcal{S}$ is partially integrable and is an IC-sheaf, it must be equivariant with respect to a minimal parabolic subgroup $P_s\subset G$ associated with a simple reflection $s$. Then $\mathcal{N}_s\star \mathcal{S}$ lives in perverse cohomological degree $1$. However, we have an isomorphism $$\mathcal{N}_s\star\mathrm{Conv}^{\mathrm{Hecke}}(\widehat{\mathcal{Z}}'_{\nu})\simeq \mathrm{Conv}^{\mathrm{Hecke}}(\widehat{\mathcal{Z}}'_{st_\nu})$$ 
   by Proposition \ref{prop convolution costandard 2}, so that the above complex is perverse. Therefore 
   $$\mathrm{Hom}(\mathcal{N}_s\star\mathrm{Conv}^{\mathrm{Hecke}}(\widehat{\mathcal{Z}}'_{\nu}),\mathcal{N}_s\star \mathcal{S})=0,$$ which is a contradiction since the functor $\mathcal{N}_s\star(-)$ is an equivalence (with inverse $\mathcal{D}_s\star(-)$ by the first point of Lemma \ref{lem conv standard}).
		\end{enumerate}\end{proof}

We now construct a morphism 
\begin{equation}\label{eq morphism constandard}
  \varphi: \mathrm{Conv}^{\mathrm{Hecke}}(\widehat{\mathcal{Z}}'_0)\to\nabla^{\frac{\infty}{2}}_0. 
\end{equation}
Let $\mu\in\mathbf{Y}$. Recall (§\ref{section Baby Verma and co-Verma modules}) that by  $(\widehat{\mathcal{Z}}'_0)_\mu=\varinjlim_{\lambda\in\mathbf{Y}^+-\mu} \nabla_{\mu+\lambda}\star \mathcal{I}^{-w_0(\lambda)}_*.$ For any $\lambda\in\mathbf{Y}^+-\mu$ we define a morphism
\begin{align*}
    \nabla_{\mu+\lambda}\star \mathcal{I}^{-w_0(\lambda)}_*\star\mathcal{IC}^{\frac{\infty}{2}}_{-\mu}&\xrightarrow[\sim]{\text{Prop. \ref{prop hecke geometric}}}\bigoplus_{\nu}\nabla_{\mu+\lambda}\star\mathcal{IC}^{\frac{\infty}{2}}_{\nu-\mu}\otimes\mathrm{Sat}(\mathcal{I}^{-w_0(\lambda)}_*)(w_0(\nu)) \\
    &\to \nabla_{\mu+\lambda}\star\mathcal{IC}^{\frac{\infty}{2}}_{-\mu-\lambda}\xrightarrow{\eqref{eq can}} \mathcal{N}_{\mu+\lambda}\star\nabla^{\frac{\infty}{2}}_{-\mu-\lambda}\\
    &\xrightarrow[\sim]{\text{Prop. \ref{Prop Action of convolution on standard objects}}}\nabla^{\frac{\infty}{2}}_{0}.
\end{align*}
Thanks to the commutativity of the diagrams \eqref{diagram 50} and \eqref{diagram 51}, the above morphisms are compatible with the transition morphisms in the limit defining $(\widehat{\mathcal{Z}}'_0)_\mu$. We therefore obtain a collection of morphisms $((\widehat{\mathcal{Z}}'_0)_\mu\star\mathcal{IC}^{\frac{\infty}{2}}_{-\mu}\to \nabla^{\frac{\infty}{2}}_{0})_{\mu\in\mathbf{Y}}$. The fact that this collection of morphisms factors through $\mathrm{Conv}^{\mathrm{Hecke}}(\widehat{\mathcal{Z}}'_0)$ follows from the construction of the structures of right $\mathcal{R}$-module on $\widehat{\mathcal{Z}}'_0$ (§\ref{section Baby Verma and co-Verma modules}) and left $\mathcal{R}$-module on $\mathrm{IC}^{\frac{\infty}{2}}$ (§\ref{section Action of the regular perverse sheaf}). So we have defined a map $\varphi:\mathrm{Conv}^{\mathrm{Hecke}}(\widehat{\mathcal{Z}}'_0)\to\nabla^{\frac{\infty}{2}}_0$.
\begin{prop}\label{prop varphi iso}
    The morphism $\mathrm{Av}_{(I_\mathrm{u},\psi)}(\varphi)$ is an isomorphism.
\end{prop}
\begin{proof}
    We have an isomorphism $\mathrm{Av}_{(I_\mathrm{u},\psi)}(\mathrm{Conv}^{\mathrm{Hecke}}(\widehat{\mathcal{Z}}'_0))\simeq \mathrm{Conv}^{\mathrm{Hecke}}(\widehat{\mathcal{Z}}'^{\mathcal{IW}}_0)$ thanks to propositions \ref{prop commutes ave} and \ref{Prop averaging simples}. But $\widehat{\mathcal{Z}}'^{\mathcal{IW}}_0\simeq\widehat{\mathcal{L}}^{\mathcal{IW}}_0$, so that we have an isomorphism  
    $$\mathrm{Av}_{(I_\mathrm{u},\psi)}(\mathrm{Conv}^{\mathrm{Hecke}}(\widehat{\mathcal{Z}}'_0))\simeq \mathrm{IC}_0^\psi$$ by Corollary \ref{coro simples}. On the other hand, we also have an isomorphism 
    $$\mathrm{Av}_{(I_\mathrm{u},\psi)}( \nabla^{\frac{\infty}{2}}_{0})\overset{\text{Prop. \ref{prop averaging simples 2}}}{\simeq}\nabla^{\psi}_{0}\overset{\text{Prop. \ref{prop clean IC}}}{\simeq} \mathrm{IC}_0^\psi.$$ Therefore it suffices to check that  $\mathrm{Av}_{(I_\mathrm{u},\psi)}(\varphi)$ is nonzero to conclude. For that, it suffices to check that the image of
    $$\nabla_{\mu+\lambda}\star\mathcal{IC}_{-\lambda-\mu}^\frac{\infty}{2}\to \mathcal{N}_{\mu+\lambda}\star\nabla^{\frac{\infty}{2}}_{-\mu-\lambda}\simeq \nabla^{\frac{\infty}{2}}_{0}$$
    through $\mathrm{Av}_{(I_\mathrm{u},\psi)}$ is nonzero for every $\mu\in\mathbf{Y},~\lambda\in\mathbf{Y}^{++}-\mu$, which brings back to checking that 
    $$\mathrm{Av}_{(I_\mathrm{u},\psi)}(\mathcal{N}_{\mu+\lambda}\star\mathrm{IC}_{-\lambda-\mu}^\frac{\infty}{2}\to \mathcal{N}_{\mu+\lambda}\star\nabla^{\frac{\infty}{2}}_{-\mu-\lambda})$$
    is nonzero. But the latter is isomorphic to 
    $$\mathrm{Av}_{(I_\mathrm{u},\psi)}(\mathcal{N}_{\mu+\lambda})\star\mathrm{IC}_{-\lambda-\mu}^\frac{\infty}{2}\to \mathrm{Av}_{(I_\mathrm{u},\psi)}(\mathcal{N}_{\mu+\lambda})\star\nabla^{\frac{\infty}{2}}_{-\mu-\lambda},$$
    where $\mathrm{Av}_{(I_\mathrm{u},\psi)}:\mathrm{Perv}_{I_\mathrm{u}}(\mathrm{Fl})\to \mathrm{Perv}_{\mathcal{IW}}(\mathrm{Fl})$ is the averaging functor. The above morphism is obviously nonzero by construction of the convolution product.
\end{proof}
Next, we construct a morphism
$$\gamma:\nabla^{\frac{\infty}{2}}_0\to\mathrm{Conv}^{\mathrm{Hecke}}(\widehat{\mathcal{Z}}'_0),$$
assuming that Conjecture \ref{conj commut diagram} holds. Equivalently, we construct a morphism
$$\mathbb{D}(\gamma):\mathrm{Conv}^{\mathrm{Hecke}}(\widehat{\mathcal{Z}}_0)\to\Delta^{\frac{\infty}{2}}_0. $$
By Corollary \ref{coro ind verma}, we need to construct (for all $\mu\in\mathbf{Y}$) the last arrow below
\begin{equation}\label{eq last morphism}
    (\widehat{\mathcal{Z}}_0)_\mu\star\mathcal{IC}^{\frac{\infty}{2}}_{-\mu}\simeq ({^{w_0}(\widehat{\mathcal{Z}}'_{w_0t_{-2\zeta}})})_\mu\star\mathcal{IC}^{\frac{\infty}{2}}_{-\mu}\to \Delta^{\frac{\infty}{2}}_0.
\end{equation} 
One can easily deduce from Proposition \ref{prop convolution costandard} that we have an isomorphism 
$$({^{w_0}(\widehat{\mathcal{Z}}'_{w_0t_{-2\zeta}})})_\mu\simeq \mathcal{N}_{w_0}\star^I({^{w_0}(\widehat{\mathcal{Z}}'_{t_{-2\zeta}})})_\mu.$$
Thus, by the first point of Lemma \ref{lem conv standard}, constructing the morphism \eqref{eq last morphism} is equivalent to constructing a morphism between $({^{w_0}(\widehat{\mathcal{Z}}'_{t_{-2\zeta}})})_\mu\star\mathcal{IC}^{\frac{\infty}{2}}_{-\mu}$ and $\mathcal{D}_{w_0}\star^I\Delta^{\frac{\infty}{2}}_0\simeq \Delta^{\frac{\infty}{2}}_{w_0}$. Now, recall that 
$$({^{w_0}(\widehat{\mathcal{Z}}'_{t_{-2\zeta}})})_\mu:=  \varinjlim_{\lambda} \nabla_{w_0(\mu)+\lambda-2\zeta}\star \mathcal{I}^{-w_0(\lambda)}_*$$
and let us write $\delta:=w_0(\mu)+\lambda-2\zeta$ for concision. We define a morphism
\begin{align*}
   \nabla_{\delta}\star \mathcal{I}^{-w_0(\lambda)}_*\star\mathcal{IC}^{\frac{\infty}{2}}_{-\mu}&\xrightarrow[\sim]{\text{Prop. \ref{prop hecke geometric}}}\bigoplus_{\nu}\nabla_{\delta}\star\mathcal{IC}^{\frac{\infty}{2}}_{\nu-\mu}\otimes\mathrm{Sat}(\mathcal{I}^{-w_0(\lambda)}_*)(w_0(\nu)) \\
    &\to \nabla_{\delta}\star\mathcal{IC}^{\frac{\infty}{2}}_{-w_0(\lambda)-\mu}\xrightarrow{\text{Conj. \ref{conj commut diagram}}} \mathcal{N}_{\delta}\star\Delta^{\frac{\infty}{2}}_{w_0t_{-w_0(\lambda)-\mu+2\zeta}}\\
    &\xrightarrow[\sim]{\text{\eqref{eq standard times costandard}}}\Delta^{\frac{\infty}{2}}_{w_0}.
\end{align*}
Thanks to the commutativity of the diagrams \eqref{eq 513} and \eqref{eq 512}, the above morphisms are compatible with the transition morphisms in the limit defining $(\widehat{\mathcal{Z}}_0)_\mu$. We therefore obtain a collection of morphisms $((\widehat{\mathcal{Z}}_0)_\mu\star\mathcal{IC}^{\frac{\infty}{2}}_{-\mu}\to \Delta^{\frac{\infty}{2}}_{0})_{\mu\in\mathbf{Y}}$. Once again, the fact that this collection of morphisms factors through $\mathrm{Conv}^{\mathrm{Hecke}}(\widehat{\mathcal{Z}}_0)$ follows follows from the construction of the structures of right $\mathcal{R}$-module on $\widehat{\mathcal{Z}}'_0$ (§\ref{section Baby Verma and co-Verma modules}) and left $\mathcal{R}$-module on $\mathcal{IC}^{\frac{\infty}{2}}$ (§\ref{section Action of the regular perverse sheaf}).

\begin{prop}
Assume that Conjecture \ref{conj commut diagram} holds. The morphism $\varphi$ is an isomorphism.
\end{prop}
\begin{proof}
    First, notice that the same arguments as in the proof of Proposition \ref{prop varphi iso} show that the morphism $\mathrm{Av}_{(I_\mathrm{u},\psi)}(\gamma)$ is an isomorphism.

    Thus, the morphism $\mathrm{Av}_{(I_\mathrm{u},\psi)}(\varphi\circ\gamma)\simeq \mathrm{Av}_{(I_\mathrm{u},\psi)}(\gamma)\circ\mathrm{Av}_{(I_\mathrm{u},\psi)}(\varphi)$ is an isomorphism thanks to the previous claim and Proposition \ref{prop varphi iso}. In particular, the composition
    $$\nabla^{\frac{\infty}{2}}_0\xrightarrow{\gamma}\mathrm{Conv}^{\mathrm{Hecke}}(\widehat{\mathcal{Z}}'_0)\xrightarrow{\varphi}\nabla^{\frac{\infty}{2}}_0$$
    is nonzero. Therefore $\varphi\circ\gamma$ is the identity up to a nonzero scalar.
    
   On the other hand, the morphism $\mathrm{Av}_{(I_\mathrm{u},\psi)}(\gamma\circ\varphi)\simeq \mathrm{Av}_{(I_\mathrm{u},\psi)}(\varphi)\circ\mathrm{Av}_{(I_\mathrm{u},\psi)}(\gamma)$ is a nonzero endomorphism of 
   $$\mathrm{Av}_{(I_\mathrm{u},\psi)}(\mathrm{Conv}^{\mathrm{Hecke}}(\widehat{\mathcal{Z}}'_0))\simeq\mathrm{Conv}^{\mathrm{Hecke}}_\mathcal{IW}(\mathrm{Av}_{(I_\mathrm{u},\psi)}(\widehat{\mathcal{Z}}'_0))\simeq\mathrm{Conv}^{\mathrm{Hecke}}_\mathcal{IW}(\widehat{\mathcal{L}}^\mathcal{IW}_0)\simeq\mathrm{IC}^\psi_0,$$ 
   so there exists $\alpha\in\mathbb{k}-\{0\}$ such that $\mathrm{Av}_{(I_\mathrm{u},\psi)}(\alpha\cdot\gamma\circ\varphi)=\mathrm{id}$. Assume that $\gamma\circ\varphi$ is not invertible. Thus $\alpha\cdot\gamma\circ\varphi$ is not invertible. Since $\mathrm{Conv}^{\mathrm{Hecke}}(\widehat{\mathcal{Z}}'_0)$ is indecomposable by the third point of Lemma \ref{lem indecomposable}, its endomorphism ring is local. Therefore $\mathrm{id}-\alpha\cdot\gamma\circ\varphi$ must be invertible. But
   $$\mathrm{Av}_{(I_\mathrm{u},\psi)}(\mathrm{id}-\alpha\cdot\gamma\circ\varphi)=\mathrm{id}- \mathrm{Av}_{(I_\mathrm{u},\psi)}(\alpha\cdot\gamma\circ\varphi)=0,$$
   which is a contradiction. So  $\gamma\circ\varphi$ is invertible.

   From the two previous paragraphs we deduce that $\varphi$ has both a left and a right inverse, so that it is invertible.
\end{proof}
\subsection{Proof of Conjecture \ref{main conjecture} assuming Conjecture \ref{conj commut diagram}}\label{section proof of conjecture} Until the end, we assume that  Conjecture \ref{conj commut diagram} holds.  We can then finally conclude.
\begin{coro}
    For all $w\in W_\mathrm{est}$, we have isomorphisms
\begin{equation*}
    \mathrm{Conv}^{\mathrm{Hecke}}(\widehat{\mathcal{Z}}'_w)\simeq \nabla^{\frac{\infty}{2}}_{w},\quad \mathrm{Conv}^{\mathrm{Hecke}}(\widehat{\mathcal{Z}}_w)\simeq \Delta^{\frac{\infty}{2}}_{w}.
\end{equation*}
\end{coro}

\begin{proof}
    By Lemma \ref{lem verdier commute} and the fact that Verdier duality sends standard objects to costandard objects, it is enough to construct isomorphisms
\begin{equation*}\label{eq costandard}
    \mathrm{Conv}^{\mathrm{Hecke}}(\widehat{\mathcal{Z}}'_w)\simeq \nabla^{\frac{\infty}{2}}_{w}
\end{equation*}
for all $w\in W_{\mathrm{ext}}$. First, if $w=t_\lambda$ for some $\lambda\in \mathbf{Y}^+$, then the isomorphism is obtained from $\varphi$ by convolution with $\mathcal{N}_\lambda$ thanks to Lemma \ref{lem commute convo} and Proposition \ref{Prop Action of convolution on standard objects}.

Next, let $w:=xt_\lambda\in W_\mathrm{ext}$ be an arbitrary element, and pick $\mu\in\mathbf{Y}^+$ dominant enough so that $\lambda+\mu\in\mathbf{Y}^+$. Then we have
$$\mathcal{N}_\mu\star \mathrm{Conv}^{\mathrm{Hecke}}(\widehat{\mathcal{Z}}'_\lambda)\xrightarrow[\text{Lem. \ref{lem commute convo}}]{\sim}\mathrm{Conv}^{\mathrm{Hecke}}(\widehat{\mathcal{Z}}'_{\lambda+\mu})\simeq  \nabla^{\frac{\infty}{2}}_{\lambda+\mu},$$
where the last isomorphism follows from the first paragraph. Thanks to the first point of Lemma \ref{lem conv standard}, we can apply $\mathcal{D}_{-\mu}\star ()$ to obtain an isomorphism 
$$ \mathrm{Conv}^{\mathrm{Hecke}}(\widehat{\mathcal{Z}}'_\lambda)\simeq \mathcal{D}_{-\mu}\star \nabla^{\frac{\infty}{2}}_{\lambda+\mu}.$$
But thanks to Proposition \ref{Prop Action of convolution on standard objects} (and the first point of Lemma \ref{lem conv standard}), we know that 
$$\mathcal{D}_{-\mu}\star \nabla^{\frac{\infty}{2}}_{\lambda+\mu}\simeq  \nabla^{\frac{\infty}{2}}_{\lambda},$$
so that we have an isomorphism $\mathrm{Conv}^{\mathrm{Hecke}}(\widehat{\mathcal{Z}}'_\lambda)\simeq  \nabla^{\frac{\infty}{2}}_{\lambda}$.

Finally, notice that we have isomorphisms $\mathcal{N}_{w_0}\star\mathrm{Conv}^{\mathrm{Hecke}}(\widehat{\mathcal{Z}}'_{w_0t_\lambda})\simeq \mathrm{Conv}^{\mathrm{Hecke}}(\widehat{\mathcal{Z}}'_\lambda)$ (Proposition \ref{prop convolution costandard 2}) and $\mathcal{N}_{w_0}\star \nabla^{\frac{\infty}{2}}_{w_0t_\lambda}\simeq \nabla^{\frac{\infty}{2}}_{\lambda}$ (equation \eqref{add length}). Thanks to the first point of Lemma \ref{lem conv standard}, we deduce that there exists an isomorphism $\mathrm{Conv}^{\mathrm{Hecke}}(\widehat{\mathcal{Z}}'_{w_0t_\lambda})\simeq  \nabla^{\frac{\infty}{2}}_{w_0t_\lambda}$. Applying the functor $\mathcal{N}_{xw_0}\star()$ to the latter isomorphism, we obtain the desired isomorphism thanks to Proposition \ref{prop convolution costandard 2} and equation \eqref{add length}.
\end{proof}
We can conclude the proof of Theorem \ref{main conjecture}. Combining the results of the two previous subsections, we obtain that the functor $\mathrm{Conv}^{\mathrm{Hecke}}$ is exact and satisfies 
$$\mathrm{Conv}^{\mathrm{Hecke}}(\widehat{\mathcal{Z}}_w)\simeq \Delta^{\frac{\infty}{2}}_{w},~\mathrm{Conv}^{\mathrm{Hecke}}(\widehat{\mathcal{Z}}'_w)\simeq \nabla^{\frac{\infty}{2}}_{w},~\mathrm{Conv}^{\mathrm{Hecke}}(\widehat{\mathcal{L}}_w)\simeq \mathrm{IC}^{\frac{\infty}{2}}_{w}$$
for all $w\in W_{\mathrm{ext}}$. Since both categories $\mathrm{mod}^\mathbf{Y}_{I_\mathrm{u}}(\mathcal{R})$ and $\mathrm{Perv}_{{I_\mathrm{u}}}(\mathcal{F}l^{\frac{\infty}{2}})$ are Artinian and satisfy the Ext vanishing property between standard and costandard objects (cf. Proposition \ref{prop ext} and Proposition \ref{prop ext ind objects}), we can apply \cite[Lemma 6.2.7]{ABBGM} to conclude that $\mathrm{Conv}^{\mathrm{Hecke}}$ is an equivalence of categories.

	\bibliographystyle{alphaurl}
	\bibliography{Semi_infi}
\end{document}